\documentclass[10pt]{article}
\usepackage{amsmath}
\usepackage{amssymb}
\usepackage{graphicx}
\usepackage{color}
\usepackage{amsthm}
\usepackage{amsfonts}
\usepackage{amscd}
\usepackage{mathrsfs}
\usepackage{bm}
\usepackage{enumerate}
\usepackage{algorithmic, algorithm}

\usepackage[utf8x]{inputenc}
\usepackage{graphicx}
\usepackage{float}
\usepackage{latexsym}
\usepackage{amsmath}
\usepackage{amssymb}
\usepackage{amsthm}
\usepackage{amsfonts}
\usepackage{color}

\newcommand{\B}{\mathbf}
\newcommand{\bx}{\mathbf{x}}
\newcommand{\diver}{\mbox{div}}
\newcommand{\bn}{\mathbf{n}}
\newcommand{\cM}{\mathcal{M}}
\newcommand{\by}{\mathbf{y}}

\newcommand{\p}{\partial}
\newcommand{\al}{\alpha}

\newcommand{\hk}{R_t(\bx, \by)}
\newcommand{\rhk}{\bar{R}_t(\bx, \by)}
\newcommand{\rrhk}{\bar{\bar{R}}_t(\bx, \by)}

\newcommand{\rhkpipj}{\bar{R}_t({\bf p}_i, {\bf p}_j)}

\newcommand{\hkxpj}{R_t(\bx, {\bf p}_j)}

\newcommand{\rhkxpj}{\bar{R}_t(\bx, {\bf p}_j)}

\newcommand{\bff}{{\bf f}}

\newcommand{\bfu}{{\bf u}}
\newcommand{\bfp}{{\bf p}}
\newcommand{\bfs}{{\bf s}}
\newcommand{\bfn}{{\bf n}}
\newcommand{\bz}{\mathbf{z}}

\newcommand{\bq}{\mathbf{q}}

\newcommand{\M}{{\mathcal M}}

\newcommand{\bV}{\mathbf{V}}
\newcommand{\bA}{\mathbf{A}}
\newcommand{\invt}{\frac{1}{t}}

\newcommand{\mathd}{\mathrm{d}}

\newtheorem{theorem}{\textbf{Theorem}}[section]
\newtheorem{lemma}{\textbf{Lemma}}[section]
\newtheorem{remark}{\textbf{Remark}}[section]
\newtheorem{proposition}{\textbf{Proposition}}[section]
\newtheorem{assumption}{\textbf{Assumption}}[section]
\newtheorem{definition}{\textbf{Definition}}[section]

\newcommand{\R}{\mathbb{R}}

\numberwithin{equation}{section}

\makeatother

\begin{document}

\title{Convergence of the Point Integral method for Poisson equation on point cloud}

\author{
Zuoqiang Shi%
\thanks{Yau Mathematical Sciences Center, Tsinghua University, Beijing, China,
100084. \textit{Email: zqshi@math.tsinghua.edu.cn.}%
}
\and Jian Sun %
\thanks{Yau Mathematical Sciences Center, Tsinghua University, Beijing, China,
100084. \textit{Email: jsun@math.tsinghua.edu.cn.}%
}
%\and Zhen Li%
%\thanks{Mathematical Sciences Center, Tsinghua University, Beijing, China,
%100084. \textit{Email: zli12@mails.tsinghua.edu.cn.}%
}

%\keywords{Laplace-Beltrami operator; Neumann boundary; point cloud; point integral method; convergence analysis.}
\date{}
 \maketitle

\begin{abstract}
The Laplace-Beltrami operator (LBO) is a fundamental object associated to Riemannian manifolds, 
which encodes all intrinsic geometry of the manifolds and has many desirable properties.
Recently, we proposed a novel numerical method, Point Integral method (PIM), to discretize the Laplace-Beltrami operator on point clouds \cite{LSS}.
In this paper, we analyze the convergence of Point Integral method (PIM) for Poisson equation with Neumann boundary condition
on submanifolds isometrically embedded in Euclidean spaces. 
\end{abstract}

%\newpage
\section{Introduction}

The partial differential equations on manifolds arise in a wide variety of applications. In many problems, 
including material science \cite{CFP97,EE08}, fluid flow \cite{GT09,JL04},
biology and biophysics \cite{BEM11,ES10,NMWI11,WD08}, people need to study the physical process, for instance diffusion and 
convection, in curved surfaces which introduce different kinds of PDEs in surfaces. 
It has been several decades to develop numerical methods for solving PDEs in surfaces. Many methods have been developed, 
such as 
surface finite element method \cite{DE-Acta}, level set method \cite{Bertalmio,XZ03}, grid based particle method \cite{LZ09,LZ11} 
and closest point method \cite{RM08,MR09}. 

Recently, manifold model attracts more and more attentions in data analysis and image processing \cite{Peyre09,LDMM,belkin2003led,Coifman05geometricdiffusions}. 
 It is well known that 
 PDEs on the manifold, especially the Laplace equation, encode lots of intrinsic information of the
manifold which is very helpful to reveal the underlying structure hidden in the data. 
In the data analysis problems, data is usually represented as a collection of
points embedding in a high dimensional Euclidean space, which is refereed as point cloud. 
The point cloud gives a sample of the manifold and we need to solve PDEs on the unstructured point cloud. 
Usually, the point cloud is embedded in a high dimensional space, the traditional methods for PDEs on 2D surfaces do not work in this case.

  %  Also in many data analysis problems, 
% PDE based methods provide powerful tools \cite{Peyre09,LDMM,belkin2003led,Coifman05geometricdiffusions}.
% In many problems, 
% Usually, the point cloud data lies in a manifold whose 
% dimension is much lower than the ambient Euclidean space.
% However, It is very challenging to solve PDEs on point cloud. 

% In recent years, it has attracted many attentions to develop efficient numerical methods 
% for solving PDEs on manifolds, 
% Finite element method has very good theoretical peoperties. Dziuk~\cite{Dziuk88} showed FEM
% converges quadratically in $L^2$ norm and linearly in $H^1$ norm for solving the Poisson equations. 
% Recently, Wardetzky~\cite{Wardetzky06} 
% showed the same convergence rate holds even with approximating sequences of meshes. The eigensystem 
% of Laplace-Beltrami operator also converges for FEM linearly~\cite{Strang73, Dodziuk76, Wardetzky06}.
% Boffi~\cite{Boffi10} showed the convergence of FEM for general compact operators. 
% However, to apply FEM, global mesh is needed which is not easy to generate especially in high dimensional space. 
% The other methods also need extra information, level set function in level set method, closest point function in closest point method. 
% These information is not easy to obtain from point cloud neither.

In past few years, many efforts were devoted to develop alternative numerical methods to 
discretize the differential operators on point cloud.
Liang et al. proposed to discretize the differential operators on point cloud by local least square approximations of the manifold \cite{Liang13}.
Their method can achieve high
order accuracy and enjoy more flexibility since no mesh is needed. In principle, it can be applied to
manifolds with arbitrary dimensions and co-dimensions with or without boundary. However, if the dimension of the manifold is high, 
this method may not be stable since high order polynomial is used to fit the data. Later, Lai et al. proposed local mesh method to approximate the 
differential operators on point cloud \cite{Lai13}. The main idea is to construct mesh locally around each point by using K nearest neighbors. The local mesh is easier 
to construct than global mesh. Based on the local mesh, it is easy to discretize differential operators and compute integrals. However, when the dimension of the 
manifold is high, even local mesh is not easy to construct.

In~\cite{LSS}, we have proposed a novel numerical method, point integral method (PIM), 
to solve the Poisson equation on point cloud.
The main idea of the point integral method is to
approximate the Poisson equation by the following integral equation:
\begin{equation}
\label{eq:integral-intro}
  -\int_\M  \Delta_\M u(\by)\rhk d\mu_\by\approx \invt\int_{\M} \hk(u(\bx) - u(\by))\mathd\mu_\by-2\int_{\p\M}\rhk \frac{\p u}{\p\bn}(\by)\mathd \tau_\by,
\end{equation}
where $\bn$
is the out normal of $\M$, $\M$ is a smooth $k$-dimensional manifold embedded in $\mathbb{R}^d$ and $\p\M$ is the boundary of $\M$. 
$R_t(\bx,\by)$ and $\bar{R}_t(\bx,\by)$ are kernel functions given as follows
\begin{equation}
\label{eq:kernel}
R_t(\bx, \by) = C_tR\left(\frac{|\bx -\by|^2}{4t}\right),\quad
\bar{R}_t(\bx, \by) = C_t\bar{R}\left(\frac{|\bx -\by|^2}{4t}\right)
\end{equation}
where $C_t = \frac{1}{(4\pi t)^{k/2}}$ is the normalizing factor.
 $R\in C^2(\mathbb{R}^+) $ be a positive function which is integrable over $[0,+\infty)$,
\begin{equation*}
  \bar{R}(r)=\int_r^{+\infty}R(s)\mathd s.
\end{equation*}
$\Delta_\mathcal{M}=\diver(\nabla)$ is the Laplace-Beltrami operator on $\mathcal{M}$.  
Let $\Phi: \Omega\subset \mathbb{R}^k\rightarrow \cM\subset\mathbb{R}^d$ be a local parametrization of $\cM$ and $\theta\in \Omega$.
For any differentiable function $f:\cM\rightarrow \mathbb{R}$,
%let $F(\theta)=f(X(\theta))$,
 define the gradient on the manifold
\begin{align}
  \label{eq:diff-M}
  \nabla f(\Phi(\theta))&=\sum_{i,j=1}^m g^{ij}(\theta)\frac{\p \Phi}{\p\theta_i}(\theta)\frac{\p f(\Phi(\theta))}{\p\theta_j}(\theta),
\end{align}
and for vector field $F:\M\rightarrow T_\bx\M$ on $\M$, where $T_\bx\M$ is the tangent space of $\M$ at $\bx\in \M$, the divergence is defined as
\begin{align}
\label{eq:diver}
\diver (F)&= \frac{1}{\sqrt{\det G}}\sum_{k=1}^d\sum_{i,j=1}^m\frac{\p}{\p \theta_i}\left(\sqrt{\det G}g^{ij}F^k(\Phi(\theta))\frac{\p \Phi^k}{\p\theta_j}\right)
\end{align}
where $(g^{ij})_{i,j=1,\cdots,k}=G^{-1}$, $\det G$ is the determinant of matrix $G$ and $G(\theta)=(g_{ij})_{i,j=1,\cdots,k}$ is the first fundamental form which is defined by
\begin{eqnarray}
  \label{eq:remainn}
  g_{ij}(\theta)=\sum_{k=1}^d\frac{\p \Phi_k}{\p\theta_i}(\theta)\frac{\p \Phi_k}{\p\theta_j}(\theta),\quad i,j=1,\cdots,m.
\end{eqnarray}
and $(F^1(\bx),\cdots,F^d(\bx))^t$ is the representation of $F$ in the embedding coordinates.

Using the integral approximation, we transfer the Laplace-Beltrami operator to an integral operator. The integral operator is easy to
be discretized on point clouds using some quadrature rule, since there is not any differential operators inside.
This is the essential ingredient in the point integral method.
Similar integral approximation is also used in nonlocal diffusion and peridynamic model \cite{Du-SIAM,book-nonlocal,DGLZ13,DJTZ13,ZD10}.

The point integral method is also related with the graph Laplacian. Graph Laplacian is a discrete object associated to a graph, which reveals many 
properties of graphs~\cite{Chung}. 
It is observed in~\cite{BelkinN05, Lafon04diffusion, Hein:2005:GMW, Singer06} that
the graph Laplacian with the Gaussian weights well approximates the LBO when the 
vertices of the graph are assumed to sample the underlying manifold.
When there is no boundary, Belkin and Niyogi~\cite{CLEM_08} showed the spectra of the graph Laplacian with 
Gaussian weights converges to that of $\Delta_\M$. The main issue that remains is how to deal with the boundary. 
In fact, near the boundary, it was observed~\cite{Lafon04diffusion, BelkinQWZ12} that the graph Laplacian is 
dominated by the first order derivative and thus fails to be true Laplacian.
Recently, Singer and Wu~\cite{Singer13} showed the spectral convergence of the graph Laplacian in 
the presence of the Neumann boundary. In both~\cite{CLEM_08} and~\cite{Singer13}, the convergence analysis is 
based on the connection between the graph Laplacian and the heat operator, and thus the Gaussian weights are essential. 

The main contribution of this paper is that, for Poisson equation with Neumann boundary condition,  we prove that the numerical solution computed 
by the PIM converges to the exact solution in $H^1$ norm as the density of the sample points tends to infinity. 
Unlike the methods used in graph Laplacian, we do not relate the integral operator to heat kernel. Instead, we use the strategy which is standard in 
numerical analysis to prove the convergence. 

It is well known in the numerical analysis that the convergence is the summation of consistency and stability. 
We prove that the coercivity of the original Laplace-Beltrami operator is preserved in the point integral method. This 
implies the stability of the point integral method. Together with the estimate of the truncation error, we get the convergence of the point integral method.
% We will show that the PIM is consistent (Theorem \ref{thm:integral_error} and \ref{thm:dis_error}) and 
% stable (Theorem~\ref{thm:regularity} and \ref{thm:regularity_boundary}).
% In proving the stability of the PIM, the key part is that that the point integral method preserves 
% the coercivity of the Laplace-Beltrami operator (see Theorem \ref{thm:elliptic_v} and \ref{thm:elliptic_L_t}), 
% which is interesting on its own. 

The remaining of this paper is organized as following. In Section 2, 
we describe the point integral method for Poisson equation with Neumann boundary condition. %a brief introduction of the Point Integral method. 
The convergence result is stated in Section 3. The structure of the proof is shown in Section 4. 
% In section 6, we show several basic estimates related to the properties of smooth submanifolds 
% and the convolutions with the kernel $R$. 
%which will be used often in the proofs. 
The main body of the proof is in Section 5, Section 6 and Section 7. 
Finally, conclusions and discussion on the future work are given in Section 8.

\section{Point Integral Method}
\label{sec:methods}

In this paper, we consider Poisson equation on a smooth, compact $k$-dimensional submanifold $\mathcal{M}$ 
in $\R^d,\; d\ge k$ with the Neumann boundary 
\begin{equation} 
\left\{\begin{array}{rl}
      \Delta_\mathcal{M} u(\bx)=f(\bx),&\bx\in \mathcal{M} \\
      \frac{\p u}{\p \bn}(\bx)=b(\bx),& \bx\in \p \mathcal{M}
\end{array}
\right.  
\label{eqn:neumann} 
\end{equation}
% It is known that the solution of the Neumann problem \eqref{eqn:neumann} is not unique.  
% To fix the solution, we require that the average of the solution to be $0$, i.e. the solution of \eqref{eqn:neumann}  
% satisfies that $ \int_\M u(\bx)\mathd \mu_\bx=0$.

The manifold $\M$ is sampled with a set of sample points $P$ and a subset $S\subset P$ sampling the boundary of $\M$.
List the points in $P$ respectively $S$ in a fixed
order $P=(\bfp_1, \cdots, \bfp_n)$ where $\bfp_i \in \R^d, 1\leq i\leq n$, respectively $S=(\bfs_1, \cdots, \bfs_m)\subset P$.

In addition, assume we are given
two vectors $\bV = (V_1, \cdots, V_n)^t$ where $V_i$ is an volume weight of $\bfp_i$ in $\M$, and
$\bA= (A_1, \cdots, A_m)^t$ where $A_i$ is an area weight of $\bfs_i$ in $\p \M$, so that for any $f\in C^1(\M)$ and $g\in C^1(\M)$, 
$$\sum_{i=1}^n f(\bfp_i) V_i\approx\int_\M f(\bx) d\mu_\bx,\quad   \sum_{i=1}^m g(\bfs_i) A_i\approx \int_{\p \M} g(\bx) d\tau_\bx .$$ 
 Here $d\mu_\bx$ and $d\tau_\bx$ are the volume form
of $\M$ and $\p \M$, respectively.

Using the integral approximation \eqref{eq:integral-intro}, the Poisson equation is approximated by an integral equation,
\begin{equation}
  \label{eq:integral}
  \invt\int_{\M} \hk(u(\bx) - u(\by))\mathd\mu_\by-2\int_{\p\M}\rhk b(\by)\mathd \tau_\by=\int_\M  f(\by)\rhk d\mu_\by
\end{equation}
In the integral equation, there is not any differential operators. It is easy to discretize on the point cloud with the weight vectors, $\bV$ and $\bA$,
\begin{equation}
  \label{eq:dis}
  \sum_{\bfp_j\in P}R_t(\bfp_i,\bfp_j)(u_i-u_j)V_j-2\sum_{\bfs_j\in S}\bar{R}_t(\bfp_i,\bfs_j)b(\bfs_j)A_j=\sum_{\bfp_j\in P}\bar{R}_t(\bfp_i,\bfp_j)f(\bfp_j)V_j
\end{equation}
The solution ${\bf u} = (u_1, \cdots, u_n)^t$ to above linear system
gives an approximation of the solution to the problem~\eqref{eqn:neumann}. 

\begin{remark} In the point integral method, we need the volume weight $\bf{V}$ and area weight $\bf{A}$. We remark that 
it is much easier to obtain the volume weight $\bf{V}$ than to generate a consistent mesh for $\M$ with good shaped elements. 
If $\bV$ and $\bA$ are not given, they can be estimated as follows.
\begin{itemize}
\item If a mesh with the vertices $P$ approximating $\mathcal{M}$ is given, both weight vectors $\bV$ and $\bA$ can be easily
estimated from the given mesh by summing up the volume of the simplices incident to the vertices. 
Note that there is no requirement on 
the shape of the elements in the mesh. 
\item If the points in $P$ and $S$ are independent samples from some distribution on $\M$ and $\p\M$ respectively,
then $\bV$ and $\bA$ can be obtained from the distribution.%  as the constant vector. The integral of the functions
% on $\M$ and $\partial \M$ can be estimated using Monte Carol method. In this case, from the central limit theorem, $h$ is of 
% the order of $1/\sqrt{|P|}$;
\item Finally, following~\cite{LuoSW09}, one can estimate the vectors $\bV$ and  $\bA$ by locally
approximating tangent spaces of $\M$ and $\p \M$, respectively. Specifically, for a point $\bfp\in P$,
project the samples near to $\bfp$ in $P$ onto the approximated tangent space at $\bfp$ and take the volume
of the Voronoi cell of $\bfp$ as its weight. In this way, one avoids constructing globally consistent meshes
for $\M$ and $\p\M$. 
\end{itemize}
\end{remark}

\section{Main Results}
The main contribution in this paper is to establish the convergence results for the point integral method
for solving the problem~\eqref{eqn:neumann}.
To simplify the notation and make the proof concise, in the analysis, we consider the homogeneous Neumann boundary
conditions, i.e. 
\begin{equation} 
\left\{\begin{array}{rl}
      -\Delta_\mathcal{M} u(\bx)=f(\bx),&\bx\in \mathcal{M} \\
      \frac{\p u}{\p \bn}(\bx)=0,& \bx\in \p \mathcal{M}
\end{array}
\right.  
\label{eqn:neumann-homo} 
\end{equation}
 The analysis can be easily generalized to 
the non-homogeneous boundary conditions.

The corresponding numerical scheme is 
\begin{equation}
\invt\sum_{\bfp_j \in P} R_t(\bfp_i,\bfp_j)(u_i - u_j)V_j  = \sum_{\bfp_j \in P} \bar{R}_t(\bfp_i,\bfp_j) f_jV_j.
\label{eqn:dis-homo}
\end{equation}
where $f_j=f(\bfp_j)$.

Before proving the convergence of the point integral method, we need to clarify the meaning of the convergence between the point cloud $(P,\mathbf{V})$
and the manifold $\M$. In this paper, we consider the convergence in the sense that $h(P,\mathbf{V},\M)\rightarrow 0$ where $h(P,\mathbf{V},\M)$ is the 
{\it integral accuracy index} defined as following,
\begin{definition}[Integral Accuracy Index]
  \label{def:h}
For the point cloud $(P,\mathbf{V})$ which samples the manifold $\M$, the integral accuracy index $h(P,\mathbf{V},\M)$ is defined as
\begin{equation*}
h(P,\mathbf{V},\M)=\sup_{f\in C^1(\M)}\frac{\left|\int_\M f(\by) \mathd\mu_\by - \sum_{\bfp_i\in P} f(\bfp_i)V_i\right|}{|\text{\rm supp}(f)|\|f\|_{C^1(\M)}}.
\end{equation*}
where $\|f\|_{C^1(\M)} = \|f\|_\infty +\|\nabla f\|_\infty$ and
$|\text{\rm supp}(f)|$ is the volume of the support of $f$.
\end{definition}

Using the definition of integrable index, we say that the point cloud $(P,\mathbf{V})$ converges to the manifold $\M$ if 
$h(P,\mathbf{V},\M)\rightarrow 0$. In the convergence analysis, we assume that $h(P,\mathbf{V},\M)$ is small enough.
\begin{remark} In some sense, $h(P,\mathbf{V},\M)$ is a measure of the density of the point cloud.
\begin{itemize}
\item If the volume weight $\bV$ comes from a mesh, one can obtain the integral accuracy index $h(P,\mathbf{V},\M)=O(\rho)$ where 
$\rho$ is the size of the elements in
the mesh and the angle between the normal space of an element and the normal space of $\M$ at the
vertices of the element is of order $\rho^{1/2}$~\cite{Wardetzky06}. 
\item If the point cloud is sampled from some distribution, 
from central limit theorem, $h(P,\mathbf{V},\M)\sim O(1/\sqrt{n})$ where $n$ is the number of point in $P$.
\end{itemize}
\end{remark}
\begin{remark}
  To consider the non-homogeneous Neumann boundary condition or Dirichlet boundary condition, we have to also assume that $h(S,\mathbf{A},\p\M)\rightarrow 0$,
where $S$ is the point set sample the boundary $\p\M$ and $\mathbf{A}$ is the corresponding volume weight on the boundary $\p\M$.
\end{remark}

To get the convergence, we also need some assumptions on the regularity of the submanifold $\M$ and
the integral kernel function $R$.
\begin{assumption}
\label{assumptions}
\begin{itemize}
\item[]
\item \rm Smoothness of the manifold: $\M, \p\M$ are both compact and $C^\infty$ smooth $k$-dimensional submanifolds isometrically embedded in a Euclidean space $\mathbb{R}^d$.
\item \rm Assumptions on the kernel function $R(r)$:
\begin{itemize}
\item[\rm (a)] \rm Smoothness: $R\in C^2(\mathbb{R}^+)$;
\item[(b)] Nonnegativity: $R(r)\ge 0$ for any $r\ge 0$.
\item[(c)] Compact support:
%$R(r) \le 1$ for $\forall r \in \R^+$ and $R(r) = 0$ for $\forall r >1$.
$R(r) = 0$ for $\forall r >1$;
\item[(d)] Nondegeneracy:
 $\exists \delta_0>0$ so that $R(r)\ge\delta_0$ for $0\le r\le\frac{1}{2}$.
\end{itemize}
%\item Assumptions on $t$ and $h$: $t$ and $h/\sqrt{t}$ are small enough, i.e., $t$ and $h/\sqrt{t}$ are less than a positive constant
%which only depends on $\M$.
\end{itemize}
\end{assumption}
\begin{remark}
  The assumption on the kernel function is very mild. The compact support assumption can be relaxed to exponentially decay, like Gaussian kernel.
 In the nondegeneracy assumption, $1/2$ may be replaced by a positive number $\theta_0$ with $0<\theta_0<1$. Similar assumptions on the kernel function 
is also used in analysis the nonlocal diffusion problem \cite{DLZ13}.
\end{remark}

\begin{remark}
It is for the sake of simplicity that $R$ is assumed to be compactly supported. After some mild modifications
of the proof,
the same convergence results also hold for any kernel function that decays exponentially, like the Gaussian kernel 
$G_t(\bx, \by) = C_t\exp\left(-\frac{|\bx-\by|^2}{4t}\right)$. 
In fact, for any $s\ge 1$ and any $\epsilon>0$, 
the $H^s$ mass of the Gaussian kernel over the domain $\Omega=\{ \by \in \M | |\bx-\by|^2\geq t^{1+\epsilon}\}$
decays faster than any polynomial in $t$ as $t$ goes to $0$, i.e., 
$\lim_{t\rightarrow 0} \frac{\|G_t(\bx, \by)\|_{H^s(\Omega)}}{t^\alpha} = 0$ for any $\alpha$. In this way, 
we can bound any influence of the integral outside a compact support. 
\end{remark}
All the analysis in this paper is under the assumptions in Assumption \ref{assumptions} and $h(P,\mathbf{V},\M)$, $t$ are small enough. 
In the theorems and the proof, without introducing any confusions, we omit the statement of the assumptions.

The solution of the point integral method is a vector $\bf u$ while
the solution of the problem~\eqref{eqn:neumann-homo} is a function defined on $\M$. To make them 
comparable, for any solution $\bfu = (u_1, \cdots, u_n)^t$ to the problem~\eqref{eqn:dis-homo}, 
we construct a function on $\M$ 
\begin{equation}
\label{eqn:interp_neumann}
I_{\bff}(\bfu) (\bx) = \frac{ \sum_{\bfp_j \in P} \hkxpj u_j V_j - t\sum_{\bfp_j \in P} \rhkxpj f_jV_j} {\sum_{\bfp_j \in P} \hkxpj V_j }.
\end{equation}
It is easy to verify that $I_{\bff}(\bfu)$ interpolates $\bfu$ at the sample points $P$, i.e., 
$I_{\bff}(\bfu)(\bfp_j) = u_j$ for any $\bfp_j\in P$. 
%Since $I_{\bff, \bfg}(\bfu)$ is smooth, it is 
%well-controlled by the vector $\bfu$ provided that $P$ densely samples $\M$. 
The following theorem guarantees the convergence of the  point integral method. 
\begin{theorem}
Let $u$ be the solution to Problem~\eqref{eqn:neumann-homo} with  $f\in C^1(\M)$ and  
the vector $\bfu$ be the solution to the problem~\eqref{eqn:dis-homo}. Then there exists constants $C$ and $T_0$ 
only depend on $\M$,
such that for any $t\le T_0$
\begin{equation}
  \|u-I_{\bff}(\bfu)\|_{H^1(\M)} \leq C\left(t^{1/2} + \frac{h(P,\bf{V},\M)}{t^{3/2}}\right)\|f\|_{C^1(\M)}. 
\end{equation}
where $h(P,\bf{V},\M)$ is the integral accuracy index.
\label{thm:poisson_neumann}
\end{theorem}

\section{Structure of the Proof}
\label{sec:intermediate}
To simplify the notation, we introduce an integral operator,
\begin{equation}
  \label{eq:Lt}
L_t u=\invt\int_{\M} \hk(u(\bx) - u(\by))\mathd\mu_\by
\end{equation}
Roughly speaking,
 the proof the convergence includes estimate of the truncation error $L_t(u-I_{\bff}(\bfu))$
and the stability of the integral operator $L_t$. Here $u(\bx)$ is the solution
of the problem~\eqref{eqn:neumann-homo} and $\bfu$ is the solution of
the problem~\eqref{eqn:dis-homo}.

% This strategy is 
% standard in numerical analysis. It is well known that consistency together with stability imply convergence.
% On the other hand, the point integral method has some 
% special structures both in truncation error and stability, which makes the analysis a little more involved. 

First, we have following theorem regarding the stability of the operator $L_t$. 
\begin{theorem}
Let $u(\bx)$ solves the integral equation
\begin{eqnarray*}
  L_t u = r(\bx)
\end{eqnarray*}
where $r\in H^1(\M)$ with $\int_\M r(\bx)\mathd\mu_\bx=0$.
There exist constants $C>0, T_0>0$ independent on $t$, such that
\begin{eqnarray*}
  \|u\|_{H^1(\M)}\le C\left(\|r\|_{L^2(\M)}+t\|\nabla r\|_{L^2(\M)}\right)
\end{eqnarray*}
as long as $t\le T_0$.
\label{thm:regularity}
\end{theorem}

To apply the stability result, we need $L_2$ estimate of $L_t(u-I_{\bff}(\bfu))$ and $\nabla L_t(u-I_{\bff}(\bfu))$.
These truncation errors are analyzed in following two theorems by spliting the truncation error $L_t(u-I_{\bff}(\bfu))$
\begin{equation*}
  L_t(u-I_{\bff}(\bfu))=L_t(u-u_t))+L_t(u_t-I_{\bff}(\bfu))
\end{equation*}
where $u_t$ is the solution of the integral equation 
\begin{equation}
  \label{eq:integral-homo}
\invt\int_{\M} \hk(u(\bx) - u(\by))\mathd\mu_\by=\int_\M f(\by)\bar{R}_t(\bx,\by)\mathd \mu_\by.
\end{equation}

For the second term, we have
\begin{theorem}
%Under the assumptions in Section \ref{assumptions},
Let $u_t(\bx)$ be the solution of the problem~~\eqref{eq:integral-homo} and $\bfu$ be the solution of
the problem~\eqref{eqn:dis-homo}. If
$f\in C^1(\M)$ ,
%$f\in C(\M), g\in C(\p\M)$
 then there exists constants
$C, T_0$ depending only on $\M$, so that
\begin{eqnarray}
\|L_{t} \left(I_{\bff}\bfu - u_{t}\right)\|_{L^2(\M)} &\leq& \frac{Ch(P,\mathbf{V},\M)}{t^{3/2}}\|f\|_{C^1(\M)},\\
\label{eqn:dis_error_l2}
\|\nabla L_{t} \left(I_{\bff}\bfu - u_{t}\right)\|_{L^2(\M)} &\leq& \frac{Ch(P,\mathbf{V},\M)}{t^{2}}\|f\|_{C^1(\M)}.
\label{eqn:dis_error_dl2}
\end{eqnarray}
as long as $t\le T_0$ and $\frac{h(P,\mathbf{V},\M)}{\sqrt{t}}\le T_0$, $h(P,\mathbf{V},\M)$ is the integral difference index in 
Definition \ref{def:h}.
\label{thm:dis_error}
\end{theorem}

In the analysis, we found that the error term $L_t(u-u_t)$ has boundary layer structure. In the interior region, 
it is $O(\sqrt{t})$ and in a layer adjacient to the boundary with width $O(\sqrt{t})$, the error is $O(1)$.
\begin{theorem}
% Under the assumptions in Section \ref{assumptions},
 Let $u(\bx)$ be the solution
of the problem~\eqref{eqn:neumann-homo} and $u_t(\bx)$ be the solution of the corresponding
integral equation \eqref{eq:integral-homo}. Let
\begin{align}
\label{eq:error_boundary}
&I_{bd} =\sum_{j=1}^d \int_{\p\M}n^j(\by)(\bx-\by)\cdot\nabla(\nabla^ju(\by))\rhk p(\by)\mathd \tau_\by,
\end{align}
and
\begin{align}
L_t (u- u_t)=I_{in}+I_{bd}\nonumber.
\end{align}
where $\bn(\by)=(n^1(\by),\cdots,n^d(\by))$ is the out normal vector of $\p\M$ at $\by$, $\nabla^j$ is the $j$th component of gradient $\nabla$.
%$\al=X^{-1}(\bx), \xi(\bx,\by)=X^{-1}(\bx)-X^{-1}(\by)$ and $\eta(\bx,\by)=\xi^{i'}(\bx,\by)\p_{i'}\Phi(\al)$.

If $u\in H^3(\M)$, then there exists constants
$C, T_0$ depending only on $\M$ and $p(\bx)$, so that,
\begin{eqnarray}
\label{eqn:integral_error_int}
\left\|I_{in}\right\|_{L^2(\M)}\le Ct^{1/2}\|u\|_{H^3(\mathcal{M})},\quad
\left\|\nabla I_{in}\right\|_{L^2(\M)} \leq C\|u\|_{H^3(\mathcal{M})},
\end{eqnarray}
as long as $t\le T_0$.
\label{thm:integral_error}
\end{theorem}

% Using the definition of the boundary term $I_{bd}$, \eqref{eq:error_boundary},
% it is easy to check that 
% \begin{eqnarray*}
% \left\|I_{bd}\right\|_{L^2(\M)}= O(t^{1/4}),\quad
% \left\|\nabla I_{bd}\right\|_{L^2(\M)} = O(t^{-1/2}),
% \end{eqnarray*}
% Based on this estimation, Theorem \ref{thm:regularity} and Theorem \ref{thm:integral_error} give that 
% \begin{equation*}
%   \|u-u_t\|_{H^1(\M)}=O(t^{1/4}).
% \end{equation*}
% This proves the convergence, however the convergence rate is relatively low. This low rate comes from the boundary term. From interior term only, 
% the rate is $\sqrt{t}$. Notice that the boundary term has a specific integral formula given in \eqref{eq:error_boundary}. 
% Using this formula, we know that the boundary term concentrates in a small layer adjacent to the boundary whose width is of the order of $\sqrt{t}$ and vanish in the 
% interior region.
% Utilizing this special structure, we could get better convergence rate with the help of a stability estimate specifically for the boundary term, which is 
% given in Theorem \ref{thm:regularity_boundary}.
To utilizing the boundary layer structure, we need a stability result specifically for the boundary term.
\begin{theorem}
Let $u(\bx)$ solves the integral equation
\begin{eqnarray*}
  L_t u = \int_{\p\M}\mathbf{b}(\by)\cdot(\bx-\by)\rhk \mathd \tau_\by-\bar{b}
\end{eqnarray*}
where $|\M|=\int_\M \mathd \mu_\bx$ and 
\begin{eqnarray*}
  \bar{b}=\frac{1}{|\M|}\int_\M \left(\int_{\p\M}\mathbf{b}(\by)\cdot(\bx-\by)\rhk \mathd \tau_\by\right)\mathd\bx.
\end{eqnarray*}
Then, there exist constant $C>0, T_0>0$ independent on $t$, such that
\begin{eqnarray*}
  \|u\|_{H^1(\M)}\le C\sqrt{t}\;\|\mathbf{b}\|_{H^1(\M)}.
\end{eqnarray*}
as long as $t\le T_0$.
\label{thm:regularity_boundary}
\end{theorem}

Theorem~\ref{thm:poisson_neumann} is an easy corollary from Theorems \ref{thm:regularity}, \ref{thm:dis_error}, \ref{thm:integral_error} and \ref{thm:regularity_boundary}. 
Theorem \ref{thm:dis_error} and Theorem \ref{thm:regularity} imply that 
 $$\|u_t-I_{\bff}(\bfu)\|_{H^1(\M)}=O\left(\frac{h(P,\mathbf{V},\M)}{t^{3/2}}\right).$$ 
and Theorem \ref{thm:regularity}, \ref{thm:integral_error} and \ref{thm:regularity_boundary} imply
 $$\|u-u_t\|_{H^1(\M)}=O\left(t^{1/2}\right),$$ 
which prove Theorem~\ref{thm:poisson_neumann}. 

In the rest of the paper, we prove Theorem 
  \ref{thm:regularity}, \ref{thm:dis_error}, \ref{thm:integral_error} and \ref{thm:regularity_boundary} respectively.

%\subsection{The Neumann boundary}
%\input{intermediate}
%\subsection{The Dirichlet boundary}
%\input{intermediate_robin}

% \section{Basic Estimates}
% \input{basic_estimate}

\section{Error analysis of the integral approximation (Theorem \ref{thm:integral_error})}
\label{sec:err_int}

% Let $r(\bx)=L_t u-L_tu_t$ where $u$ and $u_t$ are the solution of \eqref{eq:neumann-homo} and \eqref{eq:integral-homo} respectively. 
% Using integration by parts, we have%that $u_t$ satisfies the integral equation \eqref{eq:integral-homo}, we have
% \begin{align}
% \label{eq:r1}
%   r(\bx)=&\frac{1}{t}\int_{\M}R_t(\bx,\by)(u(\bx)-u(\by))p(\by)\mathd\mu_\by-\int_{\M}\diver(p^2(\by)\nabla u(\by))\frac{\bar{R}_t(\bx,\by)}{p(\by)}\mathd\mu_\by\\
% & \hspace{0cm}-2\int_{\p\M}\bar{R}_t(\bx,\by)\frac{\p u}{\p \bn}(\by)p(\by)\mathd\tau_\by\nonumber\\
% =&\frac{1}{t}\int_{\M}(u(\bx)-u(\by)-(\bx-\by)\cdot\nabla u(\by))R_t(\bx,\by)p(\by)\mathd\mu_\by\nonumber\\
% &\hspace{0cm}
%  -\int_\M\Delta_\M u(\by)\bar{R}_t(\bx,\by)p(\by)\mathd\mu_\by\nonumber
% \end{align}
% The main idea of the proof is the Taylor expansion,
% \begin{equation*}
%   u(\bx)-u(\by)-(\bx-\by)\cdot\nabla u(\by)=\frac{1}{2}(\bx-\by)^T\cdot\mathbf{H}_u(\by)\cdot (\bx-\by)+O(|\bx-\by|^3)
% \end{equation*}
% where $\mathbf{H}_u(\by)$ is the Hessian matrix of $u$ at $\by$. 

% Using the integration by parts, the second order term actually gives Laplace-Beltrami operator which cancel with the second term in \eqref{eq:r1}. 

% In manifold, the Taylor expansion and integration by parts are more complicated. 
% To make the whole idea rigorous
In this section, we need to introduce a special parametrization of the manifold $\M$. This parametrization is based on following proposition.
\begin{proposition}
Assume both $\M$ and $\p \M$ are $C^2$ smooth and
$\sigma$ is the minimum of the reaches of $\M$ and $\p \M$. 
For any point $\bx\in \M$, there is a neighborhood $U\subset \M$ of $\bx$, 
so that there is a parametrization 
$\Phi: \Omega \subset \R^k \rightarrow U$ satisfying the following conditions. For any $\rho \leq 0.1$, 
\begin{enumerate}
\item[(i)] $\Omega$ is convex and contains at least half of the ball
$B_{\Phi^{-1}(\bx)}(\frac{\rho}{5} \sigma)$, 
i.e., $vol(\Omega\cap B_{\Phi^{-1}(\bx)}(\frac{\rho}{5} \sigma)) > \frac{1}{2}(\frac{\rho}{5}\sigma)^k w_k$ where $w_k$ is the volume of unit ball in $\R^k$;
\item[(ii)] $B_{\bx}(\frac{\rho}{10} \sigma) \cap \M \subset U$. 
\item[(iii)] The determinant the Jacobian of $\Phi$ is bounded:
$(1-2\rho)^k \leq |D\Phi|  \leq (1+2\rho)^k$ over $\Omega$.  
\item[(iv)] For any points $\by, \bz \in U$, 
$1-2\rho \leq \frac{|\by-\bz|}{\left|\Phi^{-1}(\by) - \Phi^{-1}(\bz)\right|}  \leq 1+3\rho$.
\end{enumerate}
\label{prop:local_param}
\end{proposition}
This proposition basically says there exists a local parametrization of small distortion
if $(\M, \p\M)$ satisfies certain smoothness, and moreover, the parameter domain is 
convex and big enough. The proof of this proposition can be found in Appendix A. Next, we introduce a special parametrization of the manifold $\M$.

Let $\rho=0.1$, $\sigma$ be the minimum of the reaches of $\M$ and $\p \M$ and $\delta=\rho \sigma/20$. For any $\bx\in \M$, denote
\begin{eqnarray}
\label{eq:net}
  B_{\bx}^\delta=\left\{\by\in \mathcal{M}: |\bx-\by|\le \delta\right\},\quad \mathcal{M}_{\bx}^t=\left\{\by\in \mathcal{M}: |\bx-\by|^2\le 4t\right\}
\end{eqnarray}
and we assume $t$ is small enough such that $2\sqrt{t}\le \delta$.
%where $\delta>0$ is a positive number which is small enough and only depends on the manifold $\M$ and we let .

Since the manifold $\mathcal{M}$ is compact, there exists a $\delta$-net, $\mathcal{N}_\delta=\{ \bq_i\in \mathcal{M},\;i=1,\cdots,N\}$, such that
\begin{eqnarray}
  \mathcal{M}\subset \bigcup_{i=1}^N B_{\bq_i}^\delta.\nonumber
\end{eqnarray}
and there exists a partition of $\M$, $\{\mathcal{O}_i, \;i=1,\cdots,N\}$, such that $\mathcal{O}_i\cap \mathcal{O}_j=\emptyset,\; i\neq j$ and
\begin{equation*}
  \M=\bigcup_{i=1}^N \mathcal{O}_i,\quad \mathcal{O}_i\subset B_{\bq_i}^\delta,\quad i=1,\cdots,N.
\end{equation*}

Using Proposition \ref{prop:local_param}, there exist
a parametrization $\Phi_i: \Omega_i\subset\mathbb{R}^k \rightarrow U_i\subset \mathcal{M},\; i=1,\cdots, N$, such that
\begin{itemize}
\item[1.] (Convexity) $B_{\bq_i}^{2\delta}\subset U_i$ and $\Omega_i$ is convex.
\item[2.] (Smoothness) $\Phi_i\in C^3(\Omega_i)$;
\item[3.] (Locally small deformation) 
For any points $\theta_1, \theta_2\in \Omega_i$, 
$$\frac{1}{2}\left|\theta_1-\theta_2\right| \leq \left\|\Phi_i(\theta_1)-\Phi_i(\theta_2)\right\|  
\leq 2\left|\theta_1-\theta_2\right|.$$
\end{itemize}
Using the partition, $\{\mathcal{O}_i, \;i=1,\cdots,N\}$, for any $\by\in \M$, there exists unique $J(\by)\in \{1,\cdots,N\}$, %$\mathcal{O}_j$, 
such that 
\begin{equation}
\label{def:index-j}
\by\in \mathcal{O}_{J(\by)}\subset B_{\bq_{J(\by)}}^\delta.
\end{equation}
Moreover, using the condition, $2\sqrt{t}\le \delta$, we have $\mathcal{M}_{\by}^t \subset B_{\bq_{J(\by)}}^{2\delta}\subset U_{J(\by)}$.
Then $\Phi_{J(\by)}^{-1}(\bx)$ and $\Phi_{J(\by)}^{-1}(\by)$ are both well defined for any $\bx\in \mathcal{M}_{\by}^t$.

Now, we define an auxiliary function, $\eta(\bx,\by)$ for any $\by\in \M,\;\bx\in \mathcal{M}_{\by}^t$. 
Let 
\begin{equation}
\label{def:eta}
\xi(\bx,\by)=\Phi_{J(\by)}^{-1}(\bx)-\Phi_{J(\by)}^{-1}(\by)\in \mathbb{R}^k,\quad \eta(\bx,\by)=\xi(\bx,\by)\cdot \p \Phi_{J(\by)}(\alpha(\bx,\by))\in \mathbb{R}^d, 
\end{equation}
where $\alpha(\bx,\by)=\Phi^{-1}_{J(\by)}(\by)$ and 
$\p$ is the gradient operator in the parameter space, i.e.
\begin{equation*}
  \p \Phi_j(\theta)=\left(\frac{\p \Phi_j}{\p \theta_1}(\theta),\frac{\p \Phi_j}{\p \theta_2}(\theta),\cdots,\frac{\p \Phi_j}{\p \theta_k}(\theta)\right),
\quad \theta\in \Omega_j\subset \mathbb{R}^k.
\end{equation*}

%Before the derivation, we give some notations.

%\subsection{Proof of Theorem \ref{thm:integral_error}}
Now we state the proof of Theorem \ref{thm:integral_error}.
\begin{proof}

Let $r(\bx)=-(L_t u-L_tu_t)$ be the residual, then we have
\begin{eqnarray}
  r(\bx)&=&-\frac{1}{t}\int_{\M}R_t(\bx,\by)(u(\bx)-u(\by))\mathd\mu_\by+2\int_{\p\M}\bar{R}_t(\bx,\by)g(\by)\mathd\tau_\by-\int_{\M}\bar{R}_t(\bx,\by)f(\by)\mathd\mu_\by
\nonumber\\
&=&-\frac{1}{t}\int_{\M}R_t(\bx,\by)(u(\bx)-u(\by))\mathd\mu_\by+\int_{\M}\bar{R}_t(\bx,\by)\Delta_\M u(\by)\mathd \mu_\by
\nonumber\\
&&+\frac{1}{t}\int_{\M}(\bx-\by)\cdot \nabla u(\by)R_t(\bx,\by)\mathd \mathd \mu_\by\nonumber\\
&=& -\frac{1}{t}\int_{\M}R_t(\bx,\by)(u(\bx)-u(\by)-(\bx-\by)\cdot \nabla u(\by))\mathd\mu_\by+\int_{\M}\bar{R}_t(\bx,\by)\Delta_\M u(\by)\mathd \mu_\by.
\nonumber
\end{eqnarray}
Here we use that fact that 
\begin{eqnarray}
  \int_{\M}\bar{R}_t(\bx,\by)f(\by)\mathd\mu_\by=\int_{\M}\bar{R}_t(\bx,\by)\Delta_\M u(\by)\mathd \mu_\by,\nonumber
\end{eqnarray}
and
\begin{eqnarray}
  \int_{\p\M}\bar{R}_t(\bx,\by)g(\by)\mathd\tau_\by&=&\int_{\p\M}\bar{R}_t(\bx,\by)\frac{\p u}{\p \bn}(\by)\mathd\tau_\by\nonumber\\
&=&\int_{\M}\bar{R}_t(\bx,\by)\Delta_\M u(\by)\mathd \mu_\by+\int_{\M}\nabla_\by \bar{R}_t(\bx,\by)\cdot \nabla u(\by)\mathd \mu_\by
\nonumber\\
&=&\int_{\M}\bar{R}_t(\bx,\by)\Delta_\M u(\by)\mathd \mu_\by+\frac{1}{2t}\int_{\M}(\bx-\by)\cdot \nabla u(\by)R_t(\bx,\by)\mathd \mu_\by,\nonumber
\end{eqnarray}
where the last equality comes from:
\begin{eqnarray}
\label{eqn:grad-m1}
  && \int_\mathcal{M}\nabla u(\by)\cdot \nabla_\by \bar{R}_t(\bx,\by) \mathd\mu_\by \\ 
&=&\frac{1}{2t}\int_\mathcal{M}\left(\p_{i'}\Phi^l g^{i'j'}\p_{j'}u(\by)\right)\;
\left( \p_{m'}\Phi^l g^{m'n'}\p_{n'}\Phi^j (x^j-y^j) R_t(\bx,\by)\right) \mathd\mu_\by \nonumber\\
&=&\frac{1}{2t}\int_\mathcal{M}\left( \p_{n'}\Phi^j g^{j'n'}\p_{j'}u(\by)\right)\;
\left(  (x^j-y^j) R_t(\bx,\by)\right) \mathd\mu_\by \nonumber\\
&=&\frac{1}{2t}\int_\mathcal{M}(x^j-y^j)\nabla^j u(\by) R_t(\bx,\by)  \mathd\mu_\by\nonumber\\
&=&\frac{1}{2t}\int_{\M}(\bx-\by)\cdot \nabla u(\by)R_t(\bx,\by)\mathd \mu_\by \nonumber.
\end{eqnarray}
Here, $\Phi^i,\; i=1,\cdots, d$ is the $i$th component of the parameterization function $\Phi$ and the parameterization function $\Phi=\Phi_{J(\by)}$, $J(\by)$ is the index function 
given in \eqref{def:index-j}. 
In the rest of the proof, without introducing any confusion, we always drop the subscript of the parameterization function. 
%In above derivation, we need the convexity property of the parameterization function to make sure all the integrals are well defined.

First, we split the residual $r(\bx)$ to four terms 
\begin{align}
%\label{eq:r1}
  r(\bx)
=&r_1(\bx)+r_2(\bx)+r_3(\bx)-r_4(\bx)\nonumber
\end{align}
where
\begin{eqnarray}
r_1(\bx)&=&\frac{1}{t}\int_{\M}\left(u(\bx)-u(\by)-(\bx-\by)\cdot \nabla u(\by)-\frac{1}{2}\eta^i\eta^j(\nabla^i\nabla^j u(\by))\right)R_t(\bx,\by) \mathd\mu_\by,\nonumber\\
r_2(\bx)&=&\frac{1}{2t}\int_{\M}\eta^i\eta^j(\nabla^i\nabla^j u(\by))R_t(\bx,\by) \mathd\mu_\by
-\int_\M \eta^i(\nabla^i\nabla^j u(\by)\nabla^j\bar{R}_t(\bx,\by) \mathd \mu_\by,\nonumber\\
r_3(\bx)&=& \int_\M \eta^i(\nabla^i\nabla^j u(\by)\nabla^j\bar{R}_t(\bx,\by) \mathd \mu_\by
+\int_\M \mbox{div} \; \left(\eta^i(\nabla^i\nabla u(\by)\right)
\bar{R}_t(\bx,\by) \mathd \mu_\by,\nonumber\\
r_4(\bx)&=&\int_\M \mbox{div} \; \left(\eta^i(\nabla^i\nabla u(\by)\right)\bar{R}_t(\bx,\by) \mathd \mu_\by
+ \int_{\M}\Delta_\M u(\by)\bar{R}_t(\bx,\by) \mathd \mu_\by.\nonumber
\nonumber
\end{eqnarray}
where $\nabla^i,\; i=1,\cdots,d$ is the $i$th component of the gradient $\nabla$, $\eta^i, \; i=1,\cdots,d$ is the $i$th component of $\eta(\bx,\by)$ defined in \eqref{def:eta}. 
To simplify the notation, 
we drop the variable $(\bx,\by)$ in the function $\eta(\bx,\by)$.
% Based on above parametrization, for any $\by\in \M$, there exists $\bx_i\in \mathcal{N}_\delta$, such that $\by\in B_{\bx_i}^\delta $ 
% Moreover, we denote
% $\Phi(\beta)=\bx, \Phi(\al)=\by, \xi=\beta-\al, \eta=\xi^i\p_i\Phi(\al)$,

% Now, $r(\bx)$ is splited to four terms as following
% \begin{eqnarray}
%  r(\bx)=r_1(\bx)+r_2(\bx)+r_3(\bx)-r_4(\bx).\nonumber
% \end{eqnarray}
Next, we will prove the theorem by estimating above four terms one by one. 
First, we consider $r_1$. Let
$$d(\bx,\by)=u(\bx)-u(\by)-(\bx-\by)\cdot \nabla u(\by)-\frac{1}{2}\eta^i\eta^j(\nabla^i\nabla^j u(\by)).$$
we have
\begin{eqnarray}
  \int_\M |r_1(\bx)|^2\mathd\mu_\bx&=&  \int_\M \left|\int_\M R_t(\bx,\by)d(\bx,\by) \mathd \mu_\by\right|^2\mathd\mu_\bx\nonumber\\
&\le & (\max_\by  )^2 \int_\M \left(\int_\M R_t(\bx,\by)\mathd \mu_\by\right)\left(\int_\M R_t(\bx,\by)|d(\bx,\by)|^2\mathd \mu_\by\right)\mathd\mu_\bx\nonumber\\
&\le & C\int_\M \int_\M R_t(\bx,\by)|d(\bx,\by)|^2\mathd \mu_\by\mathd\mu_\bx\nonumber
\end{eqnarray}
and
\begin{eqnarray}
 \int_\M \int_\M R_t(\bx,\by)|d(\bx,\by)|^2\mathd \mu_\by\mathd\mu_\bx&=&
 \sum_{i=1}^N \int_\M\int_{\mathcal{O}_i}R_t(\bx,\by)|d(\bx,\by)|^2\mathd \mu_\by\mathd\mu_\bx\nonumber\\
%&=&\sum_{i=1}^N \int_{B_{\bq_i}^{2\delta}}\int_{\mathcal{O}_i}R_t(\bx,\by)|d(\bx,\by)|^2\mathd \mu_\by\mathd\mu_\bx\nonumber\\
&=&\sum_{i=1}^N \int_{\mathcal{O}_i}\left(\int_{\M_{\by}^t}R_t(\bx,\by)|d(\bx,\by)|^2\mathd \mu_\bx\right)\mathd\mu_\by.\nonumber
\end{eqnarray}
Using Newton-Leibniz formula, we get
\begin{eqnarray}
  d(\bx,\by)&=&u(\bx)-u(\by)-(\bx-\by)\cdot \nabla u(\by)-\frac{1}{2}\eta^i\eta^j(\nabla^i\nabla^j u(\by))\nonumber\\
&=&\xi^{i}\xi^{i'}\int_0^1\int_0^1\int_0^1s_1\frac{d}{d s_3}\left(\p_{i}\Phi^j(\al+s_3 s_1\xi)\p_{i'}\Phi^{j'}(\al+s_3s_2 s_1\xi)\nabla^{j'}\nabla^ju(\Phi(\al+s_3s_2s_1 \xi))\right)
\mathd s_3\mathd s_2\mathd s_1\nonumber\\
&=&\xi^{i}\xi^{i'}\xi^{i''}\int_0^1\int_0^1\int_0^1s_1^2s_2\p_{i}\Phi^j(\al+s_3 s_1\xi)\p_{i''}\p_{i'}\Phi^{j'}(\al+s_3s_2 s_1\xi)\nabla^{j'}\nabla^ju(\Phi(\al+s_3s_2s_1 \xi))
\mathd s_3\mathd s_2\mathd s_1\nonumber\\
&&+\xi^{i}\xi^{i'}\xi^{i''}\int_0^1\int_0^1\int_0^1s_1^2\p_{i''}\p_{i}\Phi^j(\al+s_3 s_1\xi)\p_{i'}\Phi^{j'}(\al+s_3s_2 s_1\xi)\nabla^{j'}\nabla^ju(\Phi(\al+s_3s_2s_1 \xi))
\mathd s_3\mathd s_2\mathd s_1\nonumber\\
&&+\xi^{i}\xi^{i'}\xi^{i''}\int_0^1\int_0^1\int_0^1s_1^2s_2\p_{i}\Phi^j(\al+s_3s_2 s_1\xi)\p_{i'}\Phi^{j'}(\al+s_3s_2 s_1\xi)\p_{i''}\Phi^{j''}(\al+s_3s_2 s_1\xi)\nonumber\\
&&\hspace{4cm}\nabla^{j''}\nabla^{j'}\nabla^ju(\Phi(\al+s_3s_2s_1 \xi))\mathd s_3\mathd s_2\mathd s_1\nonumber
\end{eqnarray}
Here, $\alpha=\alpha(\bx,\by)=\Phi_{J(\by)}^{-1}(\by)$, $\xi=\xi(\bx,\by)=\Phi_{J(\by)}^{-1}(\bx)-\Phi_{J(\by)}^{-1}(\by)$.
%In the rest of the proof, without introducing any confusion, we always to use these short notations to save the space. 
In above derivation, we need the convexity property of the parameterization function to make sure all the integrals are well defined.

Using above equality and the smoothness of the parameterization functions, it is easy to show that
\begin{eqnarray}
&&  \int_{\mathcal{O}_i}\left(\int_{\M_{\by}^t}R_t(\bx,\by)|d(\bx,\by)|^2\mathd \mu_\bx\right)\mathd\mu_\by\nonumber\\
&\le & Ct^3 \int_0^1\int_0^1\int_0^1\int_{\mathcal{O}_i}\int_{\M_{\by}^t}R_t(\bx,\by)\left|D^{2,3}u(\Phi_{J(\by)}(\al+s_3s_2s_1 \xi))\right|^2\mathd \mu_\bx\mathd\mu_\by
\mathd s_3\mathd s_2\mathd s_1\nonumber\\
&\le & Ct^3 \max_{0\le s\le 1}\int_{\mathcal{O}_i}\int_{\M_{\by}^t}R_t(\bx,\by)\left|D^{2,3}u(\Phi_{i}(\al+s \xi))\right|^2\mathd \mu_\bx\mathd\mu_\by,\nonumber
\end{eqnarray}
where we use the fact that $J(\by)=i,\; \by \in \mathcal{O}_i$ and 
\begin{eqnarray*}
  \left|D^{2,3}u(\bx)\right|^2=\sum_{j,j',j''=1}^d|\nabla^{j''}\nabla^{j'}\nabla^ju(\bx)|^2
+\sum_{j,j'=1}^d|\nabla^{j'}\nabla^ju(\bx)|^2.
\end{eqnarray*}
% \begin{eqnarray}
%  && \int_{B_{\bq_i}^{r}}\int_{\M_{\by}^t}R_t(\bx,\by)\left|D^3u(\Phi_i(\al+s \xi))\right|^2\mathd \mu_\bx\mathd\mu_\by\nonumber\\
% &=& \int_{B_{\bq_i}^{r}}\int_{\M_{\by}^t}R_t(\bx,\by)\left|D^3u(\Phi_i(\al+s_3s_2s_1 \xi))\right|^2\mathd \mu_\bx\mathd\mu_\by
% \end{eqnarray}

Let $\bz_i=\Phi_i(\al+s \xi),\; 0\le s\le 1$, then for any $\by\in \mathcal{O}_i\subset B_{\bq_i}^\delta$ and $\bx\in \M_{\by}^t$,
\begin{eqnarray*}
  |\bz_i-\by|\le 2s|\xi|\le 4s|\bx-\by|\le 8s\sqrt{t},\quad |\bz_i-\bq_i|\le |\bz_i-\by|+|\by-\bq_i|\le \delta+8s\sqrt{t}.
\end{eqnarray*}
We can assume that $t$ is small enough such that $8\sqrt{t}\le \delta$, then we have
\begin{eqnarray*}
  \bz_i\in B_{\bq_i}^{2\delta}.
\end{eqnarray*}
After changing of variable, we obtain
\begin{eqnarray}
&&  \int_{\mathcal{O}_i}\int_{\M_{\by}^t}R_t(\bx,\by)\left|D^{2,3}u(\Phi_i(\al+s \xi))\right|^2\mathd \mu_\bx\mathd\mu_\by\nonumber\\
&\le & \frac{C}{\delta_0} \int_{\mathcal{O}_i}\int_{B_{\bq_i}^{2\delta}}\frac{1}{s^k}R\left(\frac{|\bz_i-\by|^2}{128s^2t}\right)
\left|D^{2,3}u(\bz_i)\right|^2\mathd \mu_{\bz_i}\mathd\mu_\by\nonumber\\
&=&\frac{C}{\delta_0} \int_{\mathcal{O}_i}\frac{1}{s^k}R\left(\frac{|\bz_i-\by|^2}{128s^2t}\right)\mathd\mu_\by
\int_{B_{\bq_i}^{2\delta}}\left|D^{2,3}u(\bz_i)\right|^2\mathd \mu_{\bz_i}\nonumber\\
&\le & C \int_{B_{\bq_i}^{2\delta}}\left|D^{2,3}u(\bx)\right|^2\mathd \mu_{\bx}.\nonumber
\end{eqnarray}
This estimate would give us that
\begin{eqnarray}
\label{est:r1}
  \|r_1(\bx)\|_{L^2(\M)}\le Ct^{1/2}\|u\|_{H^3(\M)}
\end{eqnarray}
Now, we turn to estimate the gradient of $r_1$.
\begin{eqnarray*}
  \int_\M |\nabla_\bx r_1(\bx)|^2\mathd\mu_\bx&\le &  C\int_\M \left|\int_\M \nabla_\bx R_t(\bx,\by)d(\bx,\by) \mathd \mu_\by\right|^2\mathd\mu_\bx\\
&&+C\int_\M \left|\int_\M  R_t(\bx,\by)\nabla_\bx d(\bx,\by) \mathd \mu_\by\right|^2\mathd\mu_\bx.
\end{eqnarray*}
where $\nabla_\bx$ is the gradient in $\M$ with respect to $\bx$.

Using the same techniques in the calculation of $\|r_1(\bx)\|_{L^2(\M)}$, we get that
the first term of right hand side can bounded as follows
\begin{eqnarray}
  \int_\M \left|\int_\M \nabla_\bx R_t(\bx,\by)d(\bx,\by) \mathd \mu_\by\right|^2\mathd\mu_\bx\le C  \|u\|_{H^3(\M)}^2.\nonumber
\end{eqnarray}
The estimation of second term is a little involved. First, we have
\begin{eqnarray}
\int_\M \left|\int_\M  R_t(\bx,\by)\nabla_\bx d(\bx,\by) \mathd \mu_\by\right|^2\mathd\mu_\bx
%&\le & (\max_\bx p(\bx))^2\int_\M \left(\int_\M R_t(\bx,\by)\mathd \mu_\by\right)\left(\int_\M R_t(\bx,\by)|\nabla_\bx d(\bx,\by)|^2\mathd \mu_\by\right)\mathd\mu_\bx\nonumber\\
&\le & C\int_\M \left(\int_\M R_t(\bx,\by)|\nabla_\bx d(\bx,\by)|^2\mathd \mu_\by\right)\mathd\mu_\bx\nonumber\\
&=& C\sum_{i=1}^N \int_{\mathcal{O}_i}\left(\int_{\M_{\by}^t}R_t(\bx,\by)|\nabla_\bx d(\bx,\by)|^2\mathd \mu_\bx\right)\mathd\mu_\by.\nonumber
\end{eqnarray}
Also using Newton-Leibniz formula, we have
\begin{eqnarray}
  d(\bx,\by)
&=&\xi^{i}\xi^{i'}\int_0^1\int_0^1s_1\left(\p_{i}\Phi^j(\al+ s_1\xi)\p_{i'}\Phi^{j'}(\al+s_2 s_1\xi)\nabla^{j'}\nabla^ju(\Phi(\al+s_2s_1 \xi))\right)
\mathd s_2\mathd s_1\nonumber\\
&&-\xi^{i}\xi^{i'}\int_0^1\int_0^1s_1\left(\p_{i}\Phi^j(\al)\p_{i'}\Phi^{j'}(\al)\nabla^{j'}\nabla^ju(\Phi(\al))\right)
\mathd s_2\mathd s_1\nonumber
\end{eqnarray}
Then the gradient of $d(\bx,\by)$ has following representation,
\begin{eqnarray}
\nabla_\bx  d(\bx,\by)
&=&\xi^{i}\xi^{i'}\nabla_\bx\left(\int_0^1\int_0^1s_1\left(\p_{i}\Phi^j(\al+ s_1\xi)\p_{i'}\Phi^{j'}(\al+s_2 s_1\xi)\nabla^{j'}\nabla^ju(\Phi(\al+s_2s_1 \xi))\right)
\mathd s_2\mathd s_1\right)\nonumber\\
&&\hspace{-20mm}+\nabla_\bx\left(\xi^{i}\xi^{i'}\right)\int_0^1\int_0^1\int_0^1s_1\frac{d}{d s_3}\left(\p_{i}\Phi^j(\al+ s_3s_1\xi)\p_{i'}\Phi^{j'}(\al+s_3s_2 s_1\xi)
\nabla^{j'}\nabla^ju(\Phi(\al+s_3s_2s_1 \xi))\right)
\mathd s_3\mathd s_2\mathd s_1\nonumber\\
&=&d_1(\bx,\by)+d_2(\bx,\by).\nonumber
\end{eqnarray}
For $d_1$, we have
\begin{eqnarray}
  \int_{\mathcal{O}_i}\left(\int_{\M_{\by}^t}R_t(\bx,\by)|d_1(\bx,\by)|^2\mathd \mu_\bx\right)\mathd\mu_\by
%&\le& Ct^2\int_0^1\int_0^1 \int_{B_{\bq_i}^{\delta}}\left(\int_{\M_{\by}^t}R_t(\bx,\by)|D^{2,3}u(\Phi(\al+s_2s_1 \xi))|^2\mathd \mu_\bx\right)\mathd\mu_\by \mathd s_2\mathd s_1\nonumber\\
&\le& Ct^2\max_{0\le s\le 1} \int_{\mathcal{O}_i}\left(\int_{\M_{\by}^t}R_t(\bx,\by)|D^{2,3}u(\Phi_i(\al+s \xi))|^2\mathd \mu_\bx\right)\mathd\mu_\by,\nonumber
\end{eqnarray}
which means that
\begin{eqnarray}
\label{eqn:est-dr1-d1}
   \int_{\mathcal{O}_i}\left(\int_{\M_{\by}^t}R_t(\bx,\by)|d_1(\bx,\by)|^2\mathd \mu_\bx\right)\mathd\mu_\by\le C  \int_{B_{\bq_i}^{2\delta}}|D^{2,3}u(\bx)|^2\mathd \mu_\bx
\end{eqnarray}
For $d_2$, we have
\begin{eqnarray}
&&  d_2(\bx,\by)\nonumber\\
  &=&\nabla_\bx\left(\xi^{i}\xi^{i'}\right)\int_{[0,1]^3}s_1\frac{d}{d s_3}\left(\p_{i}\Phi^j(\al+ s_3s_1\xi)\p_{i'}\Phi^{j'}(\al+s_3s_2 s_1\xi)
\nabla^{j'}\nabla^ju(\Phi(\al+s_3s_2s_1 \xi))\right)
\mathd s_3\mathd s_2\mathd s_1\nonumber\\
&=&\nabla_\bx\left(\xi^{i}\xi^{i'}\right)\xi^{i''}\int_{[0,1]^3}s_1^2s_2\p_{i}\Phi^j(\al+s_3 s_1\xi)\p_{i''}\p_{i'}\Phi^{j'}(\al+s_3s_2 s_1\xi)\nabla^{j'}\nabla^ju(\Phi(\al+s_3s_2s_1 \xi))
\mathd s_3\mathd s_2\mathd s_1\nonumber\\
&&+\nabla_\bx\left(\xi^{i}\xi^{i'}\right)\xi^{i''}\int_{[0,1]^3}s_1^2\p_{i''}\p_{i}\Phi^j(\al+s_3 s_1\xi)\p_{i'}\Phi^{j'}(\al+s_3s_2 s_1\xi)\nabla^{j'}\nabla^ju(\Phi(\al+s_3s_2s_1 \xi))
\mathd s_3\mathd s_2\mathd s_1\nonumber\\
&&+\nabla_\bx\left(\xi^{i}\xi^{i'}\right)\xi^{i''}\int_{[0,1]^3}s_1^2s_2\p_{i}\Phi^j(\al+s_2 s_1\xi)\p_{i'}\Phi^{j'}(\al+s_3s_2 s_1\xi)\p_{i''}\Phi^{j''}(\al+s_3s_2 s_1\xi)\nonumber\\
&&\hspace{4cm}\nabla^{j''}\nabla^{j'}\nabla^ju(\Phi(\al+s_3s_2s_1 \xi))\mathd s_3\mathd s_2\mathd s_1\nonumber
\end{eqnarray}
This formula tells us that
\begin{eqnarray}
 \int_{\mathcal{O}_i}\left(\int_{\M_{\by}^t}R_t(\bx,\by)|d_2(\bx,\by)|^2\mathd \mu_\bx\right)\mathd\mu_\by
% &\le& Ct^2\int_0^1\int_0^1\int_0^1 \int_{B_{\bq_i}^{\delta}}\left(\int_{\M_{\by}^t}R_t(\bx,\by)|D^{2,3}u(\Phi(\al+s_2s_1 \xi))|^2\mathd \mu_\bx\right)\mathd\mu_\by \mathd s_3\mathd s_2\mathd s_1\nonumber\\
&\le& Ct^2\max_{0\le s\le 1} \int_{\mathcal{O}_i}\left(\int_{\M_{\by}^t}R_t(\bx,\by)|D^{2,3}u(\Phi(\al+s \xi))|^2\mathd \mu_\bx\right)\mathd\mu_\by.\nonumber
\end{eqnarray}
Using the same arguments as that in the calculation of $\|r_1\|_{L^2(\M)}$, we have
\begin{eqnarray}
\label{eqn:est-dr1-d2}
   \int_{\mathcal{O}_i}\left(\int_{\M_{\by}^t}R_t(\bx,\by)|d_2(\bx,\by)|^2\mathd \mu_\bx\right)\mathd\mu_\by\le C  \int_{B_{\bq_i}^{2\delta}}|D^3u(\bx)|^2\mathd \mu_\bx
\end{eqnarray}
Combining \eqref{eqn:est-dr1-d1} and \eqref{eqn:est-dr1-d2}, we have
\begin{eqnarray}
\label{est:dr1}
  \|\nabla r_1(\bx)\|_{L^2(\M)}\le C\|u\|_{H^3(\M)}
\end{eqnarray}
% \begin{eqnarray}
% \label{est:r1}
%   \|r_1(\bx)\|_{L^2(\M)}\le Ct^{1/2}\|u\|_{H^3(\M)}
% \end{eqnarray}
% \begin{eqnarray}
% \label{est:dr1}
%   \|\nabla r_1(\bx)\|_{L^2(\M)}\le C\|u\|_{H^3(\M)}
% \end{eqnarray}
% The estimates of $r_2$, $r_3$ and $r_4$ are similar as those in our previous paper \cite{SS-rate}. In order to
% make this proof self-consistent, we also give a complete proof of this part.

For $r_2$, first, notice that
\begin{eqnarray}
  \nabla^j\bar{R}_t(\bx,\by)&=&\frac{1}{2t}\p_{m'}\Phi^j(\al) g^{m'n'}\p_{n'}\Phi^i(\al) (x^i-y^i)\hk,\nonumber\\
\frac{\eta^j}{2t}R_t(\bx,\by)&=&\frac{1}{2t}\p_{m'}\Phi^j(\al) g^{m'n'}\p_{n'}\Phi^i(\al) \xi^{i'}\p_{i'}\Phi^i\hk.\nonumber
\end{eqnarray}
Then, we have
\begin{eqnarray}
  &&\nabla^j\bar{R}_t(\bx,\by)-\frac{\eta^j}{2t}R_t(\bx,\by)\nonumber\\
%&=&\frac{1}{2t}\p_{m'}\Phi^i g^{m'n'}\p_{n'}\Phi^j (x^j-y^j)-\eta^j\hk\nonumber\\
&=&\frac{1}{2t}\p_{m'}\Phi^i g^{m'n'}\p_{n'}\Phi^j \left(x^j-y^j-\xi^{i'}\p_{i'}\Phi^j\right)\hk\nonumber\\
&=&\frac{1}{2t}\xi^{i'}\xi^{j'}\p_{m'}\Phi^i g^{m'n'}\p_{n'}\Phi^j \left(\int_0^1\int_0^1s\p_{j'}\p_{i'}\Phi^j(\al+\tau s \xi)\mathd \tau\mathd s\right)\hk\nonumber
\end{eqnarray}
Thus, we get
\begin{eqnarray}
 \left|\nabla^j\bar{R}_t(\bx,\by)-\frac{\eta^j}{2t}R_t(\bx,\by)\right|&\le& \frac{C|\xi|^2}{t}\hk\nonumber\\
\left|\nabla_\bx\left(\nabla^j\bar{R}_t(\bx,\by)-\frac{\eta^j}{2t}R_t(\bx,\by)\right)\right|&\le& \frac{C|\xi|}{t}\hk+\frac{C|\xi|^3}{t^2}|R'_t(\bx,\by)|\nonumber
\end{eqnarray}
Then, we have following bound for $r_2$,
\begin{align}
\label{est:r2}
 & \int_{\M}|r_2(\bx)|^2\mathd \mu_\bx\\
\le& Ct\int_\M  \left(\int_\M \hk |D^2u(\by)| \mathd \mu_\by\right)^2\mathd\mu_\bx\nonumber\\
\le &  Ct\int_\M  \left(\int_\M \hk   \mathd\mu_\by\right)\int_\M \hk |D^2u(\by)|^2 \mathd \mu_\by\mathd\mu_\bx \nonumber\\
\le& Ct  \max_{\by}\left(\int_\M \hk\mathd \mu_\bx\right) \int_\M|D^2u(\by)|^2 \mathd \mu_\by \nonumber\\
\le& Ct\|u\|_{H^2(\M)}^2.\nonumber
\end{align}
Similarly, we have
\begin{align}
\label{est:dr2}
 &\int_M |\nabla r_2(\bx)|^2\mathd\mu_\bx\\
\le &  Ct\int_\M  \left(\int_\M\nabla_\bx \hk  \mathd\mu_\by\right)\int_\M \nabla_\bx\hk
|D^2u(\by)|^2 \mathd \mu_\by\mathd\mu_\bx\nonumber\\
 \le& C\sqrt{t} 
\max_{\by}\left(\int_\M \nabla_\bx\hk\mathd \mu_\bx\right) \int_\M|D^2u(\by)|^2 \mathd \mu_\by\nonumber\\
\le&  C\|u\|_{H^2(\M)}^2. \nonumber
\end{align}
$r_3$ is relatively easy to estimate by using the well known Gauss formula.
\begin{align*}
  r_3(\bx)&=\int_{\p\mathcal{M}}n^{j}\eta^{i}(\nabla^{i} \nabla^{j}u(\by)) \rhk   \mathd\tau_\by
-\int_{\mathcal{M}}\eta^{i}(\nabla^{i} \nabla^{j}u(\by)) \rhk \nabla^{j} \mathd\mu_\by\\
&=\tilde{I}_{bd}-\int_{\mathcal{M}}\eta^{i}(\nabla^{i} \nabla^{j}u(\by)) \rhk \nabla^{j} \mathd\mu_\by
\end{align*}
where $\tilde{I}_{bd}=\int_{\p\mathcal{M}}n^{j}\eta^{i}(\nabla^{i} \nabla^{j}u(\by)) \rhk   \mathd\tau_\by$.

Using the assumption that $p\in C^1(\M)$, it is easy to get that 
\begin{align}
  \label{est:r3}
  \|r_3-\tilde{I}_{bd}\|_{L^2(\M)}&\le C\sqrt{t}\|u\|_{H^2(\M)},\\
\label{est:dr3}
\|\nabla(r_3-\tilde{I}_{bd})\|_{L^2(\M)}&\le C\|u\|_{H^2(\M)}.
\end{align}
%\label{est:r3}

% Then, by direct calculation, we have
% \begin{eqnarray}
%   \label{est:r3}
%   \|r_3\|_{L^2(\M)}^2&=&\int_\M\left(\int_{\p\mathcal{M}}n^{j}\eta^{i}(\nabla^{i} \nabla^{j}u) \rhk \mathd\tau_\by\right)^2\mathd \mu_\bx\nonumber\\
% &\le &C t\int_\M\left(\int_{\p\mathcal{M}} \rhk |D^2u(\by)| \mathd\tau_\by\right)^2\mathd \mu_\bx\nonumber\\
% &\le &C t \int_\M\left(\int_{\p\mathcal{M}} \rhk \mathd\tau_\by\right)\left(\int_{\p\mathcal{M}} \rhk |D^2u(\by)|^2\mathd\tau_\by\right)\mathd \mu_\bx\nonumber\\
% &\le &C t^{1/2} \int_{\p\M}|D^2u(\by)|^2\left(\int_{\mathcal{M}} \rhk  \mathd\mu_\bx\right)\mathd \tau_\by\nonumber\\
% &\le &C t^{1/2}\|u\|^2_{H^2(\p\M)}\le C t^{1/2}\|u\|^2_{H^3(\M)}.
% \end{eqnarray}
% and
% \begin{eqnarray}
% \label{est:dr3}
%   \|\nabla r_3(\bx)\|_{L^2(\M)}^2&=& \int_\M\left|\nabla\left(\int_{\p\mathcal{M}}n^{j}\eta^{i}(\nabla^{i} \nabla^{j}u) \rhk \mathd\tau_\by\right)\right|^2\mathd \mu_\bx\nonumber\\
% &= &\int_\M\left|\int_{\p\mathcal{M}}n^{j}(\nabla^{i} \nabla^{j}u)\left(\left(\nabla\eta^{i}\right) \rhk+\eta^i \nabla \rhk\right) \mathd\tau_\by\right|^2\mathd \mu_\bx\nonumber\\
% % &\le &C \|u\|^2_{C^2(\M)}\left(\int_\M\left|\int_{\p\mathcal{M}}\rhk \mathd\tau_\by\right|^2\mathd \mu_\bx+ \int_\M\left|\int_{\p\mathcal{M}}\frac{|\xi|^2}{t} \hk
% %  \mathd\tau_\by\right|^2\mathd \mu_\bx\right)\nonumber\\
% &\le & C t^{-1/2}\|u\|^2_{H^3(\M)}
% \end{eqnarray}
Now, we turn to bound the last term $r_4$. Notice that
\begin{eqnarray}
\label{eqn:grad2Delta} 
 \nabla^j\left(\nabla^{j}u(\by)\right)&=&(\p_{k'} \Phi^j)g^{k'l'}\p_{l'}\left((\p_{m'}\Phi^j)g^{m'n'}(\p_{n'} u )\right)\nonumber\\
&=&(\p_{k'}\Phi^j)g^{k'l'}\left(\p_{l'}(\p_{m'}\Phi^j)\right)g^{m'n'}(\p_{n'}u)\nonumber\\
&&+(\p_{k'}\Phi^j)g^{k'l'}(\p_{m'}\Phi^j)\p_{l'}\left(g^{m'n'}(\p_{n'}u)\right)\nonumber\\
&=&\frac{1}{\sqrt{\det G}}(\p_{m'}\sqrt{\det G}) g^{m'n'}(\p_{n'}u)+\p_{m'}\left(g^{m'n'}(\p_{n'}u)\right)\nonumber\\
&=&\frac{1}{\sqrt{\det G}}\p_{m'}\left(\sqrt{\det G} g^{m'n'}(\p_{n'}u)\right) =\Delta_\M u(\by).
\end{eqnarray}
where $\det G$ is the determinant of $G$ and $G=(g_{ij})_{i,j=1,\cdots,k}$.
Here we use the fact that
\begin{eqnarray}
  (\p_{k'}\Phi^j)g^{k'l'}\left(\p_{l'}(\p_{m'}\Phi^j)\right)&=&(\p_{k'}\Phi^j)g^{k'l'}\left(\p_{m'}(\p_{l'}\Phi^j)\right)\nonumber\\
&=&(\p_{m'}(\p_{k'}\Phi^j))g^{k'l'}(\p_{l'}\Phi^j)\nonumber\\
&=&\frac{1}{2}g^{k'l'}\p_{m'}(g_{k'l'})\nonumber \\
&=&\frac{1}{\sqrt{\det G}}(\p_{m'}\sqrt{\det G}).\nonumber
\end{eqnarray}
Moreover, we have
\begin{eqnarray}
\label{eqn:div}
&& g^{i'j'}(\p_{j'}\Phi^{j})(\p_{i'}\xi^{l})(\p_{l}\Phi^{i})(\nabla^{i}\nabla^{j}u(\by))\\
&=&- g^{i'j'}(\p_{j'}\Phi^{j})(\p_{i'}\Phi^{i})(\nabla^{i}\nabla^{j}u(\by))\nonumber\\
&=&- g^{i'j'}(\p_{j'}\Phi^{j})(\p_{i'}\Phi^{i})(\p_{m'}\Phi^{i})g^{m'n'}\p_{n'}\left(\nabla^{j}u(\by) \right)\nonumber\\
&=&- g^{i'j'}(\p_{j'}\Phi^{j})\p_{i'}\left(\nabla^{j}u(\by)\right)\nonumber\\
&=&- \nabla^{j}\left(\nabla^{j}u(\by)\right).\nonumber
\end{eqnarray}
where the first equalities are due to that
$\p_{i'}\xi^l = -\delta_{i'}^l$.
Then we have
\begin{eqnarray}
  &&\mbox{div} \; \left(\eta^i(\nabla^i\nabla^j u(\by))\right)+\Delta_\M u(\by)\nonumber\\
&=&\frac{1}{\sqrt{\det G}}\,\p_{i'}\left(\sqrt{\det G}\,g^{i'j'}(\p_{j'}\Phi^j) \xi^{l}(\p_l \Phi^i)(\nabla^i\nabla^j u(\by))\right)
-g^{i'j'}(\p_{j'}\Phi^{j})(\p_{i'}\xi^{l})(\p_{l}\Phi^{i})(\nabla^{i}\nabla^{j}u(\by))\nonumber\\
&=&\frac{\xi^l}{\sqrt{\det G}}\,\p_{i'}\left(\sqrt{\det G}\,g^{i'j'}(\p_{j'}\Phi^j) (\p_l \Phi^i)(\nabla^i\nabla^j u(\by))\right).\nonumber
\end{eqnarray}
Here we use the equalities \eqref{eqn:grad2Delta}, \eqref{eqn:div}, $\eta^i = \xi^{l} \p_{i'}\Phi^l$ and the definition of $\mbox{div}$,
\begin{equation}
\mbox{div}X = \frac{1}{\sqrt{\det G}}\p_{i'} (\sqrt{\det G}\,g^{i'j'}\p_{j'}\Phi^k X^k).
\label{eqn:divergence}
\end{equation}
where $X$ is a smooth tangent vector
field on $\M$ and $(X^1,\dots,X^d)^t$ is its representation in embedding coordinates.

Hence,
\begin{eqnarray}
  r_4(\bx)=\int_\mathcal{M}\frac{\xi^l}{\sqrt{\det G}}\,\p_{i'}\left(\sqrt{\det G}\,g^{i'j'}
(\p_{j'}\Phi^j) (\p_l \Phi^i)(\nabla^i\nabla^j u(\by))\right) \rhk  \mathd \mu_\by\nonumber
\end{eqnarray}
Then it is easy to get that
\begin{eqnarray}
  \label{est:r4}
  \|r_4(\bx)\|_{L^2(\M)}&\le& C t^{1/2}\|u\|_{H^3(\M)},\\
\|\nabla r_4(\bx)\|_{L^2(\M)}&\le & C\|u\|_{H^3(\M)}.
\label{est:dr4}
\end{eqnarray}
By combining \eqref{est:r1},\eqref{est:dr1},\eqref{est:r2},\eqref{est:dr2},\eqref{est:r3},\eqref{est:dr3},\eqref{est:r4},\eqref{est:dr4}, we know that
\begin{eqnarray}
  \label{est:r0}
  \|r-\tilde{I}_{bd}\|_{L^2(\M)}&\le& C t^{1/2}\|u\|_{H^3(\M)},\\
\|\nabla (r-\tilde{I}_{bd})\|_{L^2(\M)}&\le & C\|u\|_{H^3(\M)}.
\label{est:dr0}
\end{eqnarray}
Using the definition of $I_{bd}$ and $\tilde{I}_{bd}$, we obtain
\begin{equation*}
  I_{bd}-\tilde{I}_{bd}=\int_{\p\mathcal{M}}n^{j}(\by)(\bx-\by-\eta(\bx,\by))\cdot(\nabla \nabla^{j}u(\by)) \rhk   \mathd\tau_\by
\end{equation*}
Using the definition of $\eta(\bx,\by)$, it is easy to check that
\begin{equation*}
  |\bx-\by-\eta(\bx,\by)|=O(|\bx-\by|^2),\quad |\nabla_{\bx}(\bx-\by-\eta(\bx,\by))|=O(|\bx-\by|)
\end{equation*}
which implies that 
\begin{eqnarray}
  \label{est:bd}
  \|I_{bd}-\tilde{I}_{bd}\|_{L^2(\M)}&\le& C t^{3/4}\|u\|_{H^2(\M)},\\
\|\nabla (I_{bd}-\tilde{I}_{bd})\|_{L^2(\M)}&\le & Ct^{1/4}\|u\|_{H^3(\M)}.
\label{est:dbd}
\end{eqnarray}
The theorem is proved by putting \eqref{est:r0}, \eqref{est:dr0}, \eqref{est:bd}, \eqref{est:dbd} together.
\end{proof}

% \begin{theorem}
% If $u\in H^3(\M)$, then
% \begin{eqnarray*}
% \left\|\right\|_{L^2(\M)}\le Ct^{1/2}\|u\|_{H^3(\mathcal{M})},
% \end{eqnarray*}

% \begin{equation}
  
% \end{equation}
% as long as $t\le T_0$.
% \label{thm:integral_error_adaptive}
% \end{theorem}

% \begin{remark}
%   Using above proof, we can also show that the $L_2$ error in the integral approximation \eqref{eq:integral-adaptive} is $O(t^{1/4})$. 
% \end{remark}

%%% Local Variables:
%%% mode: latex
%%% TeX-master: "paper_convergence_neumann_MC"
%%% End:

\section{Error analysis of the discretization (Theorem \ref{thm:dis_error})}
In this section, we estimate the discretization error introduced by approximating the integrals in~\eqref{eq:integral-homo}
that is to prove Theorem \ref{thm:dis_error}.
To simplify the notation, we introduce a intermediate operator defined as follows,

\begin{eqnarray}
 L_{t, h}u(\bx) &=& \invt\sum_{\bfp_j \in P} \hkxpj(u(\bx) - u(\bfp_j))V_j.
 \label{eqn:laplace_dis} 
\end{eqnarray}
If  $u_{t,h}=I_{\mathbf{f}}(\mathbf{u})$ with $\mathbf{u}$
satisfying Equation \eqref{eqn:dis-homo}, 
one can verify that the following equation is satisfied,
\begin{eqnarray}
L_{t, h}u_{t,h}(\bx) &=& \sum_{\bfp_j \in P} \rhkxpj f(\bfp_j)V_j.
\label{eqn:integral_dis_interp}
\end{eqnarray}

We introduce a discrete operator $\mathcal{L}: \mathbb{R}^n\rightarrow \mathbb{R}^n$ where $n=|P|$. For 
any $\bfu=(u_1,\cdots,u_n)^t$, denote
\begin{eqnarray}
\left(\mathcal{L}\mathbf{u}\right)_i &=& \invt\sum_{\bfp_j \in P} R_t(\bfp_i,\bfp_j)(u_i - u_j)V_j.
\label{eqn:laplace_Lh} 
\end{eqnarray}
For this operator, we have the following important theorem. 

\begin{theorem}
Under the assumptions in Assumption \eqref{assumptions}, there exist constants $C>0,\, C_0>0$ independent on $t$ so that for 
any ${\bf u} = (u_1, \cdots, u_n)^t \in \R^d$ with $\sum_{i=1}^n u_iV_i = 0$ and for any sufficient small $t$ 
and $\frac{h(P,\bV,\M)}{\sqrt{t}}$ 
\begin{equation}
\left<{\bf u}, \mathcal{L} {\bf u} \right>_{\bV} \geq C(1-\frac{C_0h(P,\bV,\M)}{\sqrt{t}}) \left<{\bf u}, {\bf u}\right>_{\bV}
\end{equation}
where $\left<{\bf u}, {\bf v}\right>_{\bV} = \sum_{i=1}^n u_iv_iV_i$ for any ${\bf u} = (u_1, \cdots, u_n), 
{\bf v} =(v_1, \cdots, v_n)$. 
\label{thm:elliptic_L}
\end{theorem}
The proof of the above theorem is deferred to Appendix D.

It has an easy corollary which gives a priori estimate of $\mathbf{u}=(u_1, \cdots, u_n)^t$ solving the 
discrete problem~\eqref{eqn:dis-homo}.
\begin{lemma}
Suppose $\mathbf{u}=(u_1, \cdots, u_n)^t$ with $\sum_i u_i V_i = 0$ solves the problem~\eqref{eqn:dis-homo}
and $\bff =(f(\bfp_1), \cdots, f(\bfp_n))^t$ for $f\in C(\mathcal{M})$,
%$\mathbf{u}=(u_i)$ solves the equation
%\begin{eqnarray}
  %-\frac{1}{t}\sum_{j}R_t(\bfp_i,\bfp_j)(u_i-u_j)V_j+
%\sum_{j}R_t(\bfp_i,\bfp_j)g(\bfp_j)A_j=\sum_{j}R_t(\bfp_i,\bfp_j)f(\bfp_j)V_j
%\end{eqnarray}
 there exists a constant $C>0$ such that 
\begin{eqnarray}
  \left(\sum_{i=1}^nu_i^2V_i\right)^{1/2}\le C\|f\|_\infty,
\end{eqnarray}
provided $t$ and $\frac{h(P,\bV,\M)}{\sqrt{t}}$ are small enough.
\label{lem:bound_solution_bfu}
\end{lemma}
\begin{proof}
From Theorem \ref{thm:elliptic_L}, we have
\begin{eqnarray*}
\sum_{i=1}^{n} u_i^2 V_i &\leq&  \sum_{i=1}^{n} \left( \sum_{\bfp_j \in P} \rhkpipj f_jV_j \right) u_iV_i \nonumber \\
&\leq& \left(\sum_{i=1}^{n}  u_i^2V_i\right)^{1/2} \left(\sum_{i=1}^{n} 
\left( \|f\|_\infty \sum_{p_j \in P} \rhkpipj V_j \right)^2    V_i\right)^{1/2}\nonumber \\
&\leq& C\left(\sum_{i=1}^{n}  u_i^2V_i\right)^{1/2} \|f\|_\infty.  
\end{eqnarray*}
This proves the lemma. 
\end{proof}

%\subsection{Main proof of Theorem \ref{thm:dis_error}}
We are now ready to prove Theorem \ref{thm:dis_error}.
\begin{proof} {\it of Theorem \ref{thm:dis_error}} \\ 
Denote
\begin{eqnarray}
u_{t,h}(\bx) = \frac{1}{w_{t,h}(\bx)}\left(\sum_{\bfp_j\in P}R_t(\bx,\bfp_j)u_jV_j
-t\sum_{\bfp_j\in P}\bar{R}_t(\bx,\bfp_j)f_jV_j\right)
\end{eqnarray}
where $\mathbf{u}=(u_1, \cdots, u_n)^t$ with $\sum_{i=1}^n u_i V_i = 0$ solves the problem~\eqref{eqn:dis-homo}, 
$ f_j = f(\bfp_j)$ and $w_{t,h}(\bx)=\sum_{\bfp_j\in P}R_t(\bx,\bfp_j)V_j$.  
For convenience, we set
\begin{eqnarray*}
  a_{t,h}(\bx)&=&\frac{1}{w_{t,h}(\bx)}\sum_{\bfp_j\in P}R_t(\bx,\bfp_j)u_jV_j, \\
%b_{t,h}(\bx)&=&\frac{2t}{w_{t,h}(\bx)}\sum_{\bfs_j\in S}\bar{R}_t(\bx,\bfs_j)g(\bfs_j)A_j,  \quad \text{and}\\
c_{t,h}(\bx)&=&-\frac{t}{w_{t,h}(\bx)}\sum_{\bfp_j\in P}\bar{R}_t(\bx,\bfp_j)f(\bfp_j)V_j, 
\end{eqnarray*}
and thus $u_{t,h} = a_{t, h} +c_{t, h}$.

In the proof, to simplify the notation, we denote $h=h(P,\bV,\M)$ and $n=|P|$.

First we upper bound $\|L_t(u_{t,h}) - L_{t, h}(u_{t, h})\|_{L^2(\M)}$. 
For $c_{t, h}$, we have 
\begin{eqnarray}
&&\left|\left(L_tc_{t, h} - L_{t, h}c_{t, h}\right)(\bx)\right|\nonumber \\
&=& \frac{1}{t} \left|\int_{\mathcal{M}}R_t(\bx,\by)(c_{t,h}(\bx)-c_{t,h}(\by))  \mathd \mu_\by-\sum_{\bfp_j\in P}R_t(\bx,\bfp_j)(c_{t,h}(\bx)-c_{t,h}(\bfp_j))V_j\right|\nonumber\\
&\le &\frac{1}{t} \left|c_{t,h}(\bx)\right|\left|\int_{\mathcal{M}}R_t(\bx,\by) \mathd \mu_\by-\sum_{\bfp_j\in P}R_t(\bx,\bfp_j)V_j\right|\nonumber\\
&&+ \frac{1}{t}\left|\int_{\mathcal{M}}R_t(\bx,\by)c_{t,h}(\by)  \mathd \mu_\by-\sum_{\bfp_j\in P}R_t(\bx,\bfp_j)c_{t,h}(\bfp_j)V_j\right|\nonumber\\
&\le &\frac{Ch}{t^{3/2}}\left|c_{t,h}(\bx)\right|+ \frac{Ch}{t^{3/2}} \|c_{t,h}\|_{C^1(\M)}\nonumber\\
&\le &\frac{Ch}{t^{3/2}}t\|f\|_\infty+ \frac{Ch}{t^{3/2}}(t\|f\|_\infty + t^{1/2}\|f\|_{\infty})\le\frac{Ch}{t}\|f\|_{\infty}.\nonumber
\end{eqnarray}
For $a_{t, h}$, we have 
\begin{eqnarray}
\label{eqn:a_1}
&&  \int_{\mathcal{M}}\left(a_{t,h}(\bx)\right)^2\left|\int_{\mathcal{M}}R_t(\bx,\by) \mathd \mu_\by-\sum_{\bfp_j\in P}R_t(\bx,\bfp_j)V_j\right|^2\mathd\mu_\bx\\
&\le & \frac{Ch^2}{t}\int_{\mathcal{M}}\left(a_{t,h}(\bx)\right)^2\mathd\mu_\bx 
\le  \frac{Ch^2}{t} \int_{\mathcal{M}} \left( \frac{1}{w_{t, h}(\bx)} \sum_{\bfp_j\in P}R_t(\bx,\bfp_j)u_jV_j \right)^2 \mathd\mu_\bx \nonumber \\
&\le & \frac{Ch^2}{t} \int_{\mathcal{M}} \left( \sum_{\bfp_j\in P}R_t(\bx,\bfp_j)u_j^2V_j  \right) \left( \sum_{\bfp_j\in P}R_t(\bx,\bfp_j)V_j  \right) \mathd\mu_\bx \nonumber \\
&\le & \frac{Ch^2}{t} \left( \sum_{j=1}^{n}u_j^2V_j \int_{\mathcal{M}}R_t(\bx,\bfp_j)  \mathd\mu_\bx    \right) \le \frac{Ch^2}{t}\sum_{j=1}^{n}u_j^2V_j. \nonumber
\end{eqnarray}
Let 
\begin{eqnarray}
A &=&  C_t\int_{\mathcal{M}}\frac{1}{w_{t, h}(\by)}R\left(\frac{|\bx-\by|^2}{4t}\right)R\left(\frac{|\bfp_i-\by|^2}{4t}\right) \mathd \mu_\by\nonumber\\
 &-&C_t\sum_{\bfp_j\in P} \frac{1}{w_{t, h}(\bfp_j)}R\left(\frac{|\bx-\bfp_j|^2}{4t}\right)R\left(\frac{|\bfp_i-\bfp_j|^2}{4t}\right)V_j. \nonumber
\end{eqnarray}
We have $|A|<\frac{Ch}{t^{1/2}}$ for some constant $C$ independent of $t$. In addition, notice that
only when $|\bx-\bfp_i|^2\leq 16t $ is $A\neq 0$, which implies 
\begin{eqnarray}
|A| \leq \frac{1}{\delta_0}|A|R\left(\frac{|\bx-\bfp_i|^2}{32t}\right). \nonumber
\end{eqnarray}
Then we have
\begin{eqnarray}
\label{eqn:a_2}
&&\int_{\mathcal{M}}\left|\int_{\mathcal{M}}R_t(\bx,\by)a_{t,h}(\by)  \mathd \mu_\by-\sum_{\bfp_j\in P}R_t(\bx,\bfp_j)a_{t,h}(\bfp_j)V_j\right|^2\mathd\mu_\bx\\
&=& \int_{\mathcal{M}}\left(\sum_{i=1}^{n}C_tu_iV_i A \right)^2\mathd\mu_\bx\le
\frac{Ch^2}{t} \int_{\mathcal{M}}\left(\sum_{i=1}^{n} C_t|u_i|V_i R\left(\frac{|\bx-\bfp_i|^2}{32t}\right)  \right)^2 \mathd\mu_\bx \nonumber\\
&\le &\frac{Ch^2}{t} \int_{\mathcal{M}} \left(\sum_{i=1}^{n} C_t R\left(\frac{|\bx-\bfp_i|^2}{32t}\right) u^2_iV_i\right)  \left(\sum_{\bfp_i\in P} C_t R\left(\frac{|\bx-\bfp_i|^2}{32t} \right)V_i  \right) \mathd\mu_\bx \nonumber\\
&\le &\frac{Ch^2}{t} \sum_{i=1}^{n} \left(\int_{\mathcal{M}} C_t R\left(\frac{|\bx-\bfp_i|^2}{32t}\right)  \mathd\mu_\bx \left(u^2_iV_i\right)\right)  \le
 \frac{Ch^2}{t} \left(\sum_{i=1}^{n}u_i^2V_i\right). \nonumber
\end{eqnarray} 
Combining Equation~\eqref{eqn:a_1},~\eqref{eqn:a_2} and Lemma~\ref{lem:bound_solution_bfu}, 
\begin{eqnarray}
&&\|L_ta_{t, h} - L_{t, h}a_{t, h}\|_{L^2(\M)} \nonumber \\
&=&\left(\int_M \left|\left(L_t(a_{t, h}) - L_{t, h}(a_{t, h})\right)(\bx)\right|^2 \mathd\mu_\bx\right)^{1/2} \nonumber\\
%&=&  \left(\int_{\mathcal{M}}\left|\int_{\mathcal{M}}R_t(\bx,\by)(u_{t,h}^1(\bx)-u_{t,h}^1(\by))  \mathd \mu_\by-\sum_{j}R_t(\bx,\bfp_j)(u_{t,h}^1(\bx)-u_{t,h}^1(\bfp_j))V_j\right|^2\mathd\mu_\bx \right)^{1/2}\nonumber\\
&\le & \frac{1}{t}\left(\int_{\mathcal{M}}\left(a_{t,h}(\bx)\right)^2\left|\int_{\mathcal{M}}R_t(\bx,\by) \mathd \mu_\by-\sum_{\bfp_j\in P}R_t(\bx,\bfp_j)V_j\right|^2\mathd\mu_\bx\right)^{1/2}\nonumber\\
&&+  \frac{1}{t}\left(\int_{\mathcal{M}}\left|\int_{\mathcal{M}}R_t(\bx,\by)a_{t,h}(\by)  \mathd \mu_\by-\sum_{\bfp_j\in P}R_t(\bx,\bfp_j)a_{t,h}(\bfp_j)V_j\right|^2\mathd\mu_\bx
\right)^{1/2} \nonumber \\
&\le & \frac{Ch}{t^{3/2}} \left(\sum_{i=1}^{n}u_i^2V_i\right)^{1/2} \leq  \frac{Ch}{t^{3/2}}\|f\|_\infty.\nonumber
\end{eqnarray}
Assembling the parts together, we have the following upper bound.
\begin{eqnarray}
\label{eqn:elliptic_L_1}
&&\|L_tu_{t, h} - L_{t, h}u_{t, h}\|_{L^2(\M)} \\
&\le& \|L_ta_{t, h} - L_{t, h}a_{t, h}\|_{L^2(\M)}  + \|L_tc_{t, h} - L_{t, h}c_{t, h}\|_{L^2(\M)} \nonumber \\
&\le& \frac{Ch}{t^{3/2}}\|f\|_\infty + \frac{Ch}{t}\|f\|_{\infty} \le \frac{Ch}{t^{3/2}} \|f\|_\infty\nonumber
\end{eqnarray}
At the same time, since $u_t$ respectively $u_{t,h}$ solves equation~\eqref{eq:integral-homo} respectively equation~\eqref{eqn:integral_dis_interp}, 
we have 
\begin{eqnarray}
\label{eqn:elliptic_L_2}
&&\|L_t(u_{t}) - L_{t, h}(u_{t, h})\|_{L^2(\M)}  \\ 
&=&\left(\int_\M \left(\left(L_tu_{t} - L_{t, h}u_{t, h}\right)(\bx)\right)^2     \mathd\mu_\bx\right)^{1/2}  \nonumber \\
% &\leq& 2\left( \int_\M \left(\int_{\p\mathcal{M}}\bar{R}_t(\bx,\by)g(\by) \mathd \tau_\by- \sum_{\bfs_j\in S}\bar{R}_t(\bx,\bfs_j)g(\bfs_j)A_j\right)^2 \mathd \mu_\bx\right)^{1/2}\nonumber\\
&=& \left( \int _\M \left( \int_{\mathcal{M}}\bar{R}_t(\bx,\by)f(\by) - \sum_{\bfp_j\in P}\bar{R}_t(\bx,\bfp_j)f(\bfp_j)V_j\right)^2 \mathd\mu_\bx\right)^{1/2}\nonumber\\ 
&\le&  \frac{Ch}{t^{1/2}}\|f\|_{C^1(\M)}. \nonumber
\end{eqnarray}
The complete $L^2$ estimate follows from Equation~\eqref{eqn:elliptic_L_1} and~\eqref{eqn:elliptic_L_2}. 
 
The estimate of the gradient, $\|\nabla(L_t(u_{t}) - L_{t, h}(u_{t, h}))\|_{L^2(\M)}$, can be obtained similarly. The details can be found in Appendix E.
\end{proof}

%%% Local Variables:
%%% mode: latex
%%% TeX-master: "paper_convergence_neumann"
%%% End:

\section{Stability analysis (Theorem \ref{thm:regularity} and \ref{thm:regularity_boundary})}
\label{sec:stability}
%\subsection{Coercivity of $L_t$}
%\label{sec:coercivity}
To prove Theorem \ref{thm:regularity} and \ref{thm:regularity_boundary}, we need following two theorems regarding the coercivity of the operator $L_t$.
\begin{theorem} For any function $u\in L^2(\mathcal{M})$, 
there exists a constant $C>0$ independent on $t$ and $u$, such that
  \begin{eqnarray}
    \left<u,L_t u\right>_\M \ge C\int_\mathcal{M} |\nabla v|^2
\mathd \mu_\bx
  \end{eqnarray}
where $\left<f, g\right>_{\mathcal{M}} = \int_{\mathcal{M}} f(\bx)g(\bx)\mathd\mu_\bx$ for any $f, g\in L_2(\mathcal{M})$, and
\begin{eqnarray}
v(\bx)=\frac{C_t}{w_t(\bx)}\int_{\mathcal{M}}R\left(\frac{|\bx-\by|^2}{4t}\right) u(\by)\mathd \mu_\by, 
\label{eqn:smooth_v}
\end{eqnarray}
and $w_t(\bx) = C_t\int_{\mathcal{M}}R\left(\frac{|\bx-\by|^2}{4t}\right)\mathd \mu_\by$.
\label{thm:elliptic_v}
\end{theorem}

\begin{theorem}
Assume both $\M$ and $\p\M$ are $C^\infty$. There exists a constant $C>0$ independent on $t$
so that for any function $u\in L_2(\M)$ with $\int_\M u(\bx)\mathd \mu_\bx = 0$ and for any sufficient small $t$
\begin{eqnarray}
\left<u, L_tu \right>_{\mathcal{M}} \geq C\|u\|_{L_2(\M)}^2
\end{eqnarray}
\label{thm:elliptic_L_t}
\end{theorem}

% In this section, we prove Theorem~\ref{thm:elliptic_v} and~\ref{thm:elliptic_L_t}, which show that the 
% integral Laplace operator $L_t$ preserve the coercivity of the Laplace-Beltrami operator but in a slightly
% weaker form.
%\begin{theorem} For any function $u\in L^2(\mathcal{M})$, 
%there exists a constant $C>0$ independent on $t$ and $u$, such that
  %\begin{eqnarray}
    %\left<u,L_t u\right>_\M \ge C\int_\mathcal{M} |\nabla v|^2,
%\mathd \mu_\bx
  %\end{eqnarray}
%where $\left<f, g\right>_{\mathcal{M}} = \int_{\mathcal{M}} f(\bx)g(\bx)\mathd\mu_\bx$ for any $f, g\in L_2(\mathcal{M})$, and
%\begin{eqnarray}
%v(\bx)=\frac{C_t}{w_t(\bx)}\int_{\mathcal{M}}R\left(\frac{|\bx-\by|^2}{4t}\right) u(\by)\mathd \mu_\by, 
%\label{eqn:smooth_v}
%\end{eqnarray}
%and $w_t(\bx) = C_t\int_{\mathcal{M}}R\left(\frac{|\bx-\by|^2}{4t}\right)\mathd \mu_\by$.
%\label{thm:elliptic_v}
%\end{theorem}
%\begin{theorem}
%Assume both $\M$ and $\p\M$ are $C^\infty$. There exists a constant $C>0$ independent on $t$
%so that for any function $u\in L_2(\M)$ with $\int_\M u = 0$ and for any sufficient small $t$
%\begin{eqnarray}
%\left<L_tu, u\right>_{\mathcal{M}} \geq C\|u\|_{L_2(\M)}^2
%\end{eqnarray}
%\label{thm:elliptic_L_t}
%\end{theorem} 
Theorem \ref{thm:regularity} is a direct corollary of following two lemmas.
\begin{lemma} For any function $u\in L^2(\mathcal{M})$, 
there exists a constant $C>0$ independent on $t$ and $u$, such that
  \begin{eqnarray*}
    \frac{C_t}{t}\int_\mathcal{M}\int_{\mathcal{M}}R\left(\frac{|\bx-\by|^2}{32t}\right) (u(\bx)-u(\by))^2\mathd \mu_\bx \mathd \mu_\by \ge C\int_\mathcal{M} |\nabla v|^2
\mathd \mu_\bx
  \end{eqnarray*}
where $v$ is the same as defined in~\eqref{eqn:smooth_v}.
\label{lem:elliptic_v}
\end{lemma}
\begin{lemma} If $t$ is small enough, then for any function $u\in L^2(\mathcal{M})$, 
there exists a constant $C>0$ independent on $t$ and $u$, such that
\begin{eqnarray}
   \int_\mathcal{M}\int_{\mathcal{M}}R\left(\frac{|\bx-\by|^2}{32t}\right) (u(\bx)-u(\by))^2\mathd \mu_\bx \mathd \mu_\by \le C
\int_\mathcal{M}\int_{\mathcal{M}}R\left(\frac{|\bx-\by|^2}{4t}\right) (u(\bx)-u(\by))^2\mathd \mu_\bx \mathd \mu_\by\nonumber.
\end{eqnarray}
\label{lem:bigt2smallt}
\end{lemma}

The proofs of the above two lemmas are put in Appendix B and C. 
Once we have Lemma~\ref{lem:elliptic_v} and Lemma \ref{lem:bigt2smallt}, Theorem \ref{thm:elliptic_v} becomes obvious by noticing that:
\begin{eqnarray*}
\left<u,L_tu\right>_\M&=&\int_\mathcal{M}\int_{\mathcal{M}}R\left(\frac{|\bx-\by|^2}{4t}\right) u(\bx)(u(\bx)-u(\by))\mathd \mu_\by \mathd \mu_\bx\nonumber\\
&=&-\int_\mathcal{M}\int_{\mathcal{M}}R\left(\frac{|\bx-\by|^2}{4t}\right) u(\by)(u(\bx)-u(\by))\mathd \mu_\by \mathd \mu_\bx\nonumber\\
&=&\frac{1}{2}\int_\mathcal{M}\int_{\mathcal{M}}R\left(\frac{|\bx-\by|^2}{4t}\right) (u(\bx)-u(\by))^2\mathd \mu_\by \mathd \mu_\bx.\nonumber
\end{eqnarray*}

Now, we turn to prove Theorem \ref{thm:elliptic_L_t}.
\begin{proof} {\it of Theorem \ref{thm:elliptic_L_t}} \\
%Without loss of generality, we assume that $\bar{u}=0$.
By Theorem~\ref{thm:elliptic_v} and the Poincar\'{e} inequality, there exists a constant $C>0$, such that 
  \begin{eqnarray*}
    \int_\mathcal{M} (v(\bx)-\bar{v})^2\mathd \mu_\bx \le \frac{CC_t}{t}\int_\mathcal{M}\int_{\mathcal{M}} R\left(\frac{|\bx-\by|^2}{4t}\right)  (u(\bx)-u(\by))^2\mathd \mu_\bx \mathd \mu_\by      
  \end{eqnarray*}
where $\bar{v}=\frac{1}{|\mathcal{M}|}\int_\mathcal{M} v(\bx)\mathd \mu_\bx$ and 
\begin{eqnarray*}
v(\bx)=\frac{C_t}{w_t(\bx)}\int_{\mathcal{M}}R\left(\frac{|\bx-\by|^2}{4t}\right) u(\by)\mathd \mu_\by, 
\end{eqnarray*}
At the same time, we have
\begin{eqnarray*}
&&  |\mathcal{M}||\bar{v}|=\left|\int_\mathcal{M} v(\bx)\mathd \mu_\bx\right|\nonumber\\
&=&\left|\int_\mathcal{M}\int_\mathcal{M}\frac{C_t}{w_t(\bx)}R\left(\frac{|\bx-\by|^2}{4t}\right) (u(\by)-u(\bx))\mathd \mu_\by\mathd \mu_\bx\right|\nonumber\\
&\le &\left(\int_\mathcal{M}\int_\mathcal{M}\frac{C_t}{w_t(\bx)} R\left(\frac{|\bx-\by|^2}{4t}\right)  \mathd \mu_\by\mathd \mu_\bx\right)^{1/2} \nonumber\\
&&\left(\int_\mathcal{M}\int_\mathcal{M}\frac{C_t}{w_t(\bx)} R\left(\frac{|\bx-\by|^2}{4t}\right) (u(\by)-u(\bx))^2\mathd \mu_\by\mathd \mu_\bx\right)^{1/2}\nonumber\\
&\le& C|\mathcal{M}|^{1/2}\left(C_t\int_\mathcal{M}\int_\mathcal{M} R\left(\frac{|\bx-\by|^2}{4t}\right) (u(\by)-u(\bx))^2\mathd \mu_\by\mathd \mu_\bx\right)^{1/2} 
% &\le& C\left(\frac{C_t}{\delta_0}\int_\mathcal{M}\int_\mathcal{M} R\left(\frac{|\bx-\by|^2}{4t} \right)  R\left(\frac{|\bx-\by|^2}{32t} \right)(u(\by)-u(\bx))^2\mathd \mu_\by\mathd \mu_\bx\right)^{1/2} \nonumber \\
% &\le& C\left(C_t\int_\mathcal{M}\int_\mathcal{M} R\left(\frac{|\bx-\by|^2}{32t} \right)(u(\by)-u(\bx))^2\mathd \mu_\by\mathd \mu_\bx\right)^{1/2}, \nonumber 
\end{eqnarray*}
where the second equality comes from $\int_{\M} u(\bx)\mathd \mu_\bx = 0$. 
This enables us to upper bound the $L_2$ norm of $v$ as follows. For $t$ sufficiently small, 
\begin{align*}
  \int_\mathcal{M} (v(\bx))^2\mathd \mu_\bx \le& 2\int_M (v(\bx) - \bar{v})^2 \mathd \bx + 2 \int_M \bar{v}^2 \mathd \mu_\bx \nonumber \\ 
   \le& \frac{CC_t}{t}\int_\mathcal{M}\int_{\mathcal{M}} R\left(\frac{|\bx-\by|^2}{4t}\right) (u(\bx)-u(\by))^2\mathd \mu_\bx \mathd \mu_\by
\end{align*}
Let $\delta=\frac{w_{\min}}{2w_{\max}+w_{\min}}$ where $w_{\min}=\min_\bx w_t(\bx)$ and $w_{\max}=\max_\bx w_t(\bx)$. 
If $u$ is smooth and close to its smoothed version $v$, in particular, 
\begin{eqnarray}
  \int_\mathcal{M} |v(\bx)|^2\mathd \mu_\bx\ge \delta^2 \int_\mathcal{M} |u(\bx)|^2\mathd \mu_\bx,
\label{eqn:smooth_u}
\end{eqnarray}
then the theorem is proved. 

Now consider the case where $\eqref{eqn:smooth_u}$ does not hold.  Note that we now have
\begin{align}
  \|u-v\|_{L^2(\M)} \ge& \|u\|_{L_2(\M)}-\|v\|_{L_2(\M)}> (1-\delta)\|u\|_{L_2(\M)}\nonumber\\
  >&\frac{1-\delta}{\delta}\|v\|_{L_2(\M)}=\frac{2w_{\max}}{w_{\min}}\|v\|_{L_2(\M)}.\nonumber 
\end{align}
Then we have  
\begin{eqnarray*}
  &&\frac{C_t}{t}\int_\mathcal{M}\int_{\mathcal{M}} R\left(\frac{|\bx-\by|^2}{4t}\right) (u(\bx)-u(\by))^2\mathd \mu_\bx \mathd \mu_\by\nonumber\\
&=&\frac{2C_t }{t}\int_\mathcal{M} u(\bx)\int_{\mathcal{M}} R\left(\frac{|\bx-\by|^2}{4t}\right) (u(\bx)-u(\by))\mathd \mu_\by \mathd \mu_\bx\nonumber\\
&=&\frac{2 }{t}\left(\int_\mathcal{M} u^2(\bx)w(\bx)\mathd \mu_\bx-\int_{\mathcal{M}}u(\bx)v(\bx)w(\bx)\mathd \mu_\bx\right)\nonumber\\
&=&\frac{2 }{t}\left(\int_\mathcal{M} (u(\bx)-v(\bx))^2w(\bx)\mathd \mu_\bx+
\int_{\mathcal{M}}(u(\bx)-v(\bx))v(\bx)w(\bx)\mathd \mu_\bx\right)\nonumber\\
&\ge & \frac{2 }{t}\int_\mathcal{M} (u(\bx)-v(\bx))^2w(\bx)\mathd \mu_\bx-\frac{2 }{t}
\left(\int_{\mathcal{M}}v^2(\bx)w(\bx)\mathd \mu_\bx\right)^{1/2}
\left(\int_{\mathcal{M}}(u(\bx)-v(\bx))^2w(\bx)\mathd \mu_\bx\right)^{1/2}\nonumber\\
&\ge& \frac{2 w_{\min}}{t}\int_\mathcal{M} (u(\bx)-v(\bx))^2\mathd \mu_\bx-\frac{2 w_{\max}}{t}
\left(\int_{\mathcal{M}}v^2(\bx)\mathd \mu_\bx\right)^{1/2}
\left(\int_{\mathcal{M}}(u(\bx)-v(\bx))^2\mathd \mu_\bx\right)^{1/2}\nonumber\\
&\ge&\frac{ w_{\min}}{t}\int_\mathcal{M} (u(\bx)-v(\bx))^2\mathd \mu_\bx\ge \frac{ w_{\min}}{t}(1-\delta)^2\int_\mathcal{M} u^2(\bx)\mathd \mu_\bx.
\end{eqnarray*} 
% Finally, notice that
% \begin{equation}
%   \int_\mathcal{M}\int_{\mathcal{M}} R\left(\frac{|\bx-\by|^2}{4t}\right) (u(\bx)-u(\by))^2\mathd \mu_\bx \mathd \mu_\by\le 
% \frac{1}{\delta_0}\int_\mathcal{M}\int_{\mathcal{M}} R\left(\frac{|\bx-\by|^2}{32t}\right) (u(\bx)-u(\by))^2\mathd \mu_\bx \mathd \mu_\by.\nonumber
% \end{equation}
% which shows that the elliptic inequality holds for small $t$.
This completes the proof for the theorem.  
\end{proof}

% In this section, we will prove Theorem \ref{thm:regularity} and \ref{thm:regularity_boundary}.
% Both theorems are concerned with the stability of $L_t$.
%, which plays an important role in our proof. 
\subsection{Proof of Theorem \ref{thm:regularity}}
With Theorem \ref{thm:elliptic_v} and \ref{thm:elliptic_L_t}, the proof of
Theorem \ref{thm:regularity} is straightforward.
\begin{proof} {\it of Theorem  \ref{thm:regularity}}

Using Theorem \ref{thm:elliptic_L_t}, we have
  \begin{align}
\label{eqn:stable_Lt_l2}
    \|u\|_{L^2(\M)}^2 \le& C \left<u,L_tu\right> = C\int_\M u(\bx)(r(\bx)-\bar{r})\mathd\mu_\bx \\
	 \le& C\|u\|_{L^2(\M)}\|r\|_{L^2(\M)}.\nonumber
  \end{align}
To show the last inequality, we use the fact that 
$$|\bar{r}|=\frac{1}{|\M|}\left|\int_\M r(\bx)\mathd\mu_\bx\right|\le C\|r\|_{L^2(\M)}.$$
This inequality \eqref{eqn:stable_Lt_l2} implies that
\begin{eqnarray}
  \|u\|_{L^2(\M)}\le C\|r\|_{L^2(\M)}.\nonumber
\end{eqnarray}
Now we turn to estimate $\|\nabla u\|_{L^2(\M)}$. 
Notice that we have the following expression for $u$, since $u$ satisfies the integral equation \eqref{eq:integral-homo}.
\begin{eqnarray}
u(\bx)=v(\bx)+\frac{t}{w_t(\bx)}\,(r(\bx) - \bar{r}),\nonumber
\end{eqnarray}
where
\begin{eqnarray}
  v(\bx)=\frac{1}{w_t(\bx)}\int_{\M}R_t(\bx,\by)u(\by)\mathd\mu_\by,\quad w_t(\bx)=\int_{\M}R_t(\bx,\by)\mathd\mu_\by.\nonumber
\end{eqnarray}
By Theorem \ref{thm:elliptic_v}, we have
\begin{align*}
\|\nabla u\|_{L^2(\M)}^2 \le&   2\|\nabla v\|_{L^2(\M)}^2+ 2t^2\left\|\nabla \left(\frac{r(\bx) - \bar{r}}{w_t(\bx)}\right)\right\|_{L^2(\M)}^2\\
\le & C \left<u,L_tu\right> + Ct\|r\|_{L^2(\M)}^2+Ct^2\|\nabla r\|_{L^2(\M)}^2\nonumber\\
\le& C\|u\|_{L^2(\M)}\|r\|_{L^2(\M)}+ Ct\|r\|_{L^2(\M)}^2+Ct^2\|\nabla r\|_{L^2(\M)}^2\nonumber\\
\le& C\|r\|_{L^2(\M)}^2+Ct^2\|\nabla r\|_{L^2(\M)}^2\nonumber\\
\le & C\left(\|r\|_{L^2(\M)}+t\|\nabla r\|_{L^2(\M)}\right)^2.\nonumber
\end{align*}
This completes the proof. 
\end{proof}

\subsection{Proof of Theorem \ref{thm:regularity_boundary}}

The proof of Theorem \ref{thm:regularity_boundary} is more involved.
\begin{proof}
First, we denote
\begin{align*}
r(\bx)&=\int_{\p\M}\mathbf{b}(\by)\cdot(\bx-\by)\rhk  \mathd \tau_\by,\\
\bar{r}&=\frac{1}{|\M|}\int_\M \left(\int_{\p\M}\mathbf{b}(\by)\cdot(\bx-\by)\rhk  \mathd \tau_\by\right)\mathd\bx.
\end{align*}
where $|\M|=\int_\M  \mathd \mu_\by$.

The key point of the proof is to show that
  \begin{eqnarray}
\label{eq:est_boundary_whole}
    %\left|\int_\M u(\bx)\left(\int_{\p\M}b^i(\by)\eta^j\bar{R}_t(\bx,\by)\mathd \tau_\by-\bar{b}\right)\mathd \mu_\bx\right|
	 %\le C\sqrt{t} \;\max_i\left(\|b^i\|_\infty\right) \|u\|_{H^1(\M)}.
   \left|\int_\M u(\bx)\left(r(\bx)-\bar{r}\right) \mathd \mu_\bx\right|
	\le C\sqrt{t} \;\|\mathbf{b}\|_{H^1(\M)} \|u\|_{H^1(\M)}.
  \end{eqnarray}
%where $\|\mathbf{b}\|_{\infty}=\max_i \max_{\by\in \M}|b^i(\by)|$.

First, notice that
$$|\bar{r}|\le 
C\sqrt{t}\;\|\mathbf{b}\|_{L^2(\p\M)}\le C\sqrt{t}\;\|\mathbf{b}\|_{H^1(\M)}.$$
Then it is sufficient to show that
  \begin{equation}
\label{eq:est_boundary}
    \left|\int_\M u(\bx)\left(\int_{\p\M}\mathbf{b}(\by)\cdot(\bx-\by)\bar{R}_t(\bx,\by)
 \mathd \tau_\by\right) \mathd \mu_\bx\right|\le C\sqrt{t} \;\|\mathbf{b}\|_{H^1(\M)} \|u\|_{H^1(\M)}.
  \end{equation}
Direct calculation gives that
\begin{eqnarray*}
  &&|2t\nabla\rrhk-(\bx-\by)\bar{R}_t(\bx,\by)|\le C|\bx-\by|^2\rhk,
\end{eqnarray*}
where $\rrhk=C_t\bar{\bar{R}}\left(\frac{\|\bx-\by\|^2}{4t}\right)$ and $\bar{\bar{R}}(r)=\int_{r}^{\infty}\bar{R}(s)\mathd s$.
This implies that
\begin{align}
\label{eq:est_boundary_1}
  &\left|\int_\M u(\bx) \int_{\p\M}\mathbf{b}(\by)\left((\bx-\by)\bar{R}_t(\bx,\by)+2t\nabla\rrhk\right)
 \mathd \tau_\by \mathd \mu_\bx\right|\\
\le & C\int_\M |u(\bx) |\int_{\p\M}|\mathbf{b}(\by)||\bx-\by|^2\rhk  \mathd \tau_\by\mathd \mu_\bx\nonumber\\
%\le &Ct\|\mathbf{b}\|_{}\int_\M |u(\bx)| \int_{\p\M}\rhk  \mathd \tau_\by\mathd \mu_\bx\nonumber\\
\le &Ct\|\mathbf{b}\|_{L^2(\p\M)} \left(\int_{\p\M}\left(\int_\M\rhk  \mathd\mu_\bx\right)
\left(\int_\M |u(\bx)|^2\rhk  \mathd \mu_\bx\right) \mathd \tau_\by\right)^{1/2}\nonumber\\
\le & Ct\|\mathbf{b}\|_{H^1(\M)} \left(\int_{\M} |u(\bx)|^2 
\left(\int_{\p\M} \rhk  \mathd \tau_\by\right)\mathd \mu_\bx\right)^{1/2}\nonumber\\
\le & Ct^{3/4}\|\mathbf{b}\|_{H^1(\M)}\|u\|_{L^2(\M)}.\nonumber
\end{align}
On the other hand, using the Gauss integral formula, we have
 \begin{eqnarray}
\label{eq:gauss_boundary}
&&    \int_\M u(\bx) \int_{\p\M}\mathbf{b}(\by)\cdot\nabla\rrhk  \mathd \tau_\by\mathd \mu_\bx\\
&=& \int_{\p\M} \int_{\M}u(\bx) T_\bx(\mathbf{b}(\by))\cdot\nabla\rrhk  \mathd \mu_\bx \mathd \tau_\by\nonumber\\
&=&\int_{\p\M} \int_{\p\M}\mathbf{n}(\bx)\cdot T_\bx(\mathbf{b}(\by))u(\bx)\rrhk   \mathd \tau_\bx\mathd \tau_\by\nonumber\\
&&-\int_{\p\M} \int_{\M}\text{div}_\bx[u(\bx) T_\bx(\mathbf{b}(\by))]\rrhk  \mathd \mu_\bx\mathd \tau_\by.\nonumber
  \end{eqnarray}
Here $T_\bx$ is the projection operator to the tangent space on $\bx$. To get the first equality, we use the 
fact that $\nabla\rrhk$ belongs to the tangent space on $\bx$, such that $\mathbf{b}(\by)\cdot\nabla\rrhk=T_\bx(\mathbf{b}(\by))\cdot\nabla\rrhk$
and $\bn(\bx)\cdot T_\bx(\mathbf{b}(\by))=\bn(\bx)\cdot \mathbf{b}(\by)$ where $\bn(\bx)$ is the out normal of $\p\M$ at $\bx\in \p\M$.

For the first term, we have
  \begin{align}
\label{eq:est_boundary_2}
&    \left|\int_{\p\M} \int_{\p\M}\mathbf{n}(\bx)\cdot T_\bx(\mathbf{b}(\by))u(\bx)\rrhk   \mathd \tau_\bx\mathd \tau_\by\right|\\
= &\left|\int_{\p\M} \int_{\p\M}\mathbf{n}(\bx)\cdot \mathbf{b}(\by)u(\bx)\rrhk   \mathd \tau_\bx\mathd \tau_\by\right|\nonumber\\
\le &C\|\mathbf{b}\|_{L^2(\p\M)}  \left(\int_{\p\M} \left(\int_{\p\M}|u(\bx)|\rrhk  \mathd \tau_\bx\right)^2
 \mathd \tau_\by\right)^{1/2}\nonumber\\
\le &C\|\mathbf{b}\|_{H^1(\M)}  \left(\int_{\p\M} \left(\int_{\p\M}\rrhk  \mathd \tau_\bx\right)
 \left(\int_{\p\M}|u(\bx)|^2\rrhk  \mathd \tau_\bx\right)
 \mathd \tau_\by\right)^{1/2}\nonumber\\
% &\le &Ct^{-1/4}\; \max_i\left(\|b^i\|_\infty\right) \left(\int_{\p\M} 
%  \left(\int_{\p\M}\rrhk\mathd \tau_\by\right)|u(\bx)|^2
% \mathd \tau_\bx\right)^{1/2}\nonumber\\
\le &Ct^{-1/2}\; \|\mathbf{b}\|_{H^1(\M)} \|u\|_{L^2(\p\M)}\le Ct^{-1/2}\; \|\mathbf{b}\|_{H^1(\M)} \|u\|_{H^1(\M)}.\nonumber
  \end{align}
We can also bound the second term on the right hand side of \eqref{eq:gauss_boundary}. By using the assumption that $\M\in C^\infty$, we have
\begin{align*}
&|  \text{div}_\bx[u(\bx) T_\bx(\mathbf{b}(\by))]|\nonumber\\
\le& |\nabla u(\bx)||T_\bx(\mathbf{b}(\by))|| |
+|u(\bx)||\text{div}_\bx[T_\bx(\mathbf{b}(\by))]|| |+|\nabla  ||u(\bx)T_\bx(\mathbf{b}(\by))| \\
\le & C(|\nabla u(\bx)|+|u(\bx)|)|\mathbf{b}(\by)|% \\
% \le& C \|\mathbf{b}\|_{\infty} (|\nabla u(\bx)|+|u(\bx)|)
\end{align*}
where the constant $C$ depends on the curvature of the manifold $\M$.

Then, we have
  \begin{eqnarray}
\label{eq:est_boundary_3}
&&    \left|\int_{\p\M} \int_{\M}\text{div}_\bx[u(\bx)T_\bx(\mathbf{b}(\by))]\rrhk   \mathd \mu_\bx\mathd \tau_\by\right|\\
&\le &C   \int_{\p\M}\mathbf{b}(\by)  \int_{\M}(|\nabla u(\bx)|+|u(\bx)|)\rrhk  \mathd \mu_\bx\mathd \tau_\by\nonumber\\
% &\le &C \max_i\left(\|b^i\|_\infty\right) \left(\int_{\p\M} \left(\int_\M \rrhk \mathd\mu_\bx\right)
% \left(\int_{\M}(|\nabla u(\bx)|^2+|u(\bx)|^2)\rrhk\mathd \mu_\bx\right)\mathd \tau_\by\right)^{1/2}\nonumber\\
&\le &C \|\mathbf{b}\|_{L^2(\p\M)}  \left(\int_{\M} (|\nabla u(\bx)|^2+|u(\bx)|^2) 
\left(\int_{\p\M}\rrhk  \mathd \tau_\by\right)\mathd \mu_\bx\right)^{1/2}\nonumber\\
\quad\quad&\le &C t^{-1/4}\;\|\mathbf{b}\|_{H^1(\M)}  \|u\|_{H^1(\M)}.\nonumber
  \end{eqnarray}
Then, the inequality \eqref{eq:est_boundary} is obtained from \eqref{eq:est_boundary_1},
\eqref{eq:gauss_boundary}, \eqref{eq:est_boundary_2} and \eqref{eq:est_boundary_3}.
Now, using Theorem \ref{thm:elliptic_L_t}, we have
  \begin{eqnarray}
    \|u\|_{L^2(\M)}^2\le C \int_\M u(\bx)L_tu(\bx) \mathd\mu_\bx% =C\left|\int_\M u(\bx) (r(\bx) - \bar{r})\mathd 
	 % \mu_\bx\right|
    \le C\sqrt{t}\;\|\mathbf{b}\|_{H^1(\M)}  \|u\|_{H^1(\M)}.
\label{eq:est_l2_boundary}
  \end{eqnarray}
Note $r(\bx)=\int_{\p\M}(\bx-\by)\cdot \mathbf{b}(\by)\bar{R}_t(\bx,\by) \mathd \tau_\by$. Direct calculation gives us that
\begin{eqnarray*}
  \|r(\bx)\|_{L^2(\M)}&\le& Ct^{1/4}\|\mathbf{b}\|_{H^1(\M)} ,~\text{and}\\
  \|\nabla r(\bx)\|_{L^2(\M)}&\le& Ct^{-1/4}\|\mathbf{b}\|_{H^1(\M)} .
\end{eqnarray*}
The integral equation $L_t u=r-\bar{r}$ gives that
\begin{eqnarray*}
u(\bx)=v(\bx)+\frac{t}{w_t(\bx)}\,(r(\bx)-\bar{r})
\end{eqnarray*}
where
\begin{eqnarray*}
  v(\bx)=\frac{1}{w_t(\bx)}\int_{\M}R_t(\bx,\by)u(\by) \mathd\mu_\by,\quad w_t(\bx)=\int_{\M}R_t(\bx,\by) \mathd\mu_\by.
\end{eqnarray*}
By Theorem \ref{thm:elliptic_v}, we have
\begin{eqnarray}
\label{eq:est_dl2_boundary}
&&\|\nabla u\|_{L^2(\M)}^2 \\
&\le&   2\|\nabla v\|_{L^2(\M)}^2+ 2t^2\left\|\nabla \left(\frac{r(\bx)-\bar{r}}{w_t(\bx)}\right)\right\|_{L^2(\M)}^2\nonumber\\
&\le & C \int_\M u(\bx)L_tu(\bx) \mathd\mu_\bx + Ct\|r\|_{L^2(\M)}^2+Ct^2\|\nabla r\|_{L^2(\M)}^2\nonumber\\
&\le& C\sqrt{t}\;\|\mathbf{b}\|_{H^1(\M)}  \|u\|_{H^1(\M)}+ Ct\|r\|_{L^2(\M)}^2+Ct^2\|\nabla r\|_{L^2(\M)}^2\nonumber\\
&\le& C\|\mathbf{b}\|_{H^1(\M)} \left(\sqrt{t}\|u\|_{H^1(\M)}+ Ct^{3/2}\right).\nonumber
\end{eqnarray}
Using \eqref{eq:est_l2_boundary} and \eqref{eq:est_dl2_boundary}, we have
\begin{eqnarray*}
  \|u\|_{H^1(\M)}^2\le C\|\mathbf{b}\|_{H^1(\M)}  \left(\sqrt{t}\|u\|_{H^1(\M)}+ Ct^{3/2}\right), 
\end{eqnarray*}
which proves the theorem.
\end{proof}

\section{Discussion and Future Work}
\label{sec:discussion}
We have proved the convergence of the point integral method for Poisson equations
on the submanifolds isometrically embedded in Euclidean spaces. 
Our analysis shows that the convergence rate of PIM is $h^{1/4}(P,\bV,\M)$ in $H^1$ norm.
 However, our experimental results in 
%Section \ref{sec:experiment}, as well as more detailed numerical experiments in
\cite{LSS}, show the empirical convergence rate is about linear.
%We believe that PIM indeed converges linearly. 
Indeed, there are places in our analysis where we believe the error bounds can be improved.  

On the other hand, the quadrature rule we used in the point integral method is of low accuracy. 
If we have more information, such as the local mesh or local hypersurface, we could use high order quadrature rule 
to improve the accuracy of the point integral method. 
% In addition, there are also rooms to improve the algorithm to achieve a better convergence rate. 
% We are working to improve the PIM to obtain higher accuracy by designing proper kernel 
% functions and utilizing more accurate quadrature procedure for discretizing the integral equation.

Based on the convergence result in this paper, we are able to show
that the spectra of the graph Laplacian with a proper normalization 
converge to the spectra of $\Delta_\M$ with the Neumann boundary condition. 
Moreover, we can obtain an estimate of the rate of the spectral convergence. 
The point integral method also applies to 
Poisson equation with Dirichlet boundary. And we can also show the convergence of the point integral method for the 
Dirichlet boundary. These results will be reported in the 
subsequent papers.

% The point integral method can be generalized to solve other PDEs, for example, 
% the elliptic equations with variable coefficients or jump coefficients. We are also working 
% to apply the PIM to solve the PDEs with time evolution, for instance, convection-diffusion equation 
% on manifold. Recently, the eigensystem of the Laplace-Beltrami operator has attracted many attentions
% in various areas including machine learning, geometric modeling, and data analysis.
% The point integral method can be used to solve this problem very naturally. Moreover, 
% with the help of the convergence results in this paper, we can prove the convergence of the point integral method for the eigenvalue problem of 
% Laplace-Beltrami operator on point cloud and also give an estimate of the convergence rate.  
% This work will be reported in our subsequent paper.

% Another interesting problem worth investigating is 
% the property of the integral equations~\eqref{eq:integral}
% as they can be easily discretized and well approximated from point clouds. 
% For instance, what is the regularity of the solutions to the integral equations when $t$ goes to $0$. 
% Furthermore, those integral equations
% are well-defined on more general spaces, like multiple manifolds with self intersections or even general stratified
% spaces. One may study how to infer the structures of the underlying complicated spaces from the behavior of those 
% integral equations. 

\vspace{0.5in}
\noindent
{\bf Acknowledgments.}
This research was supported by NSFC Grant 11371220. 
%and  National Basic Research Program of China (973 Program 2012CB825500 to J.S.).

%\newpage
\appendix
\label{sec:appendix}
\vspace{8mm}
\noindent
% \setcounter{section}{2}
% \setcounter{equation}{0}

% \vspace{8mm}
% \noindent
\section{ Proof of Proposition \ref{prop:local_param}}

To prove the proposition, we first cite a few results from Riemannian geometry on isometric embeddings. 
For a submanifold $\M$ embedded in $\R^d$, let $d_\M: \M\times\M\rightarrow \R$
be the geodesic distance on $\M$, and $T_\bx\M$ and $N_\bx\M$ be the tangent space and the normal space 
of $\M$ at point $\bx\in \M$ respectively.

\begin{lemma}(eg \cite{Dey11hom}) Assume $\M$ is a submanifold  isometrically embedded in $\R^d$ with reach $\sigma>0$. 
For any two $\bx, \by$ on $\M$ with $|\bx-\by| \leq \sigma/2$, $$|\bx-\by|\leq d_\M(\bx, \by) \leq |\bx-\by|(1+\frac{2|\bx-\by|^2}{\sigma^2}).$$
\label{lem:geod_bound}
\end{lemma}
\begin{lemma}(eg \cite{pcdlp2009}) Assume $\M$ is a submanifold  isometrically embedded in $\R^d$ with reach $\sigma>0$. 
For any two $\bx, \by$ on $\M$ with $|\bx-\by| \leq \sigma/2$, $$ \cos \angle T_\bx\M, T_\by \M \leq 1- \frac{2|\bx-\by|^2}{\sigma^2}. $$
\label{lem:angle_bound}
\end{lemma}
\begin{lemma}(eg \cite{NiyogiSW08}) Assume $\M$ is a submanifold  isometrically embedded in $\R^d$ with reach $\sigma>0$. 
Let $N$ be any local normal vector field around a point $\bx\in \M$. Then for any tangent vector $Y\in T_\bx\M$
$$ \frac{<Y, D_Y N>}{<Y, Y>}\leq \frac{1}{\sigma}$$
where $D$ and $<\cdot, \cdot>$ are the standard connection and the standard inner product in $\R^d$. 
\label{lem:2nd_form}
\end{lemma}

In what follows, assume the hypotheses on $\M$ and $\p \M$ in Proposition~\ref{prop:local_param} hold. 
We prove the following two lemmas which bound the distortion of certain parametrization, which are used
to build the parametrization stated in Proposition~\ref{prop:local_param}. 

For a point $\bx \in \M,$, let $U_\rho = B_\bx(\rho \sigma)\cap \M $ with $\rho \leq 0.2$. We define the following 
projection map $\Psi: U_\rho \rightarrow T_\bx \M = \R^k $ as the restriction to $U_\rho$ of the projection of $\R^d$ 
onto $T_\bx \M$. It is easy to verify that $\Psi$ is one-to-one.  
Then $\Phi = \Psi^{-1}: \Psi(U_\rho) \rightarrow U_\rho$ is a parametrization 
of $U_\rho$. See Figure~\ref{fig:closest_point_map}. 
We have the following lemma which bounds the distortion of this parametrization.  
\begin{figure}
\begin{center}
\begin{tabular}{c}
\includegraphics[width=0.5\textwidth]{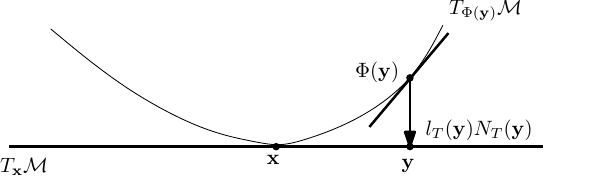}
\end{tabular}
\end{center}
\caption{Parametrization for a neighborhood of a point on $\M$. } 
\label{fig:closest_point_map}
\end{figure}

\begin{lemma}
For any point $\by \in \Psi(U_\rho)$ and any $Y\in T_\by (T_\bx \M)$ for any $\rho \leq 0.2$,  
$$ |Y| \leq | D_Y \Phi(\by) | \leq \frac{1}{1-2\rho^2}|Y|.$$
\label{lem:dist_interior}
\end{lemma}
\begin{proof}
We have $\Phi(\by) = \by - l_T(\by) N_T(\by)$ where $N_T(\by) \perp T_\bx \M$ for any $\by$ and 
$l_T(\by) = |\by - \Phi(\by)|$.
So $D_Y N_T(\by) \perp T_\bx \M$ for any $\by$ and any $Y\in T_\by (T_\bx \M)$. 
Since $D_Y \Phi = Y - N_T \left(D_Y l_T\right)- l_T \left(D_Y N_T\right)$, 
%$D_Y \Phi(\by) = Y - N_T(\by) D_Y l_T(\by)- l_T(\by) D_Y N_T(\by)$, 
the projection of $D_Y \Phi$ to $T_\bx \M$ is $Y$. 
At the same time,  $D_Y\Phi$ is on $T_{\Phi(\by)}\M$. Since $|\bx - {\Phi(\by)}|\leq \rho \sigma$, 
from Lemma~\ref{lem:angle_bound}, $\cos \angle T_\bx\M, T_{\Phi(\by)} \M \leq 1- 2\rho^2$. 
This proves the lemma. 
\end{proof}

To ensure the convexity of the parameter domain $\Omega$ in Proposition~\ref{prop:local_param}, 
We need a different parametrization for the points near the boundary. 
For a point $\bx \in \p \M$, let $U_\rho = B_\bx(\rho \sigma) \cap \M$ with $\rho \leq 0.1$. 
We construct a map $\tilde{\Psi}: U_\rho \rightarrow T_\bx\p \M \times \R = \R^k$ as follows.
For any point $\bz \in U_\rho$, let $\bar{\bz}$ be the closest point on $\p \M$ to $\bz$. 
Such $\bar{\bz}$ is unique. Let $\bfn$ be the outward normal of $\p \M$ at $\bar{\bz}$. 
The projection $P$ of $\R^d$ onto $T_{\bar{\bz}} \M$ maps $\bz$ to a point on the line $\ell$ passing through $\bar{\bz}$ 
with the direction $\bfn$. In fact, $P$ projects $N_{\bar{\bz}} \p \M$ onto the line $\ell$. 
If let $ V_{\rho_1} = N_{\bar{\bz}} \p \M \cap B_{\bar{\bz}}(\rho_1 \sigma) \cap \M $
with $\rho_1 \leq 0.2$, $P$ maps $ V_{\rho_1}$ to the line $\ell$ in the one-to-one manner. 
Let $y^k = - (P(\bz) - \bar{\bz}) \cdot \bfn$. 
Think of $\p \M$ as a submanifold. It is isometrically embedded in $\R^d$ as is $\M$. 
As $|\bar{\bz} - \bx| \leq 2|\bx - \bz| \leq 2\rho \sigma$, 
we apply Lemma~\ref{lem:dist_interior} by replacing $\M$ with $\p\M$ and obtain the map $\Psi$ that maps $\bar{\bz}$ onto $T_\bx\p \M$. 
Define $\tilde{\Psi}(\bz) = (\Psi(\bar{\bz}), y^k)$. Since 
both $P|_{V_{\rho_1}}$ and $\Psi$ are one-to-one, so is $\tilde{\Psi}$.  
Then $\tilde{\Phi} = \tilde{\Psi}^{-1}: \tilde{\Psi}(U_\rho) \rightarrow U_\rho$ is a parametrization of $U_\rho$. 
See Figure~\ref{fig:param_boundary}. 
We have the following lemma which bounds the distortion of this parametrization $\tilde{\Phi}$. 

\begin{lemma}
For any point $(\by, y^k) \in \tilde{\Psi}(U_\rho)$ with $\rho \leq 0.1$ and 
any tangent vector $Y$ at $(\by, y^k)$,   
$$ (1-2\rho)|Y| \leq | D_Y\tilde{\Phi}(\by, y^k) | \leq (1+2\rho)|Y|.$$
\label{lem:dist_boundary}
\end{lemma}

\begin{figure}
\begin{center}
\begin{tabular}{c}
\includegraphics[width=0.9\textwidth]{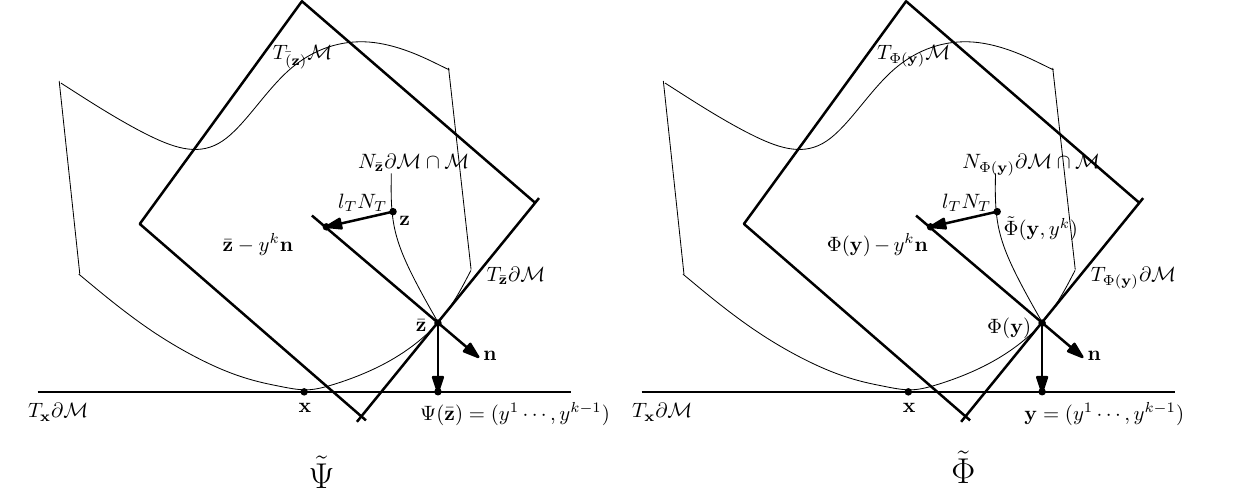}
\end{tabular}
\end{center}
\caption{Parametrization for a neighborhood of a point on $\p \M$} 
\label{fig:param_boundary}
\end{figure}

\begin{proof}
Let $\bar{\by} = \Phi(\by) - y^k \bfn$. 
We have $\tilde{\Phi}(\by, y^k) = \Phi(\by) - y^k \bfn(\Phi(\by)) - l_T(\bar{\by}) N_T(\bar{\by})$ where
$N_T(\bar{\by}) \perp T_{\Phi(\by)} \M$. See Figure~\ref{fig:param_boundary}. 
We have $$
D_Y\tilde{\Phi}(\by, y^k) = D_Y\Phi - y^k D_Y\bfn - \bfn D_Yy^k - N_TD_Yl_T - l_T D_YN_t. 
$$
Using the similar strategy of proving Lemma~\ref{lem:dist_interior}, we consider the 
projection of $D_Y\tilde{\Phi}(\by, y^k)$ to the space $T_\Phi(\by)M$ to which it is almost 
parallel. Denote $P$ this projection map . We bound $P(D_Y\tilde{\Phi}(\by, y^k))$. 
Let $Y = (Y^1, \cdots, Y^k)$, $Y_1 = (Y^1, \cdots, Y^{k-1}, 0)$ and $Y_2=(0, \cdots, 0, Y^k)$. 
We have $D_Y\tilde{\Phi}(\by, y^k) = D_{Y_1}\tilde{\Phi}(\by, y^k) + D_{Y_2}\tilde{\Phi}(\by, y^k)$. 
First consider each term involved in $D_{Y_1}\tilde{\Phi}(\by, y^k)$. 
\begin{enumerate}
\item[(i)] $D_{Y_1}\Phi(\by)$ is a vector in $T_{\Phi(\by)}\p \M$, thus $P(D_{Y_1}\Phi(\by)) = D_{Y_1}\Phi(\by)$. 
In addition, from Lemma~\ref{lem:dist_interior},  $|Y_1|\leq \left|D_{Y_1}\Phi(\by)\right| \leq \frac{1}{1-2\rho^2}|Y_1|$. 

\item[(ii)] $D_{Y_1}\bfn(\by, y^k) = D_{D_{Y_1}\Phi} \bfn(\Phi(\by))$.  First note that 
$\bfn \cdot D_{D_{Y_1}\Phi} \bfn = 0$. Second, from Lemma~\ref{lem:2nd_form}, 
we have that the projection of $D_{D_{Y_1}\Phi} \bfn$ to the space $T_{\Phi(\by)}\p\M$
is upper bounded by $\frac{1}{\sigma}\left|D_{Y_1}\Phi\right|$. Since $|y^k| <\rho \sigma$, 
$|P( y^k D_{Y_1}\bfn)| \leq \frac{\rho}{1-2\rho^2}|Y_1|$. 

\item[(iii)] Consider $D_{Y_1}N_T(\by, y^k)$. We have $N_T \perp T_{\Phi(\by)}\M$. 
Let $e_1, \cdots, e_k$ be the orthonormal basis of $T_{\Phi(\by)}\M$ so that
$D_{e_i} N_T \cdot e_j = 0$ for $i\neq j$. Locally extend $e_1, \cdots, e_k$ to be
an orthonormal basis of $T\M$ in a neighborhood of $\Phi(\by)$. We have for any $e_i$
\begin{eqnarray*}
\left|D_{Y_1}N_T(\by, y^k) \cdot e_i \right| &=& \left| D_{D_{Y_1}\Phi} N_T(\bar{\by})\cdot e_i\right| \\
&=& \left| D_{\left(D_{Y_1}\Phi\cdot e_i\right) e_i} N_T(\bar{\by})\cdot e_i\right|\\
&=& \left| D_{\left(D_{Y_1}\Phi\cdot e_i\right) e_i} N_T(\bar{\by})\cdot e_i\right|\\
&\leq& \frac{1}{\sigma}|D_{Y_1}\Phi\cdot e_i|
\end{eqnarray*}
where the last inequality is due to Lemma~\ref{lem:2nd_form}. Moreover, one can verify that
$l_T(\bar{\by}) \leq \frac{\rho^2\sigma}{2}$, which leads to 
$$\left|P(l_T D_{Y_1}N_t)\right| \leq \frac{\rho^2}{2}|D_{Y_1}\Phi| \leq \frac{\rho^2}{2(1-2\rho^2)}|Y_1|$$
\item[(iv)] It is obvious that $\bfn D_{Y_1}y^k=P(N_T D_{Y_1}l_T) = 0$. 
\end{enumerate}

Next consider each term involved in $D_{Y_2}\tilde{\Phi}(\by, y^k)$.
\begin{enumerate}
\item[(i)] $\bfn D_{Y_2}y^k= Y^k \bfn$, which lies on $T_{\Phi(\by)} \M$. Moreover 
$\bfn \perp D_{Y_1}\Phi(\by)$. 
\item[(ii)] As $N_T(\by, y^k)$ remains perpendicular to $T_{\Phi(\by)}\M$ if we only vary $y^k$, we have 
$$P(D_{Y_2}N_t(\by, y^k)) = 0.$$ 
\item[(iii)] For the remaining terms, we have $D_{Y_2}\Phi(\by) = y^kD_{Y_2}\bfn = P(N_T D_{Y_2}l_T) = 0$. 
\end{enumerate}

On the other hand, we hand $D_Y\tilde{\Phi}(\by, y^k)$ lie in the tangent space $T_{\tilde{\Phi}(\by, y^k)}\M$, 
and $$ \cos \angle T_{\tilde{\Phi}(\by, y^k)}\M, T_{\Phi(\by)}\M \leq 1-2\rho^2.$$ 
Putting everything together, we have 
$$
|Y| - \frac{\rho^2 + 2\rho}{2(1-2\rho^2)}|Y_1| \leq D_Y\tilde{\Phi}(\by, y^k) \leq \frac{1}{(1-2\rho^2)^2}|Y| + \frac{\rho^2 + 2\rho}{2(1-2\rho^2)^2}|Y_1|. 
$$
This proves the lemma. 
\end{proof}

Now we are ready to prove Proposition~\ref{prop:local_param}
\begin{proof} of Proposition~\ref{prop:local_param} \\
First consider the case where $d(\bx, \p \M) > \frac{\rho}{2}\sigma$. 
Set $U' = B_\bx(\frac{\rho}{2} \sigma) \cap \M$,  
and parametrize $U'$ using map $\Phi: \Psi(U') \rightarrow U'$.  
Since for any $\by \in \p U'$, $|\bx-\by| = \frac{\rho}{2} \sigma$, 
from Lemma~\ref{lem:dist_interior}, we have that 
$B_{\Phi^{-1}(\bx)}(\frac{\rho}{2(1+\rho)}\sigma )$ is contained in $\Psi(U')$. 
Set $\Omega = B_{\Phi^{-1}(\bx)}(\frac{\rho\sigma}{2(1+\rho)})$ and $U = \Phi(\Omega)$.
This shows the parametrization $\Phi: \Omega \rightarrow U$ satisfies the condition (i). 
By Lemma~\ref{lem:dist_interior} and Lemma~\ref{lem:geod_bound},
it is easy to verify that $\Phi$ satisfies the other three conditions. 

Next consider the case where $d(\bx, \p \M) \leq \frac{\rho}{2}\sigma$. 
Let $\bar{\bx}$ be the closest point on $\p \M$ to $\bx$. 
Set $U' = B_{\bar{\bx}}(\rho \sigma)\cap \M$
and parametrize $U'$ using map $\tilde{\Phi}: \tilde{\Psi}(U')\rightarrow U'$. 
By Lemma~\ref{lem:dist_boundary},  $\tilde{\Psi}(U')$ contains 
half of the ball $B_{\tilde{\Phi}^{-1}(\bar{\bx})}(\frac{\rho\sigma}{1+2\rho})$. 
Let $\Omega$ be that half ball and $U = \tilde{\Phi}(\Omega)$. 
It is easy to verify
that the parametrization $\tilde{\Phi}: \Omega \rightarrow U$ satisfies the condition
(iii) and (iv). To see (i), note that $|\bx - \bar{\bx}| \leq \frac{\rho}{2}\sigma$. 
From Lemma~\ref{lem:dist_boundary} and Lemma~\ref{lem:geod_bound}, 
$\left|\tilde{\Psi}(\bx) - \tilde{\Psi}(\bar{\bx})\right| \leq (1+2\rho)(1+2\rho^2)|\bx - \bar{\bx}|$. 
We have that $\Omega$ contains at least 
half of the ball centered at ${\Phi}^{-1}(\bx)$ with radius 
$(\frac{\rho}{1+2\rho} - \frac{\rho(1+2\rho)(1+2\rho^2)}{2})\sigma \ge \frac{\rho}{5}\sigma$. 
This shows that $\tilde{\Phi}$ satisfies the condition (i). Similarly, the condition (ii) follows from (i)
as $\tilde{\Phi}$ has bounded distortion (Lemma~\ref{lem:dist_boundary}) and 
geodesic distance  is bounded by Euclidean distance (Lemma~\ref{lem:geod_bound}).  
\end{proof}

\section{Proof of Lemma \ref{lem:elliptic_v}}
\begin{proof}
We start with the evaluation of the $x^i$ component of $\nabla v$. 
\begin{eqnarray*}
    \nabla^i v(\bx)
&=&\frac{C_t^2}{2t w_t^2(\bx)}\int_\mathcal{M}\int_{\mathcal{M}}\nabla^i x^j(x^j-y^j)R'\left(\frac{|\bx-\by|^2}{4t}\right)
R\left(\frac{|\bx-\by'|^2}{4t}\right)
 u(\by)\mathd \mu_\by'\mathd \mu_\by\nonumber\\
&&-\frac{C_t^2}{2t w_t^2(\bx)}\int_\mathcal{M}\int_{\mathcal{M}}\nabla^i x^j(x^j-y'^j)R'\left(\frac{|\bx-\by'|^2}{4t}\right)
R\left(\frac{|\bx-\by|^2}{4t}\right)
 u(\by)\mathd \mu_\by'\mathd \mu_\by\nonumber\\
&=&\frac{C_t^2}{4t w_t^2(\bx)}\int_\mathcal{M}\int_{\mathcal{M}}K^i(\bx,\by,\by';t)
 (u(\by)-u(\by'))\mathd \mu_\by'\mathd \mu_\by
  \end{eqnarray*}
where we set 
\begin{eqnarray*}
    K^i(\bx,\by,\by';t)=\nabla^i x^j(x^j-y^j)R'\left(\frac{|\bx-\by|^2}{4t}\right)
R\left(\frac{|\bx-\by'|^2}{4t}\right) \nonumber \\
-\nabla^i x^j(x^j-y'^j)R'\left(\frac{|\bx-\by'|^2}{4t}\right)
R\left(\frac{|\bx-\by|^2}{4t}\right). 
\end{eqnarray*}
Think of $\nabla^i x^j$ as the $i,j$ entry of the matrix $[\nabla^i x^j]$ and 
we have
\begin{eqnarray*}
\nabla^i x^j\nabla^l x^i
&=&(\p_{i'}x^i)g^{i'j'}(\p_{j'}x^j)(\p_{s'}x^l)g^{s't'}(\p_{t'}x^i)\nonumber\\
&=&g_{i't'}g^{i'j'}(\p_{j'}x^j)g^{s't'}(\p_{s'}x^l)\nonumber\\
&=&\delta_{j't'}(\p_{j'}x^j)g^{s't'}(\p_{s'}x^l)\nonumber\\
&=&(\p_{j'}x^j)g^{s'j'}(\p_{s'}x^l)\nonumber\\
&=&\nabla^l x^j.
\end{eqnarray*}
This shows that the matrix $[\nabla^i x^j]$ is idempotent. 
At the same time,  $[\nabla^i x^j]$ is symmetric, which implies
that the eigenvalues of $\nabla \bx$ are either $1$ or $0$. 
Then we have the following upper bounds. There exists a constant $C$ depending 
only on the maximum of $R$ and $R'$ so that
\begin{eqnarray*}
&&\sum_{i=1}^d K^i(\bx,\by,\by';t)^2 \nonumber \\ 
&\le& 2\left(R'\left(\frac{|\bx-\by|^2}{4t}\right) R\left(\frac{|\bx-\by'|^2}{4t}\right)\right)^2 \|[\nabla^i x^j](\bx-\by)\|^2 \nonumber \\ 
	&+& 2\left(R'\left(\frac{|\bx-\by'|^2}{4t}\right) R\left(\frac{|\bx-\by|^2}{4t}\right)\right)^2 \|[\nabla^i x^j](\bx-\by')\|^2 \nonumber \\
&\le&	 CR'\left(\frac{|\bx-\by|^2}{4t}\right) R\left(\frac{|\bx-\by'|^2}{4t}\right) \|\bx-\by\|^2 \nonumber \\ 
	&+& CR'\left(\frac{|\bx-\by'|^2}{4t}\right) R\left(\frac{|\bx-\by|^2}{4t}\right) \|\bx-\by'\|^2
\end{eqnarray*}
There exists a constant $C$ independent of $t$ so that
\begin{eqnarray*}
&&C_t^2\int_\M \int_\M \frac{\sum_{i=1}^d K^i(\bx,\by,\by';t)^2}{t} \mathd \mu_\by \mathd \mu_\by' \nonumber \\ 
&\le&	 C\int_\M C_t R'\left(\frac{|\bx-\by|^2}{4t}\right) \frac{\|\bx-\by\|^2}{t}\mathd \mu_\by \int_\M C_t R\left(\frac{|\bx-\by'|^2}{4t}\right)  \mathd \mu_\by' \nonumber \\ 
	&+& C\int_\M C_t R'\left(\frac{|\bx-\by'|^2}{4t}\right) \frac{\|\bx-\by'\|^2}{t} \mathd \mu_\by'\int_\M C_t R\left(\frac{|\bx-\by|^2}{4t}\right) \mathd \mu_\by \nonumber \\
&\le& C
\end{eqnarray*}

Since $R$ has a compact support, only when $|\by -\by'|^2 < 16t$ and $|\bx- \frac{\by +\by'}{2}|^2 < 4t$
is $K^i(\bx, \by, \by'; t) \neq 0$. Thus from the assumption on $R$, we have
$$K^i(\bx, \by, \by'; t)^2 \leq \frac{1}{\delta_0^2} K^i(\bx, \by, \by'; t)^2 R\left(\frac{|\by-\by'|^2}{32t}\right) R\left(\frac{|\bx-\frac{\by+\by'}{2}|^2}{8t}\right).$$
  %\begin{eqnarray}
    %\mathcal{M}_{\bx}^t=\left\{\by\in \mathcal{M}: |\by-\bx|^2\le 4t\right\}. 
  %\end{eqnarray}
%Then for all $\by,\by'\in \mathcal{M}_\bx^t$
  %\begin{eqnarray}
    %\frac{|\by-\by'|^2}{32t}\le \frac{1}{2},\quad \frac{|\bx-(\by+\by')/2|^2}{8t}\le \frac{1}{2}, 
  %\end{eqnarray}
%which by our assumption on $R$ implies
  %\begin{eqnarray}
    %R\left(\frac{|\by-\by'|^2}{32t}\right)\ge \delta_0,\quad R\left(\frac{|\bx-(\by+\by')/2|^2}{8t}\right)\ge \delta_0.
  %\end{eqnarray}
  
%\begin{eqnarray}
    %|\bx-\by|^2+|\bx-\by'|^2=\frac{1}{2}|\by-\by'|^2+2|\bx-(\by+\by')/2|^2
%\end{eqnarray}

%   \begin{eqnarray}
%     \widetilde{\mathcal{M}}_\bx^t=\left\{(\by,\by')\in \mathcal{M}\times\mathcal{M}: |\by-\by'|^2\le 32t,
% \quad|\bx-(\by+\by')/2|^2\le 8t \right\}
%   \end{eqnarray}
We can upper bound the norm of $\nabla v$ as follows: 
\begin{eqnarray}
&&|\nabla v(\bx)|^2=\frac{C_t^4}{16t^2 w_t^4(\bx)}\sum_{i=1}^d
\left(\int_{\mathcal{M}}\int_{\mathcal{M}}K^i(\bx,\by,\by';t)
 (u(\by)-u(\by'))\mathd \by'\mathd \by\right)^2\nonumber\\
&\le & \frac{C_t^4}{16t^2 w_t^4(\bx)}\sum_{i=1}^k
\int_{\mathcal{M}}\int_{\mathcal{M}}K_i^2(\bx,\by,\by';t)\left(R\left(\frac{|\by-\by'|^2}{32t}\right)
R\left(\frac{|\bx-\frac{\by+\by'}{2}|^2}{8t}\right)\right)^{-1}\mathd \mu_\by'\mathd \mu_\by\nonumber\\
&&\int_{\mathcal{M}}\int_{\mathcal{M}}
R\left(\frac{|\bx-\frac{\by+\by'}{2}|^2}{8t}\right) R\left(\frac{|\by-\by'|^2}{32t}\right)
(u(\by)-u(\by'))^2\mathd \mu_\by'\mathd \mu_\by\nonumber\\
&=&\frac{C_t^4}{16t \delta_0^2w_t^4(\bx)}
\int_{\mathcal{M}}\int_{\mathcal{M}}\frac{\sum_{i=1}^d K^i(\bx,\by,\by';t)^2}{t}\mathd \mu_\by'\mathd \mu_\by\nonumber\\
&&\int_{\mathcal{M}}\int_{\mathcal{M}}R\left(\frac{|\bx-\frac{\by+\by'}{2}|^2}{8t}\right) R\left(\frac{|\by-\by'|^2}{32t}\right)
 (u(\by)-u(\by'))^2\mathd \mu_\by'\mathd \mu_\by\nonumber\\
&\le &\frac{C C_t^2}{t}
\int_{\mathcal{M}}\int_{\mathcal{M}}R\left(\frac{|\bx-\frac{\by+\by'}{2}|^2}{8t}\right) R\left(\frac{|\by-\by'|^2}{32t}\right)
 (u(\by)-u(\by'))^2\mathd \mu_\by'\mathd \mu_\by.\nonumber
\end{eqnarray}
Finally, we have
\begin{eqnarray}
  &&\int_{\mathcal{M}}|\nabla v(\bx)|^2\mathd \mu_\bx\nonumber\\
&\le& \frac{C C_t^2}{t}\int_{\mathcal{M}}\left(
\int_\mathcal{M}\int_{\mathcal{M}} R\left(\frac{|\bx-\frac{\by+\by'}{2}|^2}{8t}\right) R\left(\frac{|\by-\by'|^2}{32t}\right)
(u(\by)-u(\by'))^2\mathd \mu_\by'\mathd \mu_\by\right)\mathd \mu_\bx
\nonumber\\
&=&\frac{C C_t^2}{t}
\int_\mathcal{M}\int_{\mathcal{M}}\left(\int_{\mathcal{M}}R\left(\frac{|\bx-\frac{\by+\by'}{2}|^2}{8t}\right)\mathd \mu_\bx\right) 
R\left(\frac{|\by-\by'|^2}{32t}\right)(u(\by)-u(\by'))^2\mathd \mu_\by'\mathd \mu_\by\nonumber\\
&\le & \frac{C C_t}{t} \int_\mathcal{M}\int_{\mathcal{M}} R\left(\frac{|\by-\by'|^2}{32t}\right)(u(\by)-u(\by'))^2\mathd \mu_\by'\mathd \mu_\by. 
\nonumber
\end{eqnarray}
This proves the Lemma. 
\end{proof}

% \setcounter{section}{3}
% \setcounter{equation}{0}

% \vspace{5mm}
% \noindent
\section{Proof of Lemma \ref{lem:bigt2smallt}}
Based on the partition and the parametrization of the manifold $\M$ introduced in Section \ref{sec:err_int}, we have
  \begin{eqnarray}
\label{eq:decomp_local}
    && \int_\mathcal{M}\int_{\mathcal{M}}R\left(\frac{|\bx-\by|^2}{32t}\right) (u(\bx)-u(\by))^2\mathd \mu_\bx \mathd \mu_\by\\
&= & \sum_{i=1}^N\int_{\mathcal{M}}\int_{\mathcal{O}_i}R\left(\frac{|\bx-\by|^2}{32t}\right) (u(\bx)-u(\by))^2\mathd \mu_\bx \mathd \mu_\by\nonumber\\
&= & \sum_{i=1}^N\int_{B_{\bq_i}^{2\delta}}\int_{\mathcal{O}_i}R\left(\frac{|\bx-\by|^2}{32t}\right) (u(\bx)-u(\by))^2\mathd \mu_\bx \mathd \mu_\by.\nonumber
  \end{eqnarray}
%In the last equality, we use the fact that $R$ is compactly supported and $t$ is small enough.
% By Proposition~\ref{prop:local_param}, we have there exist 
% a parametrization $\Phi_i: \Omega_i\subset\mathbb{R}^k \rightarrow U_i\subset \mathcal{M},; i=1,\cdots, N$, such that
% \begin{itemize}
% \item[1.] $B_{\bx_i}^{2r}\subset U_i$ and $\Omega_i$ is convex;
% \item[2.] $\Phi\in C^3(\Omega)$; 
% \item[3.] For any points $\bx, \by\in \Omega$, $\frac{1}{2}\left|\bx-\by\right| \leq \left\|\Phi_i(\bx)-\Phi_i(\by)\right\|  \leq 2\left|\bx-\by\right|$.
% \end{itemize}

For any $\bx\in \mathcal{O}_i$ and $\by\in B_{\bq_i}^{2\delta}$,
let 
\begin{eqnarray}
 \B{z}_j=\Phi_i\left(\left(\frac{j}{16}\right)\Phi_i^{-1}(\bx)+\left(1-\frac{j}{16}\right)\Phi_i^{-1}(\by)\right),\quad j=0,\cdots,16. 
\end{eqnarray}
Apparently, $\B{z}_0=\bx,\;\B{z}_{16}=\by$. Since $\Omega_i$ is convex, we have $\Phi_i^{-1}(\B{z}_j)\in \Omega_i,\;i=0,\cdots,16$. 
Then utilizing {\it locally small deformation} property of the parametrization, we obtain
\begin{eqnarray*}
  \|\B{z}_j-\B{z}_{j+1}\|&\le& 2   \|\Phi^{-1}(\B{z}_j)-\Phi^{-1}(\B{z}_{j+1})\|\nonumber\\
&\le & \frac{1}{8}   \|\Phi^{-1}(\B{x})-\Phi^{-1}(\B{y})\|\nonumber\\
&\le & \frac{1}{4}\|\bx-\by\|.
\end{eqnarray*}
Now, we are ready to estimate the integrals in  \eqref{eq:decomp_local}.
\begin{eqnarray*}
&&    \int_{B_{\bq_i}^{2\delta}}\int_{\mathcal{O}_i}R\left(\frac{|\bx-\by|^2}{32t}\right) (u(\bx)-u(\by))^2\mathd \mu_\bx \mathd\mu_\by\nonumber\\
&\le& 16\sum_{j=0}^{15}\int_{B_{\bq_i}^{2\delta}}\int_{\mathcal{O}_i}R\left(\frac{|\bx-\by|^2}{32t}\right) (u(\B{z}_j)-u(\B{z}_{j+1}))^2\mathd \mu_\bx \mathd \mu_\by\nonumber\\
&=&  16\sum_{j=0}^{15}\int_{\mathcal{O}_i}\left[\int_{\mathcal{M}_{\bx}^t}R\left(\frac{|\bx-\by|^2}{32t}\right) (u(\B{z}_j)-u(\B{z}_{j+1}))^2\mathd \mu_\by\right] \mathd \mu_\bx.
\end{eqnarray*}
%The last equality is due to the fact that the support of $R$ is $[0,1)$. 
For any $\by\in \mathcal{M}_\bx^t$,
\begin{eqnarray}
  \|\B{z}_j-\B{z}_{j+1}\|^2\le \frac{1}{16}\|\bx-\by\|^2\le 2t,\quad j=0,\cdots,15,
\end{eqnarray}
which implies that 
\begin{eqnarray}
  R\left(\frac{|\B{z}_j-\B{z}_{j+1}|^2}{4t}\right)\ge \delta_0,\quad j=0,\cdots,15.
\end{eqnarray}
Now, we have
  \begin{eqnarray}
 &&   \int_{\mathcal{O}_i}\left[\int_{\mathcal{M}_{\bx}^t}R\left(\frac{|\bx-\by|^2}{32t}\right) (u(\B{z}_j)-u(\B{z}_{j+1}))^2\mathd \mu_\by\right] \mathd \mu_\bx\nonumber\\
&=& \int_{\mathcal{O}_i}\left[\int_{\mathcal{M}_{\bx}^t}R\left(\frac{|\bx-\by|^2}{32t}\right)\left(R\left(\frac{|\B{z}_j-\B{z}_{j+1}|^2}{4t}\right)\right)^{-1}
R\left(\frac{|\B{z}_j-\B{z}_{j+1}|^2}{4t}\right) (u(\B{z}_j)-u(\B{z}_{j+1}))^2\mathd \mu_\by \right]\mathd \mu_\bx\nonumber\\
&\le & \frac{1}{\delta_0}  \int_{\mathcal{O}_i}\left[\int_{\mathcal{M}_{\bx}^t}
R\left(\frac{|\B{z}_j-\B{z}_{j+1}|^2}{4t}\right) (u(\B{z}_j)-u(\B{z}_{j+1}))^2\mathd \mu_\by \right]\mathd \mu_\bx\nonumber\\
&=& \frac{1}{\delta_0}  \int_{\Phi_i^{-1}(\mathcal{O}_i)}\left[\int_{\Phi_i^{-1}(\mathcal{M}_{\bx}^t)}
R\left(\frac{|\B{z}_j-\B{z}_{j+1}|^2}{4t}\right) (u(\B{z}_j)-u(\B{z}_{j+1}))^2 \left|\nabla \Phi(\theta_\by)\right|\mathd \theta_{\by}\right]\left|\nabla \Phi(\theta_\bx)\right|
 \mathd \theta_{\bx}\nonumber\\
&\le& \frac{4}{\delta_0}  \int_{\Phi_i^{-1}(\mathcal{O}_i)}\left[\int_{\Phi_i^{-1}(\mathcal{M}_{\bx}^t)}
R\left(\frac{|\B{z}_j-\B{z}_{j+1}|^2}{4t}\right) (u(\B{z}_j)-u(\B{z}_{j+1}))^2 \mathd \theta_{\by}\right] \mathd \theta_{\bx}\nonumber
  \end{eqnarray}
where $\theta_\bx=\Phi_i^{-1}(\bx),\; \theta_\by=\Phi_i^{-1}(\by)$.

Let 
\begin{eqnarray}
\theta_{\B{z}_j}=\Phi_i^{-1}(\B{z}_j)=\frac{j}{16}\theta_{\bx}+\left(1-\frac{j}{16}\right)\theta_{\by}, \quad j=0,\cdots,16.  
\end{eqnarray}
It is easy to show that $\Phi_i(\theta_{\B{z}_j})=\B{z}_j \in B_{\bq_i}^{2\delta},\; j=0,\cdots,16$ by using the facts that for any $\by\in \mathcal{M}_\bx^t$
  \begin{eqnarray}
    \|\B{z}_j-\bx\|\le \sum_{l=1}^{j}\|\B{z}_l-\B{z}_{l-1}\|\le \frac{j}{4} \|\bx-\by\|\le 15\sqrt{2t},\quad j=1,\cdots,15,
  \end{eqnarray}
and $\bx\in B_{\bq_i}^{\delta}$ and $15\sqrt{2t}\le r$. Then we have
\begin{eqnarray}
  \theta_{\B{z}_j}\in \Phi_i^{-1}\left(B_{\bq_i}^{2\delta}\right),\quad j=0,\cdots,16.
\end{eqnarray}
By changing variable, we obtain
\begin{eqnarray*}
&&\int_{\Phi_i^{-1}(\mathcal{O}_i)}\left[\int_{\Phi_i^{-1}(\mathcal{M}_{\bx}^t)}
R\left(\frac{|\B{z}_j-\B{z}_{j+1}|^2}{4t}\right) (u(\B{z}_j)-u(\B{z}_{j+1}))^2 \mathd \theta_{\by}\right] \mathd \theta_{\bx}\nonumber\\
&\le & 8^k  \int_{\Phi_i^{-1}(B_{\bq_i}^{2\delta})}\int_{\Phi_i^{-1}\left(B_{\bq_i}^{2\delta}\right)}
R\left(\frac{|\B{z}_j-\B{z}_{j+1}|^2}{4t}\right) (u(\B{z}_j)-u(\B{z}_{j+1}))^2 \mathd \theta_{\B{z}_j} \mathd \theta_{\B{z}_{j+1}}\nonumber\\
&\le & 4\cdot8^k  \int_{\Phi_i^{-1}(B_{\bq_i}^{2\delta})}\int_{\Phi_i^{-1}\left(B_{\bq_i}^{2\delta}\right)}
R\left(\frac{|\B{z}_j-\B{z}_{j+1}|^2}{4t}\right) (u(\B{z}_j)-u(\B{z}_{j+1}))^2\left|\nabla \Phi(\theta_{\B{z}_j})\right|\left|\nabla \Phi(\theta_{\B{z}_{j+1}})\right|
 \mathd \theta_{\B{z}_j} \mathd \theta_{\B{z}_{j+1}}\nonumber\\
&= & 4\cdot8^k  \int_{B_{\bq_i}^{2\delta}}\int_{B_{\bq_i}^{2\delta}}
R\left(\frac{|\B{z}_j-\B{z}_{j+1}|^2}{4t}\right) (u(\B{z}_j)-u(\B{z}_{j+1}))^2
 \mathd \mu_{\B{z}_j} \mathd \mu_{\B{z}_{j+1}}\nonumber\\
&\le & C\int_{\mathcal{M}}\int_{B_{\bq_i}^{2\delta}}
R\left(\frac{|\bx-\by|^2}{4t}\right) (u(\bx)-u(\by))^2
 \mathd \mu_\bx \mathd \mu_\by\nonumber
\end{eqnarray*}

% \begin{eqnarray}
% &&  \frac{1}{\delta_0}  \int_{\mathcal{M}_{\bx_i}^{r}}\left[\int_{\mathcal{M}_{\bx}^t}
% R\left(\frac{|\B{z}_j-\B{z}_{j+1}|^2}{16t}\right) (u(\B{z}_j)-u(\B{z}_{j+1}))^2\mathd \by \right]\mathd \bx\nonumber\\
% &\le & \frac{C}{\delta_0}  \int_{\mathcal{M}_{\bx_i}^{2r}}\int_{\mathcal{M}_{\bx_i}^{2r}}
% R\left(\frac{|\B{z}_j-\B{z}_{j+1}|^2}{16t}\right) (u(\B{z}_j)-u(\B{z}_{j+1}))^2\mathd \B{z}_j\mathd \B{z}_{j+1}\nonumber\\
% &\le & \frac{C}{\delta_0}  \int_{\mathcal{M}}\int_{\mathcal{M}_{\bx_i}^{2r}}
% R\left(\frac{|\bx-\by|^2}{16t}\right) (u(\bx)-u(\by))^2\mathd \bx\mathd \by
% \end{eqnarray}
Finally, we can prove the lemma as follows.
\begin{eqnarray*}
    && \int_\mathcal{M}\int_{\mathcal{M}}R\left(\frac{|\bx-\by|^2}{32t}\right) (u(\bx)-u(\by))^2\mathd \mu_\bx \mathd \mu_\by\nonumber\\
&\le & C\sum_{i=1}^N\int_{\mathcal{M}}\int_{B_{\bq_i}^{2\delta}}R\left(\frac{|\bx-\by|^2}{4t}\right) (u(\bx)-u(\by))^2\mathd \mu_\bx \mathd \mu_\by\nonumber\\
&\le& CN \int_{\mathcal{M}}\int_{\mathcal{M}}R\left(\frac{|\bx-\by|^2}{4t}\right) (u(\bx)-u(\by))^2\mathd \mu_\bx \mathd \mu_\by.
  \end{eqnarray*}

\section{Proof of Theorem \ref{thm:elliptic_L}}

First, we introduce a smooth function $u_t$ that approximates $\bfu$ at the samples $P$. 
\begin{eqnarray}
  u_t(\bx)=\frac{C_t}{w_{t,h}(\bx)}\sum_{i=1}^nR\left(\frac{|\bx-\bfp_i|^2}{4t}\right)u_iV_i,
\label{eqn:def_dis_u} 
\end{eqnarray}
where $w_{t,h}(\bx)=C_t\sum_{i=1}^nR\left(\frac{|\bx-\bfp_i|^2}{4t}\right)V_i$.
We have the following lemma about the function $w_{t, h}$. 
\begin{lemma}
Assume the submanifold $\M$ and $\p \M$ are $C^2$ smooth and $t$, $\frac{h(P,\bV,\M)}{t^{1/2}}$ are sufficiently small. There exists a 
constant $C_1, C_2$ and $C$, 
so that $$C_1\leq w_{t, h}(\bx) \leq C_2, \text{~and~} |\nabla w_{t, h}(\bx)| \leq \frac{C}{t^{1/2}} $$  
%\begin{itemize}
%\item[(1)] There exists a constant $C_1, C_2$ independent of $t, h, \bx$, 
%so that $$C_1\leq w_{t, h}(\bx) \leq C_2.$$  
%\item[(2)] There exists a constant $C$ independent of $t$, so that
%$$|u(x)| \leq CC_t^{1/2} (\sum_{j=1}^nu_j^2V_j)^{1/2}.$$
%\item[(3)] There exists a constant $C$ independent of $t$, so that
%$$|\nabla u(x)| \leq \frac{CC_t}{t^{1/2}}(\sum_{j=1}^n u_j^2V_j)^{1/2}.$$
%\end{itemize}
\label{lem:bound_w_t_h}
\end{lemma}
\begin{proof}
Using the definition of $h(P,\bV,\M)$, 
$$\left|w_{t, h}(\bx) - C_t\int_\M R\left(\frac{|\bx-\by|^2}{4t}\right) \mathd \mu_\by\right| \leq \frac{Ch(P,\bV,\M)}{t^{1/2}}, $$ 
which shows the bounds on $w_{t, h}(\bx)$. Next, we show the bound on the gradient. 
\begin{eqnarray*}
|\nabla w_{t, h}(\bx)|^2 &\leq& \sum_{i=1}^d \left(\frac{\p w_{t, h}}{\p x^i}\right)^2 = \sum_{i=1}^d \left( \sum_{j=1}^n C_t R'\left(\frac{|\bx-\bfp_j|^2}{4t}\right) \frac{x^i - p_j^i}{2t} V_j \right)^2 \nonumber\\
&\leq& \left( \sum_{j=1}^n C_t R'\left(\frac{|\bx-\bfp_j|^2}{4t}\right) \frac{|\bx - \bfp_j|^2}{4t^2} V_j  \right) \left( \sum_{j=1}^n C_t R'\left(\frac{|\bx-\bfp_j|^2}{4t}\right) V_j \right)\nonumber \\
&\leq& \frac{C}{t}. 
\end{eqnarray*}
\end{proof}
Now we are ready to give the proof of Theorem \ref{thm:elliptic_L}.
\begin{proof} 

In the definition of $u_t$ and $w_{t, h}$ in \eqref{eqn:def_dis_u}, replace $t$ with $t'=t/18$. We have 
%Since $R$ is compactly supported, consider
%\begin{eqnarray}
    %\mathcal{M}_{\bx}^t=\left\{\by\in \mathcal{M}: |\by-\bx|^2\le 4t\right\}.
%\end{eqnarray}
\begin{eqnarray*}
&&\int_{\mathcal{M}}  \int_{\mathcal{M}}R_{t'}(\bx,\by) \left(u_t(\bx)-u_t(\by)\right)^2\mathd \mu_\bx \mathd\mu_\by\nonumber\\
&=& \int_{\mathcal{M}}  \int_{\mathcal{M}}R_{t'}(\bx,\by) \left(\frac{1}{w_{t',h}(\bx)}\sum_{i=1}^nR_{t'}(\bx,\bfp_i)u_iV_i
-\frac{1}{w_{t',h}(\by)}\sum_{j=1}^nR_{t'}(\bfp_j,\by)u_jV_j\right)^2\mathd \mu_\bx \mathd\mu_\by\nonumber\\
&=& \int_{\mathcal{M}}  \int_{\mathcal{M}}R_{t'}(\bx,\by) \left(\frac{1}{w_{t',h}(\bx)w_{t',h}(\by)}\sum_{i,j=1}^nR_{t'}(\bx,\bfp_i)R_{t'}(\bfp_j,\by)V_i
V_j(u_i-u_j)\right)^2\mathd \mu_\bx \mathd\mu_\by\nonumber\\
&\le & \int_{\mathcal{M}}  \int_{\mathcal{M}}R_{t'}(\bx,\by) \frac{1}{w_{t',h}(\bx)w_{t',h}(\by)}\sum_{i,j=1}^nR_{t'}(\bx,\bfp_i)R_{t'}(\bfp_j,\by)V_i
V_j(u_i-u_j)^2\mathd \mu_\bx \mathd\mu_\by\nonumber\\
&=&  \sum_{i,j=1}^n\left(\int_{\mathcal{M}}  \int_{\mathcal{M}}\frac{1}{w_{t',h}(\bx)w_{t',h}(\by)}
R_{t'}(\bx,\bfp_i)R_{t'}(\bfp_j,\by)R_{t'}(\bx,\by)\mathd \mu_\bx \mathd\mu_\by\right)V_i
V_j(u_i-u_j)^2\nonumber\\
%&=&  \sum_{i,j=1}^n\left(\int_{\mathcal{M}_{\bfp_j}^{t'}}  \int_{\mathcal{M}_{\bfp_i}^{t'}}\frac{1}{w_{t',h}(\bx)w_{t',h}(\by)}
%R_{t'}(\bx,\bfp_i)R_{t'}(\bfp_j,\by)R_{t'}(\bx,\by)\mathd \mu_\bx \mathd\mu_\by\right)V_i
%V_j(u_i-u_j)^2
\label{eqn:A1}
\end{eqnarray*}
Denote
%$$A = \int_{\mathcal{M}_{\bfp_j}^{t'}}  \int_{\mathcal{M}_{\bfp_i}^{t'}}\frac{1}{w_{t',h}(\bx)w_{t',h}(\by)}
%R_{t'}(\bx,\bfp_i)R_{t'}(\bfp_j,\by)R_{t'}(\bx,\by)\mathd \mu_\bx \mathd\mu_\by$$
$$A = \int_{\mathcal{M}}  \int_{\mathcal{M}}\frac{1}{w_{t',h}(\bx)w_{t',h}(\by)}
R_{t'}(\bx,\bfp_i)R_{t'}(\bfp_j,\by)R_{t'}(\bx,\by)\mathd \mu_\bx \mathd\mu_\by$$
and then notice only when $|\bfp_i-\bfp_j|^2\le 36t'$ is $A \neq 0$. For 
$|\bfp_i-\bfp_j|^2\le 36t'$, we have
\begin{eqnarray*}
A &\le&\int_{\mathcal{M}}   \int_{\mathcal{M}}
R_{t'}(\bx,\bfp_i)R_{t'}(\bfp_j,\by)R_{t'}(\bx,\by)R\left(\frac{|\bfp_i-\bfp_j|^2}{72t'}\right)^{-1}R\left(\frac{|\bfp_i-\bfp_j|^2}{72t'}\right)\mathd \mu_\bx \mathd\mu_\by\nonumber\\
&\le&\frac{CC_t}{\delta_0}  \int_{\mathcal{M}}   \int_{\mathcal{M}}
R_{t'}(\bx,\bfp_i)R_{t'}(\bfp_j,\by)R\left(\frac{|\bfp_i-\bfp_j|^2}{72t'}\right)\mathd \mu_\bx \mathd\mu_\by\nonumber\\
&\le&CC_t  \int_{\mathcal{M}}   \int_{\mathcal{M} }
R_{t'}(\bx,\bfp_i)R_{t'}(\bfp_j,\by)R\left(\frac{|\bfp_i-\bfp_j|^2}{72t'}\right)\mathd \mu_\bx \mathd\mu_\by\nonumber\\
&\le & CC_t R\left(\frac{|\bfp_i-\bfp_j|^2}{4t}\right).\nonumber\\
\label{eqn:A2}
\end{eqnarray*}
%\begin{eqnarray}
%A &\le&\int_{\mathcal{M}_{\bfp_j}^{t'}}   \int_{\mathcal{M}_{\bfp_i}^{t'}\cap \mathcal{M}_{\by}^{t'} }
%R_{t'}(\bx,\bfp_i)R_{t'}(\bfp_j,\by)R_{t'}(\bx,\by)R\left(\frac{|\bfp_i-\bfp_j|^2}{72t'}\right)^{-1}R\left(\frac{|\bfp_i-\bfp_j|^2}{72t'}\right)\mathd \mu_\bx \mathd\mu_\by\nonumber\\
%&\le&\frac{CC_t}{\delta_0}  \int_{\mathcal{M}_{\bfp_j}^{t'}}   \int_{\mathcal{M}_{\bfp_i}^{t'}\cap \mathcal{M}_{\by}^{t'} }
%R_{t'}(\bx,\bfp_i)R_{t'}(\bfp_j,\by)R\left(\frac{|\bfp_i-\bfp_j|^2}{72t'}\right)\mathd \mu_\bx \mathd\mu_\by\nonumber\\
%&\le&CC_t  \int_{\mathcal{M}}   \int_{\mathcal{M} }
%R_{t'}(\bx,\bfp_i)R_{t'}(\bfp_j,\by)R\left(\frac{|\bfp_i-\bfp_j|^2}{72t'}\right)\mathd \mu_\bx \mathd\mu_\by\nonumber\\
%&\le & CC_t R\left(\frac{|\bfp_i-\bfp_j|^2}{4t}\right)
%\end{eqnarray}

%\begin{eqnarray}
%&&\int_{\mathcal{M}}  \int_{\mathcal{M}}R_{t'}(\bx,\by) \left(u(\bx)-u(\by)\right)^2\mathd \mu_\bx \mathd\mu_\by\nonumber\\
%&\le& C \sum_{i,j=1}^nR\left(\frac{|\bfp_i-\bfp_j|^2}{4t}\right)(u_i-u_j)^2A_iA_j
%\end{eqnarray}
Combining the above two inequalities and using Lemma~\ref{thm:elliptic_L_t}, we obtain
\begin{eqnarray}
C\frac{C_t}{t}  \sum_{i,j=1}^nR\left(\frac{|\bfp_i-\bfp_j|^2}{4t}\right)(u_i-u_j)^2V_iV_j\ge \int_{\mathcal{M}}  (u_t(\bx)-\bar{u}_t)^2\mathd\mu_\bx
\label{eqn:u_2V_0}
\end{eqnarray}

We now lower bound the RHS of the above equation.% using $<\bfu, \bfu>_{\bf V}$. 
\begin{eqnarray*}
 {|\mathcal{M}|}|\bar{u}_t|&=&\left|\int_{\mathcal{M}}u_t(\bx)\mathd \mu_\bx\right|
			  =\left| \sum_{j=1}^n \left( u_jV_j \int_{\mathcal{M}}  \frac{C_t'}{w_{t',h}(\bx)}R\left(\frac{|\bx-\bfp_j|^2}{4t'}\right) \mathd \mu_\bx \right) \right|.
\end{eqnarray*}
Let  $q(\bx) = \frac{C_t'}{w_{t',h}(\bx)}{R\left(\frac{|\bx-\bfp_j|^2}{4t'}\right)}$. There exists a constant $C$ so that $|q(\bx)|\leq CC_t' $ and
\begin{eqnarray*}
|\nabla q(\bx)| \leq \frac{C_t'}{w_{t',h}(\bx)} \left|\nabla R\left(\frac{|\bx-\bfp_j|^2}{4t'}\right)\right| + \frac{C_t' \left|\nabla w_{t',h}(\bx)\right|}{w^2_{t,h}(\bx)} R\left(\frac{|\bx-\bfp_j|^2}{4t'}\right) \leq \frac{CC_t'}{t^{1/2}}. 
\end{eqnarray*}
Then, using the definition of the integral accuracy index, %for sufficiently small $t, \frac{h}{t^{1/2}}$, 
there exists a constant $C$ 
\begin{eqnarray*}
\quad\left| \int_{\mathcal{M}}  \frac{C_t'}{w_{t',h}(\bx)}R\left(\frac{|\bx-\bfp_j|^2}{4t'}\right) \mathd \mu_\bx  - \sum_{i=1}^n \frac{C_t'}{w_{t',h}(\bfp_i)}R\left(\frac{|\bfp_i-\bfp_j|^2}{4t'}\right)V_i\right| \leq \frac{Ch}{t^{1/2}}. 
\end{eqnarray*}
Thus we have 
\begin{eqnarray}			  
\label{eqn:u_2V_1}
&&{|\mathcal{M}|}|\bar{u}_t| \\
&\le& \left| \sum_{i,j=1}^n \frac{C_t'}{w_{t',h}(\bfp_i)}R\left(\frac{|\bfp_i-\bfp_j|^2}{4t'}\right)u_jV_iV_j \right| + \frac{Ch}{t^{1/2}} (\sum_{j=1}^n |u_jV_j|) \nonumber\\
&\le& \left| \sum_{i=1}^n u_t(\bfp_i)V_i\right| +\frac{Ch}{t^{1/2}} (\sum_{i=1}^n |u_iV_i|) \nonumber \\
&=&\frac{1}{|\mathcal{M}|}\left|\sum_{i,j=1}^n\frac{C_t'}{w_{t',h}(\bfp_i)}R\left(\frac{|\bfp_i-\bfp_j|^2}{4t'}\right)(u_j-u_i)V_iV_j\right|+ \frac{Ch}{t^{1/2}} (\sum_{i=1}^n u_i^2V_i)^{1/2}\nonumber\\
&\le & \frac{CC_t^{1/2}}{\sqrt{|\mathcal{M}|}}\left(\sum_{i,j=1}^nR\left(\frac{|\bx-\bfp_j|^2}{4t'}\right)(u_i-u_j)^2V_iV_j\right)^{1/2}+\frac{Ch}{t^{1/2}} (\sum_{i=1}^n u_i^2V_i)^{1/2}\nonumber
\end{eqnarray}
where the first equality is due to $\sum_{i=1}^n u_iV_i = 0$. Denote
\begin{eqnarray*}
A = \int_{\mathcal{M}}\frac{C_t}{w_{t',h}^2(\bx)}R\left(\frac{|\bx-\bfp_i|^2}{4t'}\right)
R\left(\frac{|\bx-\bfp_l|^2}{4t'}\right)\mathd\mu_\bx-\nonumber\\
\sum_{j=1}^n\frac{C_t}{w_{t',h}^2(\bfp_j)}R\left(\frac{|\bfp_j-\bfp_i|^2}{4t'}\right)R\left(\frac{|\bfp_j-\bfp_l|^2}{4t'}\right)V_j
\end{eqnarray*}
and then $|A|\le \frac{Ch}{t^{1/2}}$. At the same time, notice that only when $|\bfp_i-\bfp_l|^2 <16t'$ is $A\neq 0$. Thus we have 
\begin{eqnarray*}
|A| \le \frac{1}{\delta_0} |A| R(\frac{|\bfp_i-\bfp_l|^2}{32t'}),
\end{eqnarray*}
and
\begin{eqnarray}
\label{eqn:u_2V_2}
&&\left|\int_{\mathcal{M}}u_t^2(\bx)\mathd\mu_\bx-\sum_{j=1}^nu_t^2(\bfp_j)V_j\right| \\
&\le&
  \sum_{i,l=1}^n|C_{t}u_iu_lV_iV_l| |A| \nonumber\\
&\le &\frac{Ch}{t^{1/2}} \sum_{i, l=1}^n\left|C_{t} R\left(\frac{|\bfp_i-\bfp_l|^2}{32t'}\right) u_iu_lV_iV_l \right|\nonumber\\
&\le &\frac{Ch}{t^{1/2}} \sum_{i, l=1}^n C_{t} R\left(\frac{|\bfp_i-\bfp_l|^2}{32t'}\right) u_i^2V_iV_l \le
\frac{Ch}{t^{1/2}} \sum_{i=1}^n u_i^2V_i.\nonumber
\end{eqnarray}
Now combining Equation~\eqref{eqn:u_2V_0},~\eqref{eqn:u_2V_1} and \eqref{eqn:u_2V_2}, we have for small $t$ 
\begin{eqnarray*}
  \sum_{i=1}^nu^2_t(\bfp_i)V_i &=& \int_{\mathcal{M}}  u_t^2(\bx)\mathd\mu_\bx+\frac{Ch}{t^{1/2}}\sum_{i=1}^n u_i^2V_i \nonumber \\
&\le& 2\int_{\mathcal{M}}  (u_t(\bx)-\bar{u}_t)^2\mathd\mu_\bx+2\bar{u}_t^2|\mathcal{M}|+ \frac{Ch}{t^{1/2}}\sum_{i=1}^n u_i^2V_i\nonumber\\
&\le & \frac{CC_t}{t}  \sum_{i,j=1}^nR\left(\frac{|\bfp_i-\bfp_j|^2}{4t}\right)(u_i-u_j)^2V_iV_j+ \frac{Ch}{t}\sum_{i=1}^n u_i^2V_i.
\end{eqnarray*}
Here we use the fact that $t=18t'$ hence $$R\left(\frac{|\bfp_i-\bfp_j|^2}{4t'}\right) \leq \frac{1}{\delta_0} R\left(\frac{|\bfp_i-\bfp_j|^2}{4t}\right).$$

%\begin{eqnarray}
 %&& \left|\int_{\mathcal{M}}u^2(\bx)\mathd\mu_\bx-\sum_{j=1}^nu^2(\bfp_j)V_j\right|\nonumber\\
%&= & \left|\sum_{i,l=1}^nC_{t'}^2u_iu_lV_iV_l\left(\int_{\mathcal{M}}\frac{1}{w_{t',h}^2(\bx)}R\left(\frac{|\bx-\bfp_i|^2}{4t'}\right)
%R\left(\frac{|\bx-\bfp_l|^2}{4t'}\right)\mathd\mu_\bx-\right.\right.\nonumber\\
%&&\left.\left.\sum_{j=1}^n\frac{1}{w_{t',h}^2(\bfp_j)}R\left(\frac{|\bfp_j-\bfp_i|^2}{4t'}\right)
%R\left(\frac{|\bfp_j-\bfp_l|^2}{4t'}\right)\right)\right|\nonumber\\
%&\le &\frac{Ch}{t^{1/2}}\left|\sum_{i,l=1}^nC_{t'}^2u_iu_lV_iV_l\chi_{\Omega_i^t}(l)\right|\nonumber\\
%&\le & \frac{Ch}{t^{1/2}}\left|\sum_{i=1}^nC_{t'}u_iV_i\left(\sum_{l} V_l\chi_{\Omega_i^t}(l)\right)^{1/2}\left(\sum_{l}u_l^2V_l\chi_{\Omega_i^t}(l)\right)^{1/2}\right|\nonumber\\
%&\le & \frac{Ch}{t^{1/2}}\left|\sqrt{C_{t'}}\sum_{i=1}^nu_iV_i\left(\sum_{l}u_l^2V_l\chi_{\Omega_i^t}(l)\right)^{1/2}\right|\nonumber\\
%&\le & \frac{Ch}{t^{1/2}}\left|\sqrt{C_{t'}}\left(\sum_{i=1}^nu_i^2V_i\right)^{1/2}\left(\sum_{i=1}^nV_i\sum_{l}u_l^2V_l\chi_{\Omega_i^t}(l)\right)^{1/2}\right|\nonumber\\
%&\le & \frac{Ch}{t^{1/2}}\left|\sqrt{C_{t'}}\left(\sum_{i=1}^nu_i^2V_i\right)^{1/2}\left(\sum_lu_l^2V_l\sum_{i=1}^nV_i\chi_{\Omega_l^t}(i)\right)^{1/2}\right|\nonumber\\
%&\le & \frac{Ch}{t^{1/2}}\left(\sum_{i=1}^nu_i^2V_i\right)
%\end{eqnarray}

Let $\delta=\frac{w_{\min}}{2w_{\max}+w_{\min}}$ with $w_{\min}=\min_\bx w_{t,h}(\bx)$
and $w_{\max}=\max_\bx w_{t,h}(\bx)$. If $$\sum_{i=1}^nu^2(\bfp_i)V_i\ge \delta^2 \sum_{i=1}^nu_i^2V_i,$$ we have completed the proof. 
Otherwise, we have
\begin{eqnarray}
\sum_{i=1}^n( u_i-u_t(\bfp_i))^2V_i = \sum_{i=1}^n u_i^2V_i + \sum_{i=1}^n u_t(\bfp_i)^2V_i - 2 \sum_{i=1}^n u_iu_t(\bfp_i)V_i \ge (1- \delta)^2\sum_{i=1}^n u_i^2V_i. \nonumber
\end{eqnarray}
This enables us to prove the theorem as follows. 
\begin{eqnarray}
&&  C_t\sum_{i,j=1}^nR\left(\frac{|\bfp_i-\bfp_j|^2}{4t'}\right)(u_i-u_j)^2V_iV_j=   2C_t\sum_{i,j=1}^nR\left(\frac{|\bfp_i-\bfp_j|^2}{4t'}\right)u_i(u_i-u_j)V_iV_j\nonumber\\
%&=& 2\sum_{i=1}^nu_i(u_i-u(\bfp_i))w_{t,h}(\bfp_i)V_i\nonumber\\
&=& 2\sum_{i=1}^n(u_i-u_t(\bfp_i))^2w_{t,h}(\bfp_i)V_i+2\sum_{i=1}^nu_t(\bfp_i)(u_i-u(\bfp_i))w_{t,h}(\bfp_i)V_i\nonumber\\
&\ge& 2\sum_{i=1}^n(u_i-u_t(\bfp_i))^2w_{t,h}(\bfp_i)V_i-2\left(\sum_{i=1}^nu_t^2(\bfp_i)w_{t,h}(\bfp_i)V_i\right)^{1/2}\left(\sum_{i=1}^n(u_i-u_t(\bfp_i))^2w_{t,h}(\bfp_i)V_i\right)^{1/2}\nonumber\\
&\ge& 2w_{\min}\sum_{i=1}^n(u_i-u_t(\bfp_i))^2V_i-2w_{\max}\left(\sum_{i=1}^nu_t^2(\bfp_i)V_i\right)^{1/2}\left(\sum_{i=1}^n(u_i-u_t(\bfp_i))^2V_i\right)^{1/2}\nonumber\\
&\ge& 2(w_{\min}(1-\delta)^2-w_{\max}\delta(1-\delta))\sum_{i=1}^nu_i^2V_i\ge w_{\min}(1-\delta)^2\sum_{i=1}^nu_i^2V_i.\nonumber
\end{eqnarray}
%The theorem can be proved using the fact that
%\begin{equation*}
  %\sum_{i,j=1}^nR\left(\frac{|\bfp_i-\bfp_j|^2}{4t'}\right)(u_i-u_j)^2V_iV_j\le \frac{1}{\delta_0}\sum_{i,j=1}^nR\left(\frac{|\bfp_i-\bfp_j|^2}{4t}\right)(u_i-u_j)^2V_iV_j.
%\end{equation*}

%\begin{eqnarray}
%&&\left|\int_{\mathcal{M}}u(\bx)\mathd \mu_\bx-\sum_{j=1}^nu(\bfp_j)V_j\right|\nonumber\\
%&= & \left|\sum_{i=1}^nu_iV_i \left(\int_{\mathcal{M}}\frac{1}{w_{t,h}(\bx)}R\left(\frac{|\bfp_i-\bx|^2}{4t}\right)\mathd\mu_\bx-\sum_{j=1}^n
%\frac{1}{w_{t,h}(\bfp_j)}R\left(\frac{|\bfp_i-\bfp_j|^2}{4t}\right)V_j\right)\right|\nonumber\\
%&\le & C\frac{h}{t^{1/2}}\sum_{i=1}^n|u_i|V_i\le C\frac{h}{t^{1/2}}\left(\sum_{i=1}^nV_i\right)^{1/2}\left(\sum_{i=1}^nu_i^2V_i\right)^{1/2}=C\frac{h}{t^{1/2}}\left(\sum_{i=1}^nu_i^2V_i\right)^{1/2}
%\end{eqnarray}
\end{proof}

\section{Estimation of $\|\nabla L_t(u_{t}-u_{t, h})\|_{L^2(\M)}$}

In this section, we upper bound $\nabla L_t(u_{t}-u_{t, h})\|_{L^2(\M)}$. 
Remember that $u_t$ satisfies the integral equation \eqref{eq:integral-homo} and
\begin{eqnarray}
u_{t,h}(\bx) = \frac{1}{w_{t,h}(\bx)}\left(\sum_{\bfp_j\in P}R_t(\bx,\bfp_j)u_jV_j
-t\sum_{\bfp_j\in P}\bar{R}_t(\bx,\bfp_j)f_jV_j\right),\nonumber
\end{eqnarray}
where $\mathbf{u}=(u_1, \cdots, u_n)^t$ with $\sum_{i=1}^n u_i V_i = 0$ solves the problem~\eqref{eqn:dis-homo}, 
$ f_j = f(\bfp_j)$ and $w_{t,h}(\bx)=\sum_{\bfp_j\in P}R_t(\bx,\bfp_j)V_j$.  
%The techniques is similar those used in the previous $L^2$ estimate.

$\nabla L_t(u_{t}-u_{t, h})\|_{L^2(\M)}$ is splitted to two terms,
$$\nabla L_t(u_{t}-u_{t, h})=\nabla (L_tu_{t}-L_{t,h}u_{t, h})+\nabla (L_{t,h}u_{t,h}-L_tu_{t, h}).$$

The second term is easy to bound.
\begin{eqnarray}
\label{eqn:du2}
&&\|\nabla\left(L_t(u_{t}) - L_{t, h}(u_{t, h})\right)\|_{L^2(\M)}  \\ 
% &\leq& 2\left( \int_\M \left(\int_{\p\mathcal{M}}\nabla_\bx\bar{R}_t(\bx,\by)g(\by) \mathd \tau_\by- \sum_{\bfs_j\in S}\nabla_\bx\bar{R}_t(\bx,\bfs_j)g(\bfs_j)A_j\right)^2 \mathd \mu_\bx\right)^{1/2}\nonumber\\
&=& \left( \int _\M \left( \int_{\mathcal{M}}\nabla_\bx\bar{R}_t(\bx,\by)f(\by) - \sum_{\bfp_j\in P}\nabla_\bx\bar{R}_t(\bx,\bfp_j)f(\bfp_j)V_j\right)^2 \mathd\mu_\bx\right)^{1/2} \nonumber \\
&\le& \frac{Ch}{t}\|f\|_{C^1(\M)}. \nonumber
\end{eqnarray}
The first term is further splited by defining
\begin{eqnarray*}
  a_{t,h}(\bx)&=&\frac{1}{w_{t,h}(\bx)}\sum_{\bfp_j\in P}R_t(\bx,\bfp_j)u_jV_j, \\
%b_{t,h}(\bx)&=&\frac{2t}{w_{t,h}(\bx)}\sum_{\bfs_j\in S}\bar{R}_t(\bx,\bfs_j)g(\bfs_j)A_j,  \quad \text{and}\\
c_{t,h}(\bx)&=&-\frac{t}{w_{t,h}(\bx)}\sum_{\bfp_j\in P}\bar{R}_t(\bx,\bfp_j)f(\bfp_j)V_j, 
\end{eqnarray*}
To simplify the notation, we denote $h=h(P,\bV,\M)$ and $n=|P|$.
Consider $\|\nabla (L_ta_{t, h} - L_{t, h}a_{t, h})\|_{L_2}$.
\begin{eqnarray}
\label{eqn:da_1}
&&  \int_{\mathcal{M}}\left|\nabla a_{t,h}(\bx)\right|^2\left|\int_{\mathcal{M}}R_t(\bx,\by) \mathd \mu_\by-\sum_{\bfp_j\in P}R_t(\bx,\bfp_j)V_j\right|^2\mathd\mu_\bx\\
&\le& \frac{Ch^2}{t}\int_{\mathcal{M}}\left|\nabla a_{t,h}(\bx)\right|^2\mathd\mu_\bx \nonumber\\
&\le & \frac{Ch^2}{t} \left(\int_{\mathcal{M}} \left| \frac{1}{w_{t, h}(\bx)} \sum_{\bfp_j\in P}\nabla R_t(\bx,\bfp_j)u_jV_j \right|^2 \mathd\mu_\bx\right.\nonumber\\
&&\left.+\int_{\mathcal{M}} \left| \frac{\nabla w_{t,h}(\bx)}{w^2_{t, h}(\bx)} \sum_{\bfp_j\in P} R_t(\bx,\bfp_j)u_jV_j \right|^2 \mathd\mu_\bx\right)\nonumber\\
&\le & \frac{Ch^2}{t^2} \int_{\mathcal{M}} \left| \sum_{\bfp_j\in P} R_{2t}(\bx,\bfp_j)u_jV_j \right|^2 \mathd\mu_\bx\nonumber\\
%&\le & \frac{Ch^2}{t^2} \int_{\mathcal{M}} \left( \sum_{\bfp_j\in P}R_{2t}(\bx,\bfp_j)u_j^2V_j  \right) \left( \sum_{\bfp_j\in P}R_{2t}(\bx,\bfp_j)V_j  \right) \mathd\mu_\bx \nonumber \\
&\le & \frac{Ch^2}{t^2} \left( \sum_{j=1}^{n}u_j^2V_j \int_{\mathcal{M}}R_{2t}(\bx,\bfp_j)  \mathd\mu_\bx    \right)  
\le  \frac{Ch^2}{t^2}\sum_{j=1}^{n}u_j^2V_j.\nonumber
\end{eqnarray}
where $R_{2t}(\bx,\bfp_j)=C_tR\left(\frac{|\bx-\bfp_j|^2}{8t}\right)$. Here we use the assumption that $R(s)>\delta_0$ for all 
$0\le s\le 1/2$.
\begin{eqnarray}
\label{eqn:da_2}
&&  \int_{\mathcal{M}}\left| a_{t,h}(\bx)\right|^2\left|\int_{\mathcal{M}}\nabla R_t(\bx,\by) \mathd \mu_\by-\sum_{\bfp_j\in P}\nabla R_t(\bx,\bfp_j)V_j\right|^2\mathd\mu_\bx\\
&\le & \frac{Ch^2}{t^2}\int_{\mathcal{M}}\left| a_{t,h}(\bx)\right|^2\mathd\mu_\bx \le  \frac{Ch^2}{t^2}\sum_{j=1}^{n}u_j^2V_j.\nonumber
\end{eqnarray}
Let 
\begin{eqnarray*}
B &=&  C_t\int_{\mathcal{M}}\frac{1}{w_{t, h}(\by)}\nabla R\left(\frac{|\bx-\by|^2}{4t}\right)R\left(\frac{|\bfp_i-\by|^2}{4t}\right) \mathd \mu_\by\nonumber\\
 &-&C_t\sum_{\bfp_j\in P} \frac{1}{w_{t, h}(\bfp_j)}\nabla R\left(\frac{|\bx-\bfp_j|^2}{4t}\right)R\left(\frac{|\bfp_i-\bfp_j|^2}{4t}\right)V_j. 
\end{eqnarray*}
We have $|B|<\frac{Ch}{t^{1/2}}$ for some constant $C$ independent of $t$. In addition, notice that
only when $|\bx-\bx_i|^2\leq 16t $ is $B\neq 0$, which implies 
\begin{eqnarray*}
|B| \leq \frac{1}{\delta_0}|B|R\left(\frac{|\bx-\bfp_i|^2}{32t}\right). 
\end{eqnarray*}
Then we have
\begin{eqnarray}
\label{eqn:da_3}
&&\int_{\mathcal{M}}\left|\int_{\mathcal{M}}\nabla R_t(\bx,\by)a_{t,h}(\by)  \mathd \mu_\by-\sum_{\bfp_j\in P}\nabla R_t(\bx,\bfp_j)a_{t,h}(\bfp_j)V_j\right|^2\mathd\mu_\bx\\
&=& \int_{\mathcal{M}}\left(\sum_{i=1}^{n}C_tu_iV_i B \right)^2\mathd\mu_\bx\nonumber\\
&\le &\frac{Ch^2}{t^2} \int_{\mathcal{M}}\left(\sum_{i=1}^{n} C_t|u_i|V_i R\left(\frac{|\bx-\bfp_i|^2}{32t}\right)  \right)^2 \mathd\mu_\bx \nonumber\\
&\le & \frac{Ch^2}{t^2} \left(\sum_{i=1}^{n}u_i^2V_i\right). \nonumber
\end{eqnarray} 
Combining Equation~\eqref{eqn:da_1}, ~\eqref{eqn:da_2} and ~\eqref{eqn:da_3}, we have
\begin{eqnarray}
&&\|\nabla (L_ta_{t, h} - L_{t, h}a_{t, h})\|_{L^2(\M)} \nonumber \\
&=&\left(\int_M \left|\left(L_t(a_{t, h}) - L_{t, h}(a_{t, h})\right)(\bx)\right|^2 \mathd\mu_\bx\right)^{1/2} \nonumber\\
&\le & \frac{1}{t}\left(\int_{\mathcal{M}}\left(\nabla a_{t,h}(\bx)\right)^2\left|\int_{\mathcal{M}}R_t(\bx,\by) \mathd \mu_\by-\sum_{\bfp_j\in P}R_t(\bx,\bfp_j)V_j\right|^2\mathd\mu_\bx\right)^{1/2}\nonumber\\
&&\frac{1}{t}\left(\int_{\mathcal{M}}\left(a_{t,h}(\bx)\right)^2\left|\int_{\mathcal{M}}\nabla_\bx R_t(\bx,\by) \mathd \mu_\by-\sum_{\bfp_j\in P}\nabla_\bx R_t(\bx,\bfp_j)V_j\right|^2\mathd\mu_\bx\right)^{1/2}\nonumber\\
&&+  \frac{1}{t}\left(\int_{\mathcal{M}}\left|\int_{\mathcal{M}}\nabla_\bx R_t(\bx,\by)a_{t,h}(\by)  \mathd \mu_\by-\sum_{\bfp_j\in P}\nabla_\bx R_t(\bx,\bfp_j)a_{t,h}(\bfp_j)V_j\right|^2\mathd\mu_\bx
\right)^{1/2} \nonumber \\
&\le & \frac{Ch}{t^{2}}\left(\sum_{i=1}^{n}u_i^2V_i\right)^{1/2}\le \frac{Ch}{t^{2}}\|f\|_{\infty}\nonumber
\end{eqnarray}
Using a similar argument, we obtain
\begin{eqnarray*}
%\|\nabla (L_tb_{t, h} - L_{t, h}b_{t, h})\|_{L^2(\M)} &\le & \frac{Ch}{t^{2}}\|g\|_\infty,\\
\|\nabla (L_tc_{t, h} - L_{t, h}c_{t, h})\|_{L^2(\M)} &\le & \frac{Ch}{t^{3/2}}\|f\|_\infty,
\end{eqnarray*}
and thus
\begin{eqnarray}
\|\nabla (L_tu_{t, h} - L_{t, h}u_{t, h})\|_{L^2(\M)} \le \frac{Ch}{t^{2}}\|f\|_{\infty}.
\label{eqn:du1}
\end{eqnarray}
Then the estimation is completed by putting \eqref{eqn:du2} and \eqref{eqn:du1} together.
%%% Local Variables: 
%%% mode: latex
%%% TeX-master: t
%%% End: 

%\input{appendix_robin}

\bibliographystyle{abbrv}
\bibliography{reference}

\begin{thebibliography}{10}

\bibitem{book-nonlocal}
F.~Andreu, J.~M. Mazon, J.~D. Rossi, and J.~Toledo.
\newblock {\em Nonlocal Diffusion Problems}.
\newblock Math. Surveys Monogr. 165, AMS, Providence, RI, 2010.

\bibitem{WD08}
R.~Barreira, C.~Elliott, and A.~Madzvamuse.
\newblock Modelling and simulations of multi-component lipid membranes and open
  membranes via diffuse interface approaches.
\newblock {\em J. Math. Biol.}, 56:347--371, 2008.

\bibitem{BEM11}
R.~Barreira, C.~Elliott, and A.~Madzvamuse.
\newblock The surface finite element method for pattern formation on evolving
  biological surfaces.
\newblock {\em J. Math. Biol.}, 63:1095--1119, 2011.

\bibitem{belkin2003led}
M.~Belkin and P.~Niyogi.
\newblock Laplacian eigenmaps for dimensionality reduction and data
  representation.
\newblock {\em Neural Computation}, 15(6):1373--1396, 2003.

\bibitem{BelkinN05}
M.~Belkin and P.~Niyogi.
\newblock Towards a theoretical foundation for laplacian-based manifold
  methods.
\newblock In {\em COLT}, pages 486--500, 2005.

\bibitem{CLEM_08}
M.~Belkin and P.~Niyogi.
\newblock Convergence of laplacian eigenmaps.
\newblock {\em preprint, short version NIPS 2008}, 2008.

\bibitem{BelkinQWZ12}
M.~Belkin, Q.~Que, Y.~Wang, and X.~Zhou.
\newblock Toward understanding complex spaces: Graph laplacians on manifolds
  with singularities and boundaries.
\newblock In S.~Mannor, N.~Srebro, and R.~C. Williamson, editors, {\em COLT},
  volume~23 of {\em JMLR Proceedings}, pages 36.1--36.26. JMLR.org, 2012.

\bibitem{pcdlp2009}
M.~Belkin, J.~Sun, and Y.~Wang.
\newblock Constructing laplace operator from point clouds in rd.
\newblock In {\em SODA '09: Proceedings of the Nineteenth Annual ACM -SIAM
  Symposium on Discrete Algorithms}, pages 1031--1040, Philadelphia, PA, USA,
  2009. Society for Industrial and Applied Mathematics.

\bibitem{Bertalmio}
M.~Bertalmio, L.-T. Cheng, S.~Osher, and G.~Sapiro.
\newblock Variational problems and partial differential equations on implicit
  surfaces.
\newblock {\em Journal of Computational Physics}, 174(2):759 -- 780, 2001.

\bibitem{CFP97}
J.~W. Cahn, P.~Fife, and O.~Penrose.
\newblock A phase-field model for diffusion-induced grain-boundary motion.
\newblock {\em Ann. Statist.}, 36(2):555--586, 2008.

\bibitem{Chung}
F.~R.~K. Chung.
\newblock {\em Spectral Graph Theory}.
\newblock American Mathematical Society, 1997.

\bibitem{Coifman05geometricdiffusions}
R.~R. Coifman, S.~Lafon, A.~B. Lee, M.~Maggioni, F.~Warner, and S.~Zucker.
\newblock Geometric diffusions as a tool for harmonic analysis and structure
  definition of data: Diffusion maps.
\newblock In {\em Proceedings of the National Academy of Sciences}, pages
  7426--7431, 2005.

\bibitem{Dey11hom}
T.~K. Dey, J.~Sun, and Y.~Wang.
\newblock Approximating cycles in a shortest basis of the first homology group
  from point data.
\newblock {\em Inverse Problems}, 27(12):124004, 2011.

\bibitem{Du-SIAM}
Q.~Du, M.~Gunzburger, R.~B. Lehoucq, and K.~Zhou.
\newblock Analysis and approximation of nonlocal diffusion problems with volume
  constraints.
\newblock {\em SIAM Review}, 54:667--696, 2012.

\bibitem{DGLZ13}
Q.~Du, M.~Gunzburger, R.~B. Lehoucq, and K.~Zhou.
\newblock A nonlocal vector calculus, nonlocal volume-constrained problems, and
  nonlocal balance laws.
\newblock {\em Math. Models Methods Appl. Sci.}, 23:493--540, 2013.

\bibitem{DJTZ13}
Q.~Du, L.~Ju, L.~Tian, and K.~Zhou.
\newblock A posteriori error analysis of finite element method for linear
  nonlocal diffusion and peridynamic models.
\newblock {\em Math. Comp.}, 82:1889--1922, 2013.

\bibitem{DLZ13}
Q.~Du, T.~Li, and X.~Zhao.
\newblock A convergent adaptive finite element algorithm for nonlocal diffusion
  and peridynamic models.
\newblock {\em SIAM J. Numer. Anal.}, 51:1211--1234, 2013.

\bibitem{DE-Acta}
G.~Dziuk and C.~M. Elliott.
\newblock Finite element methods for surface pdes.
\newblock {\em Acta Numerica}, 22:289--396, 2013.

\bibitem{EE08}
C.~Eilks and C.~M. Elliott.
\newblock Numerical simulation of dealloying by surface dissolution via the
  evolving surface finite element method.
\newblock {\em J. Comput. Phys.}, 227:9727--9741, 2008.

\bibitem{ES10}
C.~M. Elliott and B.~Stinner.
\newblock Modeling and computation of two phase geometric biomem- branes using
  surface finite elements.
\newblock {\em J. Comput. Phys.}, 229:6585--6612, 2010.

\bibitem{GT09}
S.~Ganesan and L.~Tobiska.
\newblock A coupled arbitrary lagrangian eulerian and lagrangian method for
  computation of free-surface flows with insoluble surfactants.
\newblock {\em J. Comput. Phys.}, 228:2859--2873, 2009.

\bibitem{Hein:2005:GMW}
M.~Hein, J.-Y. Audibert, and U.~von Luxburg.
\newblock From graphs to manifolds - weak and strong pointwise consistency of
  graph laplacians.
\newblock In {\em Proceedings of the 18th Annual Conference on Learning
  Theory}, COLT'05, pages 470--485, Berlin, Heidelberg, 2005. Springer-Verlag.

\bibitem{JL04}
A.~J. James and J.~Lowengrub.
\newblock A surfactant-conserving volume-of-fluid method for interfacial flows
  with insoluble surfactant.
\newblock {\em J. Comput. Phys.}, 201:685--722, 2004.

\bibitem{Lafon04diffusion}
S.~Lafon.
\newblock {\em Diffusion Maps and Geodesic Harmonics}.
\newblock PhD thesis, 2004.

\bibitem{Lai13}
R.~Lai, J.~Liang, and H.~Zhao.
\newblock A local mesh method for solving pdes on point clouds.
\newblock {\em Inverse Problem and Imaging}, 7:737--755, 2013.

\bibitem{LZ11}
S.~Leung, J.~Lowengrub, and H.~Zhao.
\newblock A grid based particle method for solving partial differential
  equations on evolving surfaces and modeling high order geometrical motion.
\newblock {\em J. Comput. Phys.}, 230(7):2540--2561, 2011.

\bibitem{LZ09}
S.~Leung and H.~Zhao.
\newblock A grid based particle method for moving interface problems.
\newblock {\em J. Comput. Phys.}, 228(8):2993--3024, 2009.

\bibitem{LSS}
Z.~Li, Z.~Shi, and J.~Sun.
\newblock Point integral method for solving poisson-type equations on manifolds
  from point clouds with convergence guarantees.
\newblock {\em arXiv:1409.2623}.

\bibitem{Liang13}
J.~Liang and H.~Zhao.
\newblock Solving partial differential equations on point clouds.
\newblock {\em SIAM Journal of Scientific Computing}, 35:1461--1486, 2013.

\bibitem{LuoSW09}
C.~Luo, J.~Sun, and Y.~Wang.
\newblock Integral estimation from point cloud in d-dimensional space: a
  geometric view.
\newblock In {\em Symposium on Computational Geometry}, pages 116--124, 2009.

\bibitem{MR09}
C.~Macdonald and S.~Ruuth.
\newblock The implicit closest point method for the numerical so- lution of
  partial differential equations on surfaces.
\newblock {\em SIAM J. Sci. Comput.}, 31(6):4330--4350, 2009.

\bibitem{NMWI11}
M.~P. Neilson, J.~A. Mackenzie, S.~D. Webb, and R.~H. Insall.
\newblock Modelling cell movement and chemotaxis using pseudopod-based
  feedback.
\newblock {\em SIAM J. Sci. Comput.}, 33:1035--1057, 2011.

\bibitem{NiyogiSW08}
P.~Niyogi, S.~Smale, and S.~Weinberger.
\newblock Finding the homology of submanifolds with high confidence from random
  samples.
\newblock {\em Discrete \& Computational Geometry}, 39(1-3):419--441, 2008.

\bibitem{LDMM}
S.~Osher, Z.~Shi, and W.~Zhu.
\newblock Low dimensional manifold model for image processing.
\newblock {\em Technical report, UCLA, CAM-report 16-04}.

\bibitem{Peyre09}
G.~Peyr\'e.
\newblock Manifold models for signals and images.
\newblock {\em Computer Vision and Image Understanding}, 113:248--260, 2009.

\bibitem{RM08}
S.~Ruuth and B.~Merriman.
\newblock A simple embedding method for solving partial differ- ential
  equations on surfaces.
\newblock {\em J. Comput. Phys.}, 227(3):1943--1961, 2008.

\bibitem{Singer06}
A.~Singer.
\newblock {From graph to manifold Laplacian: The convergence rate}.
\newblock {\em Applied and Computational Harmonic Analysis}, 21(1):128--134,
  July 2006.

\bibitem{Singer13}
A.~Singer and H.~tieng Wu.
\newblock Spectral convergence of the connection laplacian from random samples.
\newblock {\em arXiv:1306.1587}.

\bibitem{Wardetzky06}
M.~Wardetzky.
\newblock {\em Discrete Differential Operators on Polyhedral Surfaces -
  Convergence and Approximation}.
\newblock PhD thesis, 2006.

\bibitem{XZ03}
J.~Xu and H.~Zhao.
\newblock An eulerian formulation for solving partial differential equations
  along a moving interface.
\newblock {\em J. Sci. Comput.}, 19:573--594, 2003.

\bibitem{ZD10}
K.~Zhou and Q.~Du.
\newblock Mathematical and numerical analysis of linear peridynamic models with
  nonlocal boundary conditions.
\newblock {\em SIAM J. Numer. Anal.}, 48:1759--1780, 2010.

\end{thebibliography}

\end{document}